\documentclass[reqno,tbtags,intlimits,oneside,letterpaper, 12pt]{amsart}

\usepackage[foot]{amsaddr}

\usepackage[english]{babel}
\usepackage{amssymb,upref,calc,url}
\usepackage[mathscr]{eucal}

\usepackage{amsmath, amsfonts,amscd,amsthm,amssymb,verbatim}
\usepackage[all]{xy}
\usepackage{textcomp}
\usepackage[in]{fullpage}

\usepackage[title]{appendix}

\theoremstyle{definition}
\newtheorem{theorem}{Theorem}[section]
\newtheorem{proposition}{Proposition}[section]
\newtheorem{corollary}[theorem]{Corollary}
\newtheorem{lemma}[theorem]{Lemma}
\newtheorem{definition}[theorem]{Definition}
\newtheorem{remark}[theorem]{Remark}

\newenvironment{proofsketch}
  {\begin{proof}[Proof sketch]}
  {\end{proof}}

\usepackage[
backend=biber,
style=alphabetic,
sorting=nyt,
maxbibnames=99,
maxcitenames=99,
maxalphanames=99
]{biblatex}

\usepackage{subcaption}
\usepackage{graphicx}
\graphicspath{{images/}}

\newcommand{\bbC}{\mathbb{C}}
\newcommand{\bbR}{\mathbb{R}}

\newcommand{\bbZ}{\mathbb{Z}}
\newcommand{\bbN}{\mathbb{N}}
\newcommand{\bbF}{\mathbb{F}}

\newcommand{\cH}{\mathcal{H}}
\newcommand{\cM}{\mathcal{M}}
\newcommand{\cA}{\mathcal{A}}

\newcommand{\cT}{\mathcal{T}}

\newcommand{\rmH}{\mathrm{H}}
\newcommand{\rmF}{\mathrm{F}}
\newcommand{\rmE}{\mathrm{E}}
\newcommand{\rmV}{\mathrm{V}}
\newcommand{\rmW}{\mathrm{W}}

\newcommand{\rmB}{\mathrm{B}}

\newcommand{\del}{\partial}
\newcommand{\delbar}{\overline{\partial}}

\newcommand{\MCG}{\mathrm{MCG}}

\newcommand{\Diff}{\mathrm{Diff}}

\newcommand{\Id}{\mathrm{Id}} 

\newcommand{\SLgr}{\mathrm{SL}_2\mathbb{R}} 

\newcommand{\GLgr}{\mathrm{GL}_2\mathbb{R}} 

\newcommand{\PGLgr}{\mathrm{PGL}_2\mathbb{R}} 

\newcommand{\End}{\mathrm{End}} 

\newcommand{\Teichmuller}{Teichm\"uller}

\usepackage{xcolor}

\definecolor{shamrockgreen}{rgb}{0.0, 0.62, 0.38}

\usepackage{tikz-cd}

\usepackage{hyperref}
\hypersetup{
    colorlinks=true,
    linkcolor=shamrockgreen,
    filecolor=magenta,
    citecolor=blue,
    urlcolor=blue,
    pdftitle={Semisimplicity and Rigidity of the Kontsevich--Zorich Cocycle for Products of Strata}
    }

\usepackage{setspace}[1.5]
\title[Semisimplicity and Rigidity of the Kontsevich--Zorich Cocycle for Products of Strata]{Semisimplicity and Rigidity\\ of the Kontsevich--Zorich Cocycle\\ for Products of Strata}
\author{Polina Baron}
\address{University of Toronto, Canada}
\email{p.baron@utoronto.ca}

\begin{document}

\begin{abstract}
We study the Kontsevich--Zorich cocycle, its fixed tensor constructions, and finite direct sums of such constructions over products of strata of Abelian differentials. The acting groups are closed subgroups $G\subset(\mathrm{SL}_2\mathbb R)^n$ which, after permuting factors and coordinatewise conjugation by rotations, are products of diagonal $\mathrm{SL}_2\mathbb R$ blocks times a finite central group. For ergodic invariant probability measures we establish Deligne-type semisimplicity: the corresponding bundles admit Hodge-compatible decompositions into $G$-invariant pieces $V_i\otimes_{A_i}W_i$, where the $A_i$ are division algebras, the $V_i$ carry pure variations of Hodge structure, and the multiplicity spaces $W_i$ carry finite direct sums of pure Hodge structures. Every $G$-invariant subbundle arises by choosing $A_i$-submodules of the $W_i$.

Under the additional hypothesis that the block marginal measures have smooth local supports and are locally equivalent to smooth volume, we prove that Hodge-orthogonal projections onto invariant subbundles are locally polynomial in period coordinates on the normalized support. The proof derives a local product-of-quadrics model and establishes joint polynomiality from polynomiality along stable and unstable leaves. On the cone obtained by scaling each diagonal block together, the projection coefficients have homogeneous degree-zero expressions with integral powers of the reciprocal block areas. For the full product group $(\SLgr)^n$, the single-stratum measure classification supplies the support hypothesis, so the rigidity conclusion holds for every ergodic invariant probability measure. Under the transported support hypothesis, polynomial equivariant projectors also exist for arbitrary coordinatewise $\SLgr$-conjugates of these groups, although they need not be Hodge-orthogonal. 

Genus-two examples show that nonrotational conjugation can destroy Deligne-torus invariance and that independent factor scalings can require square roots of area functions.
\end{abstract}

\maketitle

\tableofcontents

\section*{Introduction}

A foundational breakthrough of Eskin--Mirzakhani--Mohammadi shows that orbit closures for the $\SLgr$-action on a moduli space of translation surfaces are
\emph{affine invariant submanifolds}: immersed suborbifolds locally cut out by real linear equations in period coordinates
(and more generally, closed invariant sets are finite unions of such manifolds) \cite{EMM15}.
This structural result has shifted many questions in \Teichmuller~dynamics to the study of affine invariant manifolds and the
linear-algebraic constraints they impose on the geometry of the surface.
Further tools for analyzing affine invariant manifolds include Wright's Cylinder Deformation Theorem and its consequences for
rank, field of definition, and cylinder configurations \cite{Wright2014}.

The principal linear object over strata is the (cohomological) Hodge bundle, equipped with the Gauss--Manin connection.
Parallel transport along $\SLgr$-orbits defines the Kontsevich--Zorich (KZ) cocycle.
Its Lyapunov exponents govern deviation of ergodic averages for translation flows and related problems (beginning with works
of Zorich and Forni, and in particular Forni's non-uniform hyperbolicity and spectral gap phenomena \cite{Forni2002Deviation}).
They also satisfy strong global constraints: for any regular $\SLgr$-invariant suborbifold, the sum of the nonnegative
Lyapunov exponents admits an explicit expression in terms of the Siegel--Veech constant \cite{EskinKontsevichZorich2014}.

A key Hodge-theoretic input to the modern study of the KZ cocycle is Filip's work on semisimplicity and rigidity
\cite{Filip2013}.
Among other consequences, the Deligne torus preserves the family of invariant subbundles of tensor powers of the Hodge bundle. These cocycles admit Deligne-type semisimplicity statements, and invariant subbundles exhibit strong regularity (real-analyticity along stable/unstable leaves and polynomial dependence in period coordinates) \cite{Filip2013}.
In a complementary direction, Filip proved that affine invariant manifolds are algebraic varieties via a splitting statement for
certain mixed Hodge structures \cite{FilipSplitting2016}.
These results have since been sharpened and extended in several directions, including a precise computation of the algebraic
hull of the KZ cocycle over any $\mathrm{GL}_2^+\bbR$-invariant subvariety and finiteness consequences
\cite{EskinFilipWright2018} and boundary (compactification) techniques, starting with Mirzakhani--Wright \cite{MirzakhaniWright2017Boundary} and continuing through
more recent compactification frameworks surveyed by Dozier \cite{Dozier2024}.

\vspace{.5cm}

\noindent\textbf{Products of strata and diagonal-type actions.}
This paper studies Hodge-theoretic constraints for group actions on \emph{products} of strata
$$
X=\cH^{(1)}(\alpha_1)\times\cdots\times \cH^{(1)}(\alpha_n),
$$
equipped with an ergodic invariant measure for a closed subgroup $G\subset (\SLgr)^n$ that projects surjectively to each factor.
Their identity components are products of diagonal blocks after coordinatewise conjugation, and the remaining components are described by one finite central subgroup, which may couple different blocks (Section~\ref{sect1.2}).
For the Hodge-theoretic results of this paper we impose the additional standing hypothesis that the coordinatewise conjugations can be chosen to be rotations. Equivalently, after permuting factors and conjugating individual coordinates by elements of $SO_2(\bbR)$, we work with groups of the form
$$
G=(D_1\cdots D_r)F,
$$
where every $D_\beta\cong\SLgr$ acts diagonally on one block of factors and trivially on the others, and $F\leq\{\pm\Id\}^n$.
Thus, even though the ambient action is a product action, the groups we consider may impose nontrivial \emph{couplings} between factors, and invariant bundles for the induced cocycles encode genuinely multi-factor phenomena.
The restriction on the conjugating maps is essential: in Section~\ref{sect5} we give an explicit genus-two example showing that Deligne-torus invariance can fail for an arbitrarily conjugated diagonal subgroup.

\vspace{.5cm}

\noindent\textbf{Main results.} Our work extends Filip's semisimplicity and rigidity mechanisms from single strata to this product setting.
Let $\rmE_\bbF$ be the Hodge bundle over $X$, a fixed pure tensor construction, or a finite direct sum of such constructions, with $\bbF=\bbR$ or $\bbC$, and let $\Gamma$ be the cocycle over $(X,\mu)$ induced by the Gauss--Manin connection and the $G$-action. Invariant subbundles are understood to be measurable and defined modulo $\mu$-null sets. Our first result is 

\vspace{.25cm}
\noindent\textbf{Theorem A (Deligne semisimplicity).}
There are pairwise nonisomorphic irreducible $G$-bundles $V_i$ and constant finite-dimensional spaces $W_i$ with a $G$-equivariant, Hodge-compatible decomposition
\begin{equation}\label{E-decomp-eq1}
 \rmE_\bbF\cong\bigoplus_i V_i\otimes_{A_i}W_i,
 \qquad A_i=\operatorname{End}_G(V_i)^{\mathrm{op}}.
\end{equation}
Each $A_i$ is a division algebra, each $V_i$ carries a pure variation of Hodge structure along the orbit polydiscs, and each $W_i$ carries a finite direct sum of pure Hodge structures. The right $A_i$-action on $V_i$ and the left $A_i$-action on $W_i$ have Hodge type $(0,0)$. If $\rmE_\bbF$ is pure, the $W_i$ are pure as well. Every invariant subbundle has the form
\[
 U=\bigoplus_i V_i\otimes_{A_i}W_i',
\]
where $W_i'\subset W_i$ are $A_i$-submodules, which need not respect the Hodge or weight decompositions.
\vspace{.25cm}

Theorem A combines Theorem~\ref{thm__Deligne-semisimplicity} and Corollary~\ref{cor:graded-semisimplicity}. The Hodge structures on the underlying irreducible bundles are part of the conclusion; a chosen inclusion $V_i\subset\rmE_\bbF$ need not preserve the ambient Hodge decomposition. Allowing several weights in $W_i$ accounts for invariant graphs between isomorphic cocycles occurring in different ambient weights.

A crucial intermediate step is a \emph{fixed-part} statement (Theorem~\ref{fixed_part_thm}) showing that if a measurable section of a VHS is flat
along almost every $G$-orbit, then each of its Hodge components is also flat. This implies that the Deligne torus acts on the poset of invariant
subbundles: if $V\subset \rmE_\bbC$ is $G$-invariant and $a\in\mathbb S$ is a Deligne torus element, then $a\cdot V$ is again $G$-invariant
(Theorem~\ref{deligne-torus-thm}). Together with reductivity of the relevant algebraic hulls (Section~\ref{sect2}), this yields the semisimplicity
classification. The restriction to rotationally conjugated diagonal blocks enters precisely here: an arbitrary conjugation need not preserve the canonical Hodge structures, and the corresponding Deligne-torus invariance can fail. 

\vspace{.25cm}
For the second theorem, we impose an additional hypothesis on the measure. The measure $\mu$ has finitely many $D$-ergodic components, and each component is a product of ergodic block marginals (Proposition~\ref{prop:affine-measure-reduction}). For every such marginal we assume:
\begin{quote}
\textup{(S)} The support is locally a smooth immersed submanifold, and the measure is locally equivalent to smooth volume on that support. We work on smooth branches off the null self-intersection locus.
\end{quote}
The measure-rigidity theorem for diagonal actions in \cite{BrownEskinFilipRodriguezHertz2025} motivates this hypothesis, which we retain as an explicit assumption.

\vspace{.25cm}
\noindent\textbf{Theorem B (Polynomial rigidity).}
Assume \textup{(S)}. Every measurable $G$-invariant subbundle $V\subset\rmE_\bbF$ has a unique continuous, real-analytic representative on each smooth branch of the support of each $D$-ergodic component. In local flat trivializations, the coefficients of its Hodge-orthogonal projection $\Pi_V$ are restrictions of ambient polynomials in the real period coordinates.

Extend $\Pi_V$ by area normalization to the cone obtained by independently scaling the original blocks $I_\beta$, using one common positive scalar within each block. If $A_\beta$ is the common area of the components in $I_\beta$, each coefficient is locally of the form
\[
 \frac{P(x,y)}{\prod_{\beta=1}^r A_\beta^{d_\beta}},
 \qquad d_\beta\in\mathbb Z_{\geq0},
\]
where $P$ is bihomogeneous of total degree $(d_\beta,d_\beta)$ in the real and imaginary period variables belonging to $I_\beta$. Thus the extension has degree zero under scalar dilation of each block.
\vspace{.25cm}

This result combines Theorems~\ref{thm__joint-polynomiality-fixed-area} and~\ref{thm:original-block-cone} with Corollary~\ref{cor:regular-representatives}. When $G=(\SLgr)^n$, each block is a singleton and \textup{(S)} follows from the single-stratum measure classification \cite{EM18}. The factorwise area-denominator formula therefore holds for every ergodic invariant probability measure for the full product action (Corollary~\ref{thm__joint-polynomiality}).

The geometric input to Theorem B is derived from \textup{(S)}. On each smooth branch, the support is locally a product of quadrics
\[
 \prod_a\{(x_a,y_a)\in L_a\times L_a:\Omega_a(x_a,y_a)=1\},
\]
where the $L_a$ are real linear subspaces of the relative period spaces and the $\Omega_a$ are nonzero alternating forms, possibly degenerate. The corresponding partition of the factors may refine the original group blocks $I_\beta$. Applying Filip's leafwise mechanism to one simultaneous geodesic flow gives polynomiality along stable and unstable affine fibers. A separate-polynomiality theorem on these quadrics then gives joint polynomiality. This argument uses the local geometry derived from \textup{(S)} and does not assume an independently scaled affine cone.

\vspace{.5cm}

\noindent\textbf{Scope and sharpness.}
Polynomial regularity also holds for a broader class of groups. If $G=cG_0c^{-1}$ with $c\in(\SLgr)^n$ and $G_0$ satisfying the standing hypothesis, and if $(c^{-1})_*\mu$ satisfies \textup{(S)} on the appropriate block marginals, every invariant subbundle admits a polynomial equivariant projector on the normalized support and a unique continuous, real-analytic representative (Corollary~\ref{cor:conjugated-polynomial-projectors}). The transported projector need not be orthogonal for the canonical Hodge metric. In contrast, a genus-two example in Section~\ref{subsec:arbitrary-conjugations} shows that nonrotational conjugation can destroy Deligne-torus invariance and hence the Hodge-compatible classification in Theorem A.

A second genus-two example, in Section~\ref{subsec:sharpness-full-product}, distinguishes the original group blocks from the finer local quadric factors. For a diagonal action, independent area normalization of the components can introduce unavoidable square roots such as $\sqrt{A_2/A_1}$, even though the projection is polynomial on the normalized support. Thus the integer area-denominator conclusion of Theorem B cannot in general be extended to independent scalings within a diagonal block.

\vspace{.5cm}

Section~\ref{sect1} recalls background on translation surfaces, strata, and products of strata, including the structure of projection-surjective
subgroups of $(\SLgr)^n$ and the standing diagonal-block hypothesis used later.
Section~\ref{sect2} reviews cocycles and algebraic hulls in the sense of \cite{Filip2013} and establishes reductivity for the cocycles considered here.
Section~\ref{sect3} summarizes the Hodge-theoretic and curvature identities used later.
Section~\ref{sect4} records the random-walk and subharmonicity input used to pass from subharmonicity to constancy.
Section~\ref{sect5} proves the fixed-part and semisimplicity results, extends the classification to finite direct sums of different weights, and explains by example why arbitrary conjugations of diagonal blocks must be excluded.
Section~\ref{sect6} proves leafwise and joint polynomiality under \textup{(S)}, obtains common-block cone localization and its full-product specialization, and gives a counterexample to independent factorwise localization. Appendix~\ref{app:local-geometry} derives the local quadric model, Appendix~\ref{app:separate-polynomiality} proves the separate-polynomiality theorem, and Appendix~\ref{app:tempered-regularity} verifies the tempered regularity needed for the invariance principle.

\vspace{.5cm}

\noindent\textbf{Acknowledgements.} The author of the paper is deeply grateful to Simion Filip for his thoughtful guidance and support during the author's work on this paper.

\section{Flat surfaces}\label{sect1}
This section is largely based on surveys by Filip \cite{Filip2024}, Matheus \cite{Matheus2018}, and Wright \cite{Wright2014}.
We also recommend the surveys by Zorich \cite{Zorich2006} and Massart \cite{Massart2021} for more introductory-level reading.

\subsection{Definitions and motivation}\label{sect1.1}

Denote by $\cM_g$ the moduli space of genus $g\ge 1$ Riemann surfaces (an orbifold).
Let $\Omega\cM_g$ be the (orbifold) Hodge bundle over $\cM_g$, whose fiber over $X\in\cM_g$ is the complex vector space
$H^0(X,\Omega^1_X)$ of holomorphic $1$--forms on $X$.
A point of $\Omega\cM_g$ is a pair $(X,\omega)$ with $\omega\not\equiv 0$, called an \emph{Abelian differential}.

Fix a topological surface $S_g$ of genus $g$.
Let $\Diff^+(S_g)$ be the group of orientation-preserving diffeomorphisms and let $\Diff^+_0(S_g)$ be its identity component.
The (classical) \Teichmuller~space is
$$
\cT_g:=\{\text{complex structures on }S_g\}/\Diff^+_0(S_g),
$$
and the mapping class group is $\MCG_g:=\Diff^+(S_g)/\Diff^+_0(S_g)$.
Let $\Omega\cT_g$ be the pullback of $\Omega\cM_g$ to $\cT_g$; equivalently, $\Omega\cT_g$ is the space of \emph{marked}
Abelian differentials.
Then $\MCG_g$ acts properly discontinuously on $\Omega\cT_g$ and
$$
\Omega\cM_g=\Omega\cT_g/\MCG_g.
$$
For notational convenience we write $\cH_g:=\Omega\cM_g$ and $\cT_g:=\Omega\cT_g$.

$$
\begin{tikzpicture}[scale=2, node/.style={}]
\node[node] (A) at (0, 0) {$\cT_g$};
\node[node] (B) at (2, 0) {$\cH_g$};
\draw [thick, ->] (A) -- (B) node[midway, above] {$\cdot/\MCG_g$};
\end{tikzpicture}
$$

\begin{remark}\label{flatsurfremark}
Let $(X,\omega)$ be a translation surface of genus $g\ge 1$. For $g\ge 2$, let $\Sigma=\{P_1,\dots,P_m\}$ be the zero set of $\omega$, and write $a_i=\mathrm{ord}_{P_i}(\omega)>0$. For $g=1$, the form $\omega$ has no zeros; by convention, we choose a marked regular point $P_1$, set $\Sigma=\{P_1\}$, and put $a_1=0$. In either case, $\sum_{i=1}^m a_i=2g-2$, and the tuple $\alpha=(a_1,\dots,a_m)$ determines the stratum $\cH(\alpha)\subset\cH_g$ containing $(X,\omega)$.

The form $\omega$ defines a singular flat metric on $X$ by declaring the length of a piecewise $C^1$ path $\gamma$ to be
$$
\ell_\omega(\gamma)=\int_\gamma |\omega|,
$$
and the distance between two points to be the infimum of the lengths of paths joining them.
Away from $\Sigma$, one can choose local coordinates in which $\omega=dz$, so this metric is locally Euclidean.
If $P_i\in\Sigma$ is a zero of order $a_i$, then in a local coordinate centered at $P_i$ one has
$$
\omega=z^{a_i}\,dz,
$$
and the induced metric is a Euclidean cone metric of total angle $2\pi(a_i+1)$ at $P_i$.
\end{remark}

\begin{remark}
The projection $\cH_g\to\cM_g$ is the Hodge orbifold bundle with its zero section removed. Its fiber over $X$ consists of nonzero elements of $H^{1,0}(X)\cong H^0(X,\Omega_X^1)$. This should be distinguished from the cohomological Hodge bundle, whose fiber is $H^1(X;\bbC)$ and which carries the flat Gauss--Manin connection. After choosing a marking of the underlying topological surface $S_g$, the monodromy of the latter is given by the mapping class group action
\[
\MCG_g\longrightarrow \mathrm{Sp}(H^1(S_g;\bbZ)).
\]
The Hodge subspaces $H^{1,0}(X)\subset H^1(X;\bbC)$ vary with the complex structure of $X$ and are generally not preserved by Gauss--Manin parallel transport. Thus the above monodromy belongs to the cohomology local system, not to the bundle of holomorphic differentials.
\end{remark}

The space $\cH_g$ admits a stratification by the zero orders:
\[
\cH_g=\bigsqcup_{\alpha}\cH(\alpha),
\qquad
\alpha=(a_1,\dots,a_m),\quad \sum_{i=1}^m a_i=2g-2,
\]
where $a_i>0$ for $g\ge 2$, while in genus one we use the convention $\alpha=(0)$. In particular, $m\ge 1$ in all cases. Each stratum $\cH(\alpha)$ is a complex-analytic orbifold of (complex) dimension $2g+m-1$. A stratum need not be connected; for example, $\cH(4)$ and $\cH(2,2)$ in genus $3$ have two connected components
(see \cite{Kontsevich-Zorich2002}).

Points of $\cH_g$ are also called \emph{translation surfaces} or \emph{flat surfaces}.
Indeed, on $X\setminus \Sigma$ the $1$--form $\omega$ provides local coordinates by integrating $\omega$,
and the transition maps between these coordinates are translations.
Thus the induced metric is flat on $X\setminus \Sigma$ with conical singularities at the zeros of $\omega$. As a consequence, one can think of a translation surface as a polygon with pairwise equal and parallel sides, where the coordinate change maps glue those sides together and the choice of $\omega$ is defined by choosing a directional vector.

\begin{figure}[h]
\centering
\begin{subfigure}{.4\textwidth}
  \centering
  \includegraphics[width=.89\linewidth]{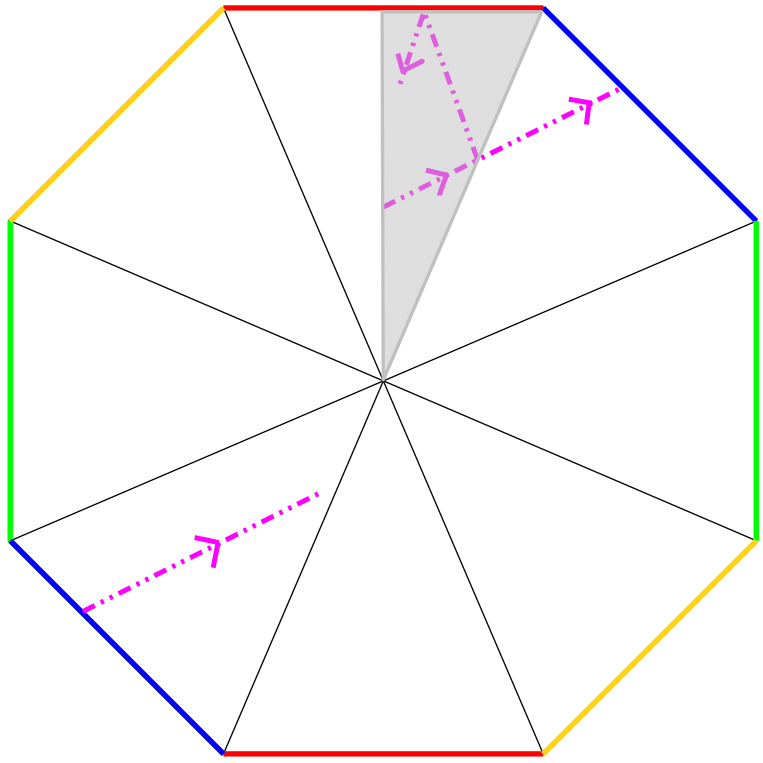}
  \label{fig:sub1}
\end{subfigure}%
\begin{subfigure}{.6\textwidth}
  \centering
  \includegraphics[width=.95\linewidth]{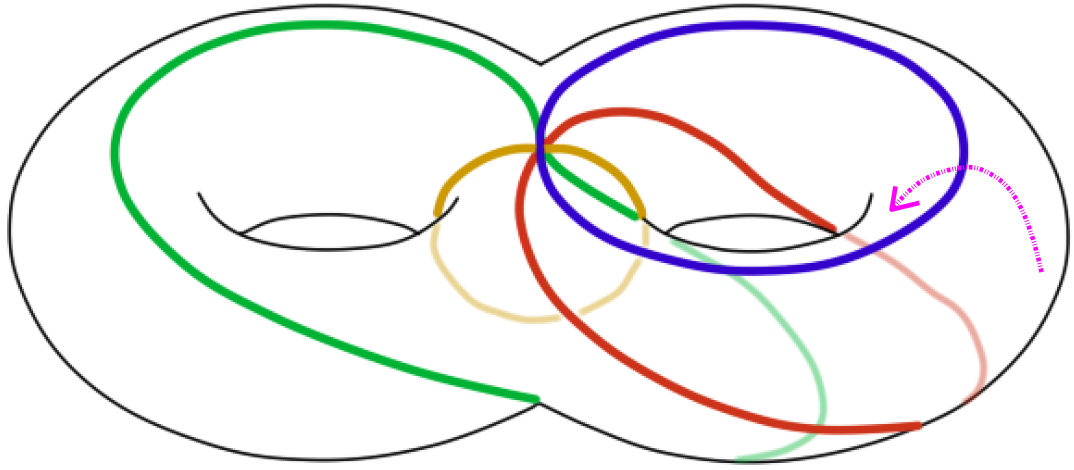}
   \vspace{.75cm}
  \label{fig:sub2}
\end{subfigure}
\caption{Unfolded billiard in the triangle with angles $(\pi/2,\,3\pi/8,\,\pi/8)$ and the flow on the resulting surface.}
\label{fig:test}
\end{figure}

\begin{remark}One of the motivations for studying flows on surfaces comes from \emph{polygonal billiards} with rational angles. Given such a rational polygon $P$ with the set of vertices $V$, we can glue a (topologically) sphere $S$ from it and its mirror copy by identifying corresponding sides. This sphere has a metric induced by the flat metric on the polygon, the holonomy of $S/V$ naturally embeds into $SO(2,\bbR)$. The holonomy representation \quad $\mathrm{rep}: \pi_1(S/V)\to SO(2,\bbR)$ induces the representation of the first homology $\mathrm{rep}: H_1(S/V,\bbR)\to SO(2,\bbR)$. Moreover, since all angles of $P$ are rational, the generators of the first homology are represented by rational angle rotation matrices. Then the image of this representation is a finite subgroup $G_P\subset SO(2,\bbR)$, so we can construct a finite $G_P$-cover $S_{G_P}$ of $S$, on which the holonomy is trivial. Thus we can "lift" a billiard trajectory into a straight line on $S_{G_P}$ unless it passes through one of the vertices. 

A simple example for a triangular rational billiard with the right angle is shown in the pictures.
\end{remark}

Let $(X,\omega)\in\cH(\alpha)$, and let $\Sigma$ consist of the zeros of $\omega$, or of a single marked regular point in genus one, as above. Write $m=|\Sigma|$.
Choose a basis $\{\gamma_1,\dots,\gamma_{2g+m-1}\}$ of the relative homology group $H_1(X,\Sigma;\bbZ)$.
The \emph{period coordinates} in a neighbourhood of $(X,\omega)$ are given by
$$
\Phi(X,\omega)=\Big(\int_{\gamma_1}\omega,\dots,\int_{\gamma_{2g+m-1}}\omega\Big)\in \bbC^{2g+m-1}.
$$
Changes of basis in $H_1(X,\Sigma;\bbZ)$ give transition maps in $GL(2g+m-1,\bbZ)$.
Identifying $\bbC\cong \bbR^2$ via $x+iy\mapsto (x,y)$, the natural $\SLgr$--action on period coordinates is given by
the standard action on the $\bbR^2$ factor:
\begin{equation*}
  \begin{pmatrix}
        a & b\\
        c & d
    \end{pmatrix}
    \begin{pmatrix}
        x_1 & \ldots & x_{2g+m-1}\\
        y_1 & \ldots & y_{2g+m-1}
    \end{pmatrix}
    =
    \begin{pmatrix}
        ax_1+by_1 & \ldots & ax_{2g+m-1}+by_{2g+m-1}\\
        cx_1+dy_1 & \ldots & cx_{2g+m-1}+dy_{2g+m-1}
    \end{pmatrix}.
\end{equation*}
The \emph{\Teichmuller~geodesic flow} is the action of the diagonal subgroup $g_t=\mathrm{diag}(e^t,e^{-t})$ on each stratum.

The \emph{(cohomological) Hodge bundle} over $\cM_g$ (with coefficients $k=\bbR$ or $k=\bbC$) is the local system
$H^1_k\to \cM_g$ whose fiber over $X$ is $H^1(X;k)$.
It carries the Gauss--Manin connection, i.e.\ parallel transport in cohomology along paths in $\cM_g$.
Pulling back to $\cH_g$ via the projection $\cH_g\to \cM_g$ gives a bundle (still denoted $H^1_k\cH_g$).

The $\SLgr$--action on strata induces the \Teichmuller~flow $g_t$ on $\cH(\alpha)$.
Lifting $g_t$ to $H^1_k\cH_g$ by Gauss--Manin parallel transport defines the \emph{Kontsevich--Zorich cocycle}
(see Subsection~\ref{KZsubsection}).

\begin{figure}
\centering
  \includegraphics[width=.9\linewidth]{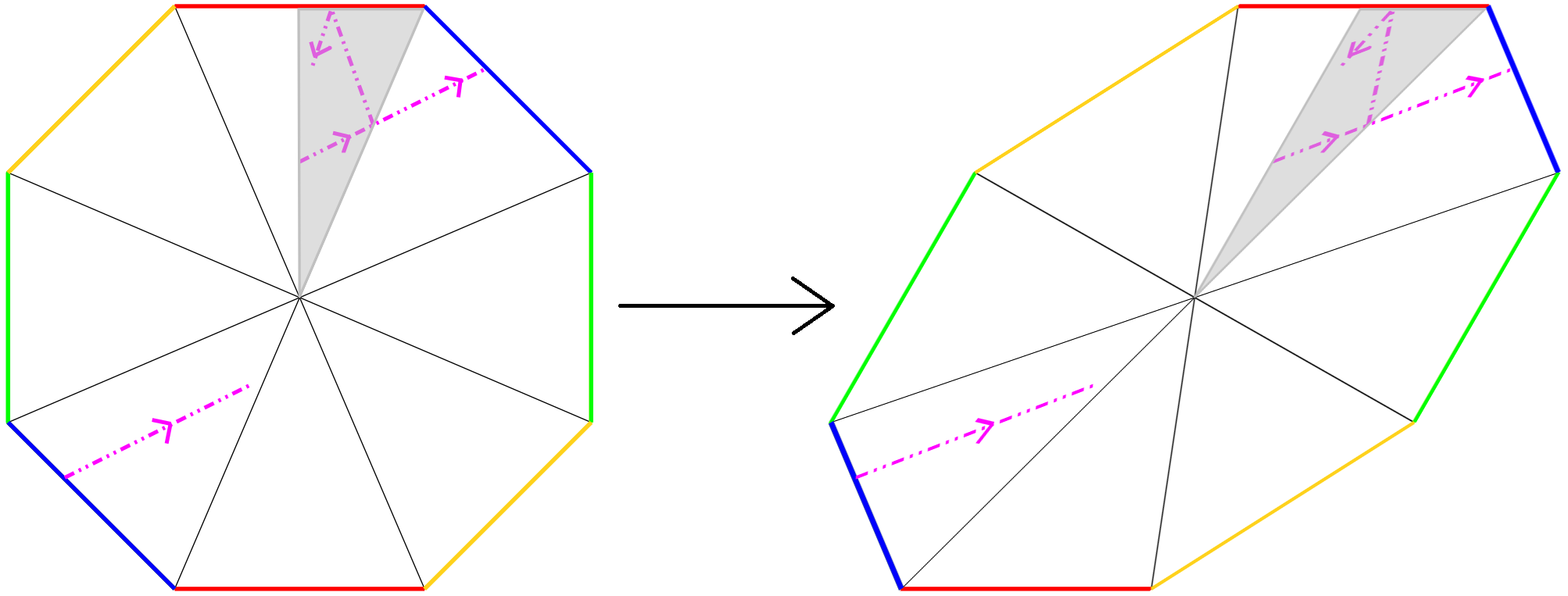}
\caption{Action of the matrix 
$\begin{pmatrix}
    1 & \frac{1}{\sqrt{3}}\\
    0 & 1
\end{pmatrix}$
on the surface from the previous example.}
\end{figure}

\subsection{The Kontsevich--Zorich cocycle}\label{KZsubsection}

Let us consider in more detail the action of $\SLgr$ on the space $\cH_g$. Clearly, each stratum $\cH(\alpha)$ is invariant under this action. The $\SLgr$--action preserves the \emph{area} (or volume) function
$$
A(X,\omega):=\frac{i}{2}\int_X \omega\wedge \bar\omega,
$$
so it is often convenient to restrict to the area-one locus $\cH^{(1)}(\alpha):=\cH(\alpha)\cap\cH^{(1)}_g$, $\cH^{(1)}_g:=A^{-1}(1)$. 

On the marked space $\cT_g=\Omega\cT_g$, the cohomology bundle is canonically trivial:
$$
H^1_\bbR \cT_g \;\cong\; \cT_g\times H^1(S_g;\bbR).
$$
The $\SLgr$--action on $\cT_g$ lifts to an action on $H^1_\bbR\cT_g$ by acting trivially on the fiber.
Passing to the quotient $\cH_g=\cT_g/\MCG_g$ identifies fibers using the mapping class group action on $H^1(S_g;\bbR)$.
The resulting fiberwise linear maps define the \emph{Kontsevich--Zorich cocycle}; equivalently, it is Gauss--Manin
parallel transport along $\SLgr$--orbits in $\cH_g$.

Since $\cH^{(1)}(\alpha)$ is an orbifold, one may pass to a torsion-free finite-index subgroup of $\MCG_g$ (a finite manifold cover)
to avoid orbifold stabilizers; this does not affect Lyapunov-exponent computations.

\subsection{Products of strata}\label{sect1.2}

Consider an $n$-tuple of natural numbers $(g)_n=(g_1,\,\ldots,\,g_n)$, $g_i\geq 1$.  Consider Cartesian products $\cM^n=\cM_{(g)_n}=\cM_{g_1}\times\ldots\times\cM_{g_n}$ and $\cH^n=\cH_{(g)_n}=\cH_{g_1}\times\ldots\times\cH_{g_n}$.
Let $H^1_k\cM^n$, where $k=\bbR$ or $k=\bbC$, be the real (complex) Hodge bundle over $\cM^n$ (which is just the product $H^1_k\cM_{g_1}\times\ldots\times H^1_k\cM_{g_n}$), and let $H^1_k\cH^n$ be its pullback under the natural projection $\cH^n\to\cM^n$.
Like in the case of one stratum, we are interested in acting on $H^1_k\cH^n$ by subgroups of $(\SLgr)^n=\prod_{i=1}^{n}\SLgr$.

\begin{remark}\label{applications-remark}
A map $T$ on a finite-measure space $(X,\mu)$ is weakly mixing if and only if $T\times T$ is ergodic, but this is far from the
only reason to study products. Strata of Abelian differentials admit compactifications by nodal curves, which arise as limits
where one or more (possibly separating) closed curves shrink to length $0$. Such nodal surfaces can be viewed as configurations
of lower-genus components attached at nodes. A survey of results on this topic is given by Dozier \cite{Dozier2024}.
\end{remark}

\begin{figure}[h]
\centering
  \includegraphics[width=.8\linewidth]{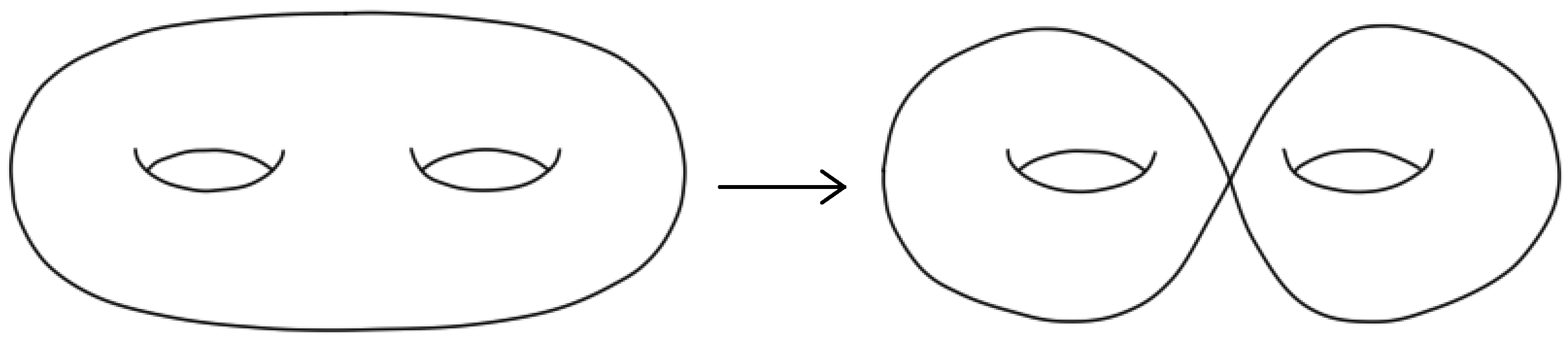}
\caption{Deformation of a genus $2$ surface into a nodal genus $(1,1)$ surface at the boundary of the compactification of $\cH_2$.}
\end{figure}

Given the motivations of our interest, we would like to consider only those subgroups of  $(\SLgr)^n$ that project surjectively onto each of its factors.
We will call such subgroups \emph{projection-surjective}.

\begin{lemma}[Normal form for projection-surjective groups]\label{prodlemma}
Let $G\subset(\SLgr)^n$ be a closed projection-surjective subgroup. There exist a partition
$\{1,\ldots,n\}=I_1\sqcup\cdots\sqcup I_r$, elements $g_i\in\GLgr$, and a subgroup
$F\leq\{\pm\Id\}^n$ such that
\[
 (g_1,\ldots,g_n)G(g_1,\ldots,g_n)^{-1}
 =(D_1\cdots D_r)F,
\]
where $D_\beta\cong\SLgr$ acts diagonally on the coordinates indexed by $I_\beta$ and trivially on the others. The identity component of the conjugated group is $D_1\cdots D_r$.
\end{lemma}

\begin{proof}
By the closed subgroup theorem, $G$ is a Lie group with at most countably many connected components. For each coordinate projection $p_i$, the image of $\operatorname{Lie}(G)$ is an ideal in $\mathfrak{sl}_2(\bbR)$, since $G^\circ$ is normal in $G$ and $p_i(G)=\SLgr$. This ideal is nonzero: otherwise $p_i$ would be constant on every connected component of $G$, giving a countable image. Simplicity therefore makes every Lie-algebra projection surjective.

The Lie-algebra form of Goursat's lemma now shows, by induction, that $\operatorname{Lie}(G)$ is a sum of diagonal copies of $\mathfrak{sl}_2(\bbR)$ on disjoint coordinate blocks, with the coordinates in a block identified by Lie-algebra automorphisms. Indeed, a new simple factor is either independent or identified with one simple factor of the previously obtained direct sum. Every automorphism of $\mathfrak{sl}_2(\bbR)$ is induced by an element of $\PGLgr$. After lifting these elements to $\GLgr$ and conjugating, the unique connected subgroup with this Lie algebra is $D=D_1\cdots D_r$. Hence $G^\circ=D$.

Every element of $G$ normalizes $D$. Its coordinatewise adjoint actions must therefore agree within each block. Two elements of $\SLgr$ with the same adjoint action differ by a central sign, so
\[
 N_{(\SLgr)^n}(D)=D Z,\qquad Z=\{\pm\Id\}^n.
\]
Write any element of $G$ as $dz$ with $d\in D$ and $z\in Z$. Since $D\subset G$, necessarily $z\in G\cap Z$. Taking $F=G\cap Z$ proves the assertion.
\end{proof}

\begin{remark}\label{prodremark}
The finite central group $F$ may couple different connected diagonal blocks. For example,
\[
 \{(g,\varepsilon g,h,\varepsilon h):
       g,h\in\SLgr,\ \varepsilon\in\{\pm1\}\}
\]
does not split as a direct product on the coordinate pairs $\{1,2\}$ and $\{3,4\}$, although its identity component and its projective image do. The normal form therefore uses one global subgroup $F$.

For the Hodge-theoretic results below we impose the stronger standing hypothesis that the required coordinatewise conjugations can be chosen in $SO_2(\bbR)$. These rotations preserve the canonical Hodge structures. Arbitrary conjugations by $\PGLgr$ provide a group-theoretic normal form but are not used to change the Hodge-theoretic dynamical system.
\end{remark}

\section{Properties of cocycles and the algebraic hull}\label{sect2}

Throughout this section, we assume that $X$ is a standard Borel space endowed with a probability measure $\mu$ and that $G$ is a Polish group whose measurable left action on $X$ is ergodic and measure-preserving. All cocycles and changes of trivialization are assumed measurable, and assertions about invariant subbundles are understood almost everywhere.

\begin{definition}[See \cite{Filip2013},  Section 2.1, p. 625]\label{invcocycledef}
  A \emph{linear cocycle} on the $n$-dimensional real (complex) vector bundle $V$ over $X$ is the lift of the action of the  group $G$ on $X$ to the action of  $G$ on $V$. More specifically, each element $g\in G$ induces a linear map of fibers $g_x: V_x\to V_{gx}$ for any $x\in X$. After choosing a measurable trivialization of $V$, we obtain a measurable map $\alpha:G\times X\to\mathrm{GL}(\mathbb{F}^n)$, where $\mathbb{F}=\bbR$ or $\bbC$, satisfying 
  $$
\alpha(gh,x)=\alpha(g,hx)\alpha(h,x)
\qquad
\text{for all }g,h\in G\text{ and $\mu$-a.e.\ }x\in X.
$$

  Two cocycles $\alpha,\beta$ are \emph{cohomologous} (or \emph{conjugate}) if there exists a measurable map
$C:X\to \mathrm{GL}(\mathbb{F}^n)$ such that
$$
\alpha(g,x)=C(gx)^{-1}\,\beta(g,x)\,C(x)
\qquad\text{for all }g\in G\text{ and $\mu$-a.e.\ }x\in X.
$$
\end{definition}

\subsection{Algebraic hulls} \label{AHsubsection}

Suppose that $H$ is a linear $\bbF$-algebraic group over a local field of characteristic $0$ (in our case, $\bbF=\bbR$ or $\bbF=\bbC$). Denote by $H(\bbF)$ its group of $\bbF$-rational points. For $H(\bbF)$-valued cocycles, cohomology is understood using measurable transfer maps $C:X\to H(\bbF)$.

\begin{theorem}[See \cite{Zimmer}, Chapter 9.2, pp. 166--167]\label{thmZim}
Given a cocycle $\alpha: G \times X \rightarrow H(\bbF)$, there exists a (unique up to conjugacy) $\bbF$-algebraic subgroup $L\subset H$ such that $\alpha$ can be conjugated into a cocycle $\alpha'$ taking values in $L(\bbF)$, and cannot be conjugated into a cocycle $\alpha''$ taking values in $M(\bbF)$ for any $\bbF$-algebraic subgroup $M\subsetneq L$.
\end{theorem}

Intuitively, we want to use the trivial fact that conjugating a cocycle amounts to changing its measurable trivialization, so we seek a minimal algebraic subgroup into which it can be conjugated. 

\begin{definition}
The equivalence class of the subgroup $L$ from Theorem \ref{thmZim} is called the \emph{algebraic hull} of $\alpha$.
\end{definition}

\begin{proofsketch}(After \cite{Zimmer})
The Zariski topology is Noetherian, so algebraic subgroups satisfy the descending chain condition. This implies the existence of a minimal algebraic subgroup into which $\alpha$ can be conjugated.

To prove uniqueness up to conjugacy, suppose that $L$ and $M$ are two such minimal subgroups. The corresponding reductions give $\alpha$-equivariant measurable maps
$$
\phi_L:X\to H(\bbF)/L(\bbF),
\qquad
\phi_M:X\to H(\bbF)/M(\bbF).
$$
Their product
$$
\phi=(\phi_L,\phi_M):
X\to H(\bbF)/L(\bbF)\times H(\bbF)/M(\bbF)
$$
is also $\alpha$-equivariant. The diagonal $H(\bbF)$-action on this product is smooth. By ergodicity, the image of $\phi$ therefore lies almost everywhere in a single $H(\bbF)$-orbit. Choosing a measurable section of this orbit gives a reduction of $\alpha$ to
$$
L\cap hMh^{-1}
$$
for some $h\in H(\bbF)$. Minimality of $L$ and $M$ implies 
$$
L=L\cap hMh^{-1}=hMh^{-1},
$$
which proves uniqueness up to conjugacy.
\end{proofsketch}

\subsection{Reductivity of cocycles}

\begin{lemma}[See \cite{Filip2013}, Section 2.3,
pp.~627--629]\label{replemma}
Let $\rho:H\to H'$ be an algebraic group homomorphism, and let $L$ be the algebraic hull of a cocycle $\alpha:G\times X\to H(\bbF)$. Then the algebraic hull of $\rho\circ\alpha$ is conjugate to $\rho(L)$.
\end{lemma}

\begin{definition}
A linear algebraic group over a field of characteristic $0$ is called \emph{reductive} if every finite-dimensional algebraic representation is semisimple. Equivalently, every invariant subspace in such a representation has an invariant complement. We call a cocycle reductive if its algebraic hull is reductive.
\end{definition}
Since images of reductive groups stay reductive, we obtain
\begin{corollary}\label{reductivitycor}
  Reductivity of cocycles is preserved under algebraic representations (in particular, tensor products of representations).
\end{corollary}

\subsection{Reductivity for Kontsevich--Zorich and products of strata}

Let $\cH^{(1)}(\alpha)$ be a stratum and let $\mu$ be an $\SLgr$-invariant ergodic probability measure on $\cH^{(1)}(\alpha)$
(or on an affine invariant submanifold). Recall that the Kontsevich--Zorich cocycle is a cocycle
$$
\alpha^{KZ}:\SLgr\times \cH^{(1)}(\alpha)\to \operatorname{Sp}(2g,\bbR)
$$
on the absolute cohomology bundle.  

\begin{remark}
Let $\cM$ be the affine invariant manifold whose area-one locus is the support of $\mu$. Choose an $\SLgr$-equivariant finite manifold cover $\pi:\widetilde{\cM}\to\cM$, for example one obtained from a torsion-free level structure as in Section~1. Work on a connected component upstairs with an ergodic invariant lift $\widetilde{\mu}$ of full support satisfying $\pi_*\widetilde{\mu}=\mu$, and let $\widetilde{\alpha}^{KZ}$ be the pullback of the Kontsevich--Zorich cocycle. Throughout this remark, the algebraic hulls downstairs and upstairs are taken with respect to $\mu$ and $\widetilde{\mu}$, respectively.
In general, the algebraic hull of $\widetilde{\alpha}^{KZ}$ need not coincide with that of $\alpha^{KZ}$, because passing to a finite cover amounts to restricting the monodromy representation to a finite-index subgroup of the orbifold fundamental group, and this can change the finite component group.

However, the identity component is unchanged. Indeed, if
$$
\rho:\pi_1^{\mathrm{orb}}(\cM)\to \mathrm{Sp}(2g,\bbR)
$$
is the monodromy representation on absolute cohomology and
$$
\widetilde{\rho}=\rho|_{\pi_1(\widetilde{\cM})},
$$
then $\pi_1(\widetilde{\cM})$ has finite index in $\pi_1^{\mathrm{orb}}(\cM)$, so $\widetilde{\rho}(\pi_1(\widetilde{\cM}))$ has finite index in $\rho(\pi_1^{\mathrm{orb}}(\cM))$.
Hence their Zariski closures have the same identity component.

Apply \cite[Theorem 1.1]{EskinFilipWright2018} to the absolute cohomology bundle over these affine supports: the algebraic hull is the stabilizer of the tautological plane inside the corresponding Zariski closure of monodromy. After identifying fibers at a point and a lift of that point, the tautological planes agree and the two monodromy closures have the same identity component. Their stabilizers of this plane therefore also have the same identity component.
In particular, reductivity is unchanged, since a linear algebraic group is reductive if and only if its identity component is reductive.
\end{remark}

\begin{theorem}[{\cite[Appendix A]{EM18}; \cite[Theorem 2.1(a)]{Avila-Eskin-Moller2017}}]
\label{KZ_reductive_thm}
The Kontsevich--Zorich cocycle $\alpha^{KZ}$ is reductive.
\end{theorem}

\begin{lemma}[Relative isotropic-isometry lemma]\label{lemma:relative-isotropic-isometry}
Let $p:(Z,\nu)\to(Y,\eta)$ be a measurable $\SLgr$-equivariant factor map of standard Borel probability-preserving $\SLgr$-spaces, where $Y$ is an area-one stratum (or an invariant measurable subset thereof) carrying the absolute Hodge bundle $H^1_{\bbR}$. If $K\subset p^*H^1_{\bbR}$ is a measurable $\SLgr$-invariant isotropic subbundle, then the pulled-back KZ cocycle acts by Hodge isometries on $K$, and its Hodge-orthogonal complement $K^{\perp_H}$ is $\SLgr$-invariant.
\end{lemma}

\begin{proof}
Write $\nu=\int\nu_\xi\,d\tau(\xi)$ for the ergodic decomposition of the $\SLgr$-action on the standard Borel probability space $Z$. Because $K$ is invariant, its rank is constant $\nu_\xi$-almost everywhere on almost every component; denote this rank by $k_\xi$. The marginal $\eta_\xi=p_*\nu_\xi$ is $\SLgr$-invariant and ergodic, being a factor of $(Z,\nu_\xi)$. Components with $k_\xi=0$ require no argument. Fix a component from the conull family with $k_\xi>0$ and write $k=k_\xi$ until the end of the componentwise proof. The equivariant measurable map
\[
 z\longmapsto\bigl(p(z),K_z\bigr)
\]
pushes $\nu_\xi$ forward to an $\SLgr$-invariant probability measure $\widehat\nu_\xi$ on the isotropic Grassmann bundle $\operatorname{IGr}_k(H^1_{\bbR})$ over $(Y,\eta_\xi)$. The tautological $k$-plane carries the ordinary KZ cocycle restricted to that plane. The Hodge metric and the second fundamental form there are inherited from the ordinary Hodge bundle on the stratum; no differentiable structure on $Z$ is required.

We use the notation $B_y$, $B_y^{\bbR}$, and $\Phi_k$ of \cite[Section~2]{ForniMatheusZorich2014LyapunovSpectrum}. Here $B_y$ is the complex bilinear form associated with the second fundamental form of the Hodge bundle, $B_y^{\bbR}$ is the corresponding bilinear form on real cohomology, and
\[
\operatorname{Ann}(B_y^{\bbR})=\{v\in H^1_y(\bbR):B_y^{\bbR}(v,w)=0\text{ for every }w\in H^1_y(\bbR)\}.
\]

Apply the Grassmannian formula of \cite[Theorem~1, (3.8)]{ForniMatheusZorich2014LyapunovSpectrum} to the full absolute Hodge bundle $V=H^1_{\bbR}$ over the ergodic measure $\eta_\xi$, taking its $SO_2(\bbR)$-invariant Grassmannian measure to be $\widehat\nu_\xi$. Since this measure is already invariant under the geodesic flow, it is an admissible weak limit in that formula. We obtain
\begin{equation}\label{eq:relative-isotropic-forni}
\Lambda^{(k)}(\widehat\nu_\xi)=\int_{\operatorname{IGr}_k(H^1_{\bbR})}\Phi_k(y,I)\,d\widehat\nu_\xi(y,I),
\end{equation}
where $\Lambda^{(k)}(\widehat\nu_\xi)$ is the mean asymptotic logarithmic Hodge-volume growth on the tautological $k$-plane. The standard Hodge norm estimate bounds the absolute value of its logarithmic Jacobian at time $t$ by $k|t|$. The asymptotic logarithmic volume growth exists almost everywhere by Oseledets' theorem, and this bound permits passage to its integral by dominated convergence.

In fact the mean volume growth vanishes without taking a Lyapunov limit. Put $a_t=\operatorname{diag}(e^t,e^{-t})$ and $r=\left(\begin{smallmatrix}0&-1\\1&0\end{smallmatrix}\right)$, and let $A_t(z):K_z\to K_{a_tz}$ denote the pulled-back KZ transport. Define
\[
 J_t(z):=\log\bigl|\det\nolimits_H A_t(z)\bigr|,
\]
where the determinant is computed with the Hodge volume norms on its source and target. Since $ra_tr^{-1}=a_{-t}$ and rotation by $r$ preserves the Hodge metric, equivariance of $K$ implies $J_t(rz)=J_{-t}(z)$ almost everywhere. The determinant cocycle identity gives $J_{-t}(z)=-J_t(a_{-t}z)$. Invariance of $\nu_\xi$ under both $r$ and $a_t$ therefore yields
\[
\int J_t\,d\nu_\xi=\int J_{-t}\,d\nu_\xi=-\int J_t\,d\nu_\xi,
\]
so $\int J_t\,d\nu_\xi=0$ for every $t$. Dividing by $t$ and using the uniform bound $|J_t|\leq k|t|$, we conclude that $\Lambda^{(k)}(\widehat\nu_\xi)=0$.

Take a Hodge-orthonormal basis $c_1,\ldots,c_k$ of an isotropic plane $I$ and extend it to an orthonormal Lagrangian basis $c_1,\ldots,c_g$. Let $\omega_1,\ldots,\omega_g$ be the corresponding holomorphic one-forms. The sum-of-squares identity \cite[(2.51)]{ForniMatheusZorich2014LyapunovSpectrum} is
\[
\Phi_k(y,I)=\sum_{a=1}^k\sum_{b=1}^g|B_y(\omega_a,\omega_b)|^2+\sum_{a=1}^k\sum_{b=k+1}^g|B_y(\omega_a,\omega_b)|^2\geq0.
\]
Equation~\eqref{eq:relative-isotropic-forni} implies $\Phi_k(y,I)=0$ for $\widehat\nu_\xi$-almost every $(y,I)$. Hence $B_y(\omega_a,\omega_b)=0$ for all $a\leq k$ and $b\leq g$, or equivalently $I\subseteq\operatorname{Ann}(B_y^{\bbR})$. Pulling this back gives
\[
K_z\subseteq\operatorname{Ann}(B_{p(z)}^{\bbR})\qquad\text{for $\nu_\xi$-almost every }z.
\]

For clarity, the final Hodge-orthogonality step is pointwise and does not require invoking a theorem about the extension. Let $E\subseteq Z$ be a measurable conull set on which the annihilator inclusion above holds. For every $T>0$, invariance of $\nu_\xi$ and Fubini give
\[
 \int_Z\int_{-T}^T\mathbf 1_{Z\setminus E}(a_tz)\,dt\,d\nu_\xi(z)=0.
\]
The measurable equivariance identity $A_t(z)K_z=K_{a_tz}$ also holds for almost every pair $(t,z)$ by Fubini, since it holds for $\nu_\xi$-almost every $z$ for each fixed $t$. Taking $T$ through the positive integers, we obtain a conull set of starting points $z$ for which both properties hold for Lebesgue-almost every $t\in\bbR$. Fix such a $z$ and arbitrary $v\in K_z$ and $w\in H^1_{p(z)}(\bbR)$; transport them by Gauss--Manin along $a_t$. The Hodge-pairing variation formula \cite[(2.39), Lemma~2.6]{ForniMatheusZorich2014LyapunovSpectrum} expresses the derivative of their Hodge pairing in terms of $B^{\bbR}$. For almost every $t$, the transported vector $A_t(z)v$ belongs to $K_{a_tz}\subseteq\operatorname{Ann}(B^{\bbR}_{p(a_tz)})$, so the derivative vanishes. The pairing is locally absolutely continuous in $t$ and hence constant for all $t$. Taking $w\in K_z$ proves Hodge isometry on $K$; taking $w\in K_z^{\perp_H}$ proves invariance of its Hodge-orthogonal complement under the diagonal flow. Rotations preserve the underlying complex structure and the Hodge metric, and therefore preserve $K^{\perp_H}$ by equivariance of $K$. Since rotations and the diagonal subgroup generate $\SLgr$, the conclusion holds for every fixed element of $\SLgr$, almost everywhere (using a finite intersection of the relevant conull sets for each such element), as in \cite[Lemma~4.3]{ForniMatheusZorich2014LyapunovSpectrum}. All these conclusions hold for $\tau$-almost every ergodic component; integrating the componentwise null sets against $\tau$ gives the assertions for $\nu$.
\end{proof}

\begin{lemma}[Measurable complements and algebraic hulls]\label{lemma:measurable-complements-hull}
Let $\alpha$ be a measurable $\mathrm{GL}(V)$-valued cocycle over an ergodic probability-preserving action on a standard Borel probability space, where $V$ is a finite-dimensional real vector space. If every invariant measurable subbundle of the associated bundle admits an invariant measurable complement, then the algebraic hull of $\alpha$ is reductive.
\end{lemma}

\begin{proof}
Let $L\subset\mathrm{GL}(V)$ be the algebraic hull and use an $L$-valued measurable trivialization. If $L$ were not reductive, its faithful defining representation on $V$ could not be completely reducible: otherwise the connected unipotent radical of $L$ would act trivially on every irreducible constituent, hence trivially on $V$, contrary to faithfulness. Thus there is an $L$-invariant subspace $U\subset V$ with no $L$-invariant complement. Its constant measurable bundle is $\alpha$-invariant, so by hypothesis it has an invariant measurable complementary bundle $W_x$. This determines an $\alpha$-equivariant measurable map from the base to the real algebraic variety $\mathcal C(U)$ of complements of $U$ in $V$.

The algebraic $L$-action on $\mathcal C(U)$ has locally closed orbits and a countably separated orbit space by \cite[Proposition~5.5]{EskinFilipWright2018} (see also the algebraic-hull uniqueness argument of Theorem~\ref{thmZim}). Consequently the invariant measurable orbit class of $W_x$ is almost surely constant by ergodicity: $W_x\in L\cdot W_0$ almost everywhere. A measurable section of $L\to L\cdot W_0$ changes trivialization so that the cocycle takes values in the algebraic stabilizer $\operatorname{Stab}_L(W_0)$. This stabilizer is proper because $U$ has no $L$-invariant complement, contradicting the minimality of $L$ as algebraic hull. Hence $L$ is reductive.
\end{proof}

Now let
$$
X=\cH^{(1)}(\alpha_1,\ldots,\alpha_n)=\cH^{(1)}(\alpha_1)\times\cdots\times \cH^{(1)}(\alpha_n),
$$
let $G\subset (\SLgr)^n$ be a closed projection-surjective subgroup, and let $\mu$ be a $G$-invariant ergodic probability measure on $X$.
Let $\Gamma$ denote the cocycle on the (absolute) cohomology bundle
$$
\rmH^1_\bbR(X)\;=\;\bigoplus_{i=1}^n \pi_i^*\rmH^1_\bbR(\cH^{(1)}(\alpha_i))
$$
induced by the Gauss--Manin connection; equivalently, $\Gamma$ is the restriction to $G$ of the product cocycle for $(\SLgr)^n$.

\begin{lemma}\label{main-reductive-lemma}
The cocycle $\Gamma$ has reductive algebraic hull, even without the rotational-conjugacy restriction on $G$.
Consequently, every fixed algebraic tensor construction induced from $\rmH^1_\bbR(X)$ has reductive algebraic hull; the same holds after complexification.
\end{lemma}

\begin{proof}
Fix a coordinate $i$, and put $E_i=\pi_i^*H^1_{\bbR}(\cH^{(1)}(\alpha_i))$. By Lemma~\ref{prodlemma}, the identity component of $G$ contains a simple subgroup $S_i\cong\SLgr$ whose projection $p_i:S_i\to\SLgr$ is an isomorphism. This conclusion does not require the rotational-conjugacy restriction used later for Hodge-theoretic arguments. We first show that every measurable $G$-invariant real subbundle $U\subset E_i$ has a measurable $G$-invariant complement.

Write $U^\dagger$ for the symplectic orthogonal of $U$, and let $K=U\cap U^\dagger$ be its symplectic radical. The intersection form is flat, so $U^\dagger$ and $K$ are $G$-invariant, and $K$ is isotropic. Restrict the $G$-action to $S_i$ and identify $S_i$ with $\SLgr$ through $p_i$. The coordinate projection $\pi_i:(X,\mu)\to(\cH^{(1)}(\alpha_i),(\pi_i)_*\mu)$ is equivariant for this action, and the cocycle on $E_i$ is the pullback of the ordinary KZ cocycle. Lemma~\ref{lemma:relative-isotropic-isometry} shows that $K^{\perp_H}$ is invariant under $S_i$.

In fact, $K^{\perp_H}$ is $G$-invariant. Indeed, the kernel $N_i=\ker(p_i:G\to\SLgr)$ fixes the $i$th surface and acts as the identity on its absolute cohomology, so it preserves the Hodge metric on $E_i$. Since $K$ is $G$-invariant, $K^{\perp_H}$ is therefore invariant under $N_i$ as well. Surjectivity of $p_i|_{S_i}$ gives $G=S_iN_i$, proving the claim.

Set $U_0=U\cap K^{\perp_H}$. Because $K\subset U$ and the Hodge metric is positive definite, $U=K\oplus U_0$. The symplectic form is nondegenerate on $U_0$: if $u\in U_0$ pairs trivially with $U_0$, it also pairs trivially with $K\subset U^\dagger$, hence with all of $U$; thus $u\in K\cap K^{\perp_H}=\{0\}$. Consequently $E_i=U_0\oplus U_0^\dagger$. Since $E_i=K\oplus K^{\perp_H}$ and $K\subset U_0^\dagger$, we have $U_0^\dagger=K\oplus(U_0^\dagger\cap K^{\perp_H})$. It follows that
\[
E_i=U\oplus\bigl(U_0^\dagger\cap K^{\perp_H}\bigr).
\]
Every subbundle in this decomposition is measurable and $G$-invariant. This proves that every $G$-invariant measurable subbundle of $E_i$ has a $G$-invariant measurable complement. By Lemma~\ref{lemma:measurable-complements-hull}, the algebraic hull of the coordinate cocycle $\Gamma_i$ is reductive.

Now let $L\subset\prod_{i=1}^n\operatorname{Sp}(2g_i,\bbR)$ be the algebraic hull of the full direct-sum cocycle $\Gamma$. By Lemma~\ref{replemma}, the coordinate projection $\rho_i(L)$ is conjugate to the algebraic hull of $\Gamma_i$, so it is reductive. If $R_u(L)$ is the connected unipotent radical of $L$, its image under each $\rho_i$ is a connected normal unipotent subgroup of the reductive group $\rho_i(L)$ and hence trivial. The coordinate projections separate points, whence $R_u(L)=\{e\}$ and $L$ is reductive. The conclusion for fixed algebraic tensor constructions follows from Corollary~\ref{reductivitycor}, and complexification preserves reductivity.
\end{proof}

\section{Hodge-theoretic preliminaries}\label{sect3}

\subsection{Setup}
\begin{definition}\label{hodge_str_def}
Let $\rmH_\bbR$ be a finite-dimensional vector space over $\bbR$, and let $\rmH_\bbC=\rmH_\bbR\otimes_\bbR\bbC$ be its complexification. A \emph{pure Hodge structure of weight $w$ on $\rmH_\bbR$} is a finite direct-sum decomposition
\begin{equation}\label{hodge_str_eq}
\rmH_\bbC=\bigoplus_{p+q=w}\rmH^{p,q},
\qquad
\overline{\rmH^{p,q}}=\rmH^{q,p}.
\end{equation}
The indices $p,q$ are arbitrary integers; only finitely many summands are nonzero. The associated decreasing \emph{Hodge filtration} is
\begin{equation}\label{filtration_eq}
\rmF^p\rmH_\bbC=\bigoplus_{i\geq p}\rmH^{i,w-i},
\qquad p\in\bbZ.
\end{equation}
It is eventually equal to $\rmH_\bbC$ for $p\ll0$ and to $0$ for $p\gg0$.
\end{definition}
The decomposition and filtration determine one another:
\[
\rmH^{p,q}=\rmF^p\rmH_\bbC\cap\overline{\rmF^q\rmH_\bbC},
\qquad p+q=w.
\]
For the effective weight-one Hodge structure on $H^1$ only the types $(1,0)$ and $(0,1)$ occur, but arbitrary integer Hodge indices arise naturally when taking duals, endomorphism bundles, and other tensor constructions.

\begin{remark}
  The idea of the Hodge structure is inspired by Hodge decomposition theorem for De Rham and Dolbeault cohomologies. While cohomologies are indeed the object whose Hodge structures we want to study in this paper, the following paragraph works in the generalized setting of a finite-dimensional vector space over $\bbR$, as the Hodge-theoretical results in this paragraph do not depend on any properties specific to cohomology, and therefore this extra data is simply unneeded.
\end{remark}

\begin{definition}\label{weil_op_def}
The \emph{Weil operator C} associated to a Hodge structure on $\rmH_\bbR$ is the operator that acts on each of the components $\rmH^{p,q}$ by multiplication by $(\sqrt{-1})^{p-q}$. This operator is real and therefore descends to $\rmH_\bbR$.
\end{definition}

\begin{definition}
A \emph{polarization} of a Hodge structure of weight $w$ on $\rmH_\bbR$ is a nondegenerate bilinear form
$I:\rmH_\bbR\times \rmH_\bbR\to \bbR$ which is $(-1)^w$-symmetric, satisfies
\[
 I(\rmH^{p,q},\rmH^{r,s})=0\quad\text{unless }(r,s)=(q,p),
\]
and is such that the Hermitian form
$$
Q(x,y):= I(Cx,\overline{y})\qquad (x,y\in \rmH_\bbC)
$$
is positive definite. The induced norm $\|x\|:=\sqrt{Q(x,x)}$ is called the \emph{Hodge norm}.
\end{definition}

\begin{definition}
A \emph{morphism} of Hodge structures of the same weight $w$ is an $\bbR$-linear map
$\phi:\rmH_\bbR\to \tilde{\rmH}_\bbR$ such that its complexification satisfies
$\phi_\bbC(\rmH^{p,q})\subseteq \tilde{\rmH}^{p,q}$ for all $p+q=w$
(equivalently, $\phi_\bbC(\rmF^p\rmH_\bbC)\subseteq \tilde{\rmF}^p\tilde{\rmH}_\bbC$ for all $p$).
\end{definition}

\begin{remark}[Hodge structures via the Deligne torus]
Let $\mathbb S:=\operatorname{Res}_{\bbC/\bbR}\mathbb{G}_m$ be the \emph{Deligne torus}. Then
$\mathbb S(\bbR)=\bbC^\times$ (via $a+ib\mapsto \begin{pmatrix}a&-b\\ b&a\end{pmatrix}$), and
$\mathbb S(\bbC)\cong \bbC^\times\times \bbC^\times$; complex conjugation on $\bbC$ exchanges the two factors.

A pure Hodge structure of weight $w$ on a real vector space $\rmH_\bbR$ is equivalent to an algebraic homomorphism
$h:\mathbb S\to GL(\rmH_\bbR)$ such that the restriction to the embedded $\mathbb{G}_{m,\bbR}\hookrightarrow \mathbb S$
(given on real points by $r\mapsto r\in\bbC^\times$) acts by $r\mapsto r^{\,w}\Id$.
After complexifying, the representation of $\mathbb S(\bbC)\cong \bbC^\times\times \bbC^\times$ yields a decomposition
$$
\rmH_\bbC=\bigoplus_{p+q=w}\rmH^{p,q},\qquad
\rmH^{p,q}=\{v\in \rmH_\bbC:\ h(z_1,z_2)v=z_1^{p}z_2^{q}v\},
$$
and complex conjugation satisfies $\overline{\rmH^{p,q}}=\rmH^{q,p}$.
Conversely, a decomposition $\rmH_\bbC=\bigoplus_{p+q=w}\rmH^{p,q}$ determines $h$ by
$h(z)v=z^{p}\bar z^{q}v$ for $z\in\bbC^\times$ and $v\in \rmH^{p,q}$.

With this convention the Weil operator is $C=h(i)$, so $C|_{\rmH^{p,q}}=i^{p-q}$.
A morphism of Hodge structures is equivalently an $\mathbb S$-equivariant $\bbR$-linear map.
Shifting every Hodge type $(p,q)$ to $(p+m,q+m)$ changes the weight by $2m$ but leaves the induced Hodge structure on $\operatorname{End}(\rmH_\bbR)$ unchanged.
For a pure Hodge structure $\rmH_\bbR$, the induced Hodge structure on $\operatorname{End}(\rmH_\bbR)$ has weight $0$ and
$$
\operatorname{End}(\rmH_\bbC)^{r,-r}
=\{T:\rmH^{p,q}\to \rmH^{p+r,q-r}\ \text{for all }p,q\},
\qquad
\operatorname{End}(\rmH_\bbR)^{0,0}=\End_{\mathrm{HS}}(\rmH_\bbR).
$$
\end{remark}

\begin{definition}
Let $A$ be an $\bbR$-algebra. A Hodge structure on $A$ is \emph{compatible with the algebra structure} if
the multiplication map $m:A\otimes A\to A$ is a morphism of Hodge structures.
If this Hodge structure is pure and $A$ is a nonzero unital algebra, its weight must be $0$.
\end{definition}

\begin{lemma}\label{end-hodge-lemma}
    Let $H$ be a finite-dimensional real vector space. A Hodge structure on $\operatorname{End}(H)$ which is compatible with the algebra structure comes from a Hodge structure on $H$, unique up to shift of weight.

The same statement is true for $H$ a module over a division algebra $A$. Namely, a Hodge structure on $\operatorname{End}_A(H)$ compatible with the algebra structure gives a Hodge structure on $H$, compatible with the $A$-action (i.e. such that $A$ acts by endomorphisms of the Hodge structure on $H$).
\end{lemma}
The complex version holds with Hodge structures interpreted as bigradings. See \cite[Lemma~3.9 and Remark~3.10]{Filip2013} for these statements, including compatibility with polarizations.

\begin{definition}\label{VHSdef}
Let $M$ be a complex manifold and let $\rmH\to M$ be a finite-dimensional real flat vector bundle with flat Gauss--Manin connection $\nabla^{GM}$. A \emph{variation of Hodge structures of weight $w$} on $M$ is a decreasing filtration of the complexification by holomorphic subbundles
\[
\cdots\supset \rmF^{p}\rmH_\bbC\supset \rmF^{p+1}\rmH_\bbC\supset\cdots,
\]
which is locally finite (so $\rmF^p=\rmH_\bbC$ for $p\ll0$ and $\rmF^p=0$ for $p\gg0$), such that:
\begin{enumerate}
\item for every $s\in M$, the filtration $\rmF^\bullet\rmH_\bbC(s)$ is the Hodge filtration of a pure Hodge structure of weight $w$; equivalently,
\[
\rmH^{p,q}(s)=\rmF^p(s)\cap\overline{\rmF^q(s)},
\qquad p+q=w,
\]
and the Hodge numbers $h^{p,q}=\dim\rmH^{p,q}(s)$ are locally constant;
\item the flat connection satisfies Griffiths transversality:
\[
\nabla^{GM}(\rmF^p)\subset \rmF^{p-1}\otimes\Omega_M^1
\qquad\text{for every }p\in\bbZ.
\]
\end{enumerate}

The variation is \emph{polarized} if the underlying real local system carries a flat polarization $I$ whose restriction to every fiber polarizes its Hodge structure; we use the same symbol for the complex-bilinear extension of $I$ to $\rmH_\bbC$.
\end{definition}

\begin{definition}[Finite direct sums of pure Hodge structures]
A \emph{split Hodge structure with finitely many weights} is a direct sum
\[
 H=\bigoplus_w H^{[w]},
\]
with only finitely many nonzero summands, each a pure Hodge structure of weight $w$. Morphisms preserve every Hodge type. A split variation is a direct sum of pure variations, with its weight summands flat. We equip it with the orthogonal sum of the pure Hodge metrics. In the complex case the same convention uses complex Hodge structures and variations. In particular,
\[
 \operatorname{End}(H)=\bigoplus_{a,b}\operatorname{Hom}(H^{[a]},H^{[b]}),
\]
whose displayed summands have weights $b-a$; the full endomorphism algebra need not have weight zero.
\end{definition}

\subsection{Differential-geometric properties of Hodge bundles}

Consider a holomorphic vector bundle $E$ over an arbitrary complex manifold $M$. The bundle is endowed with a (pseudo-)hermitian metric $ \langle\cdot,\cdot\rangle$.

\begin{definition}
    The \emph{Chern connection} on $E$-valued smooth differential forms is the unique connection $\nabla: \cA^{\bullet}(E) \rightarrow \cA^{\bullet+1}(E)$ that satisfies the following two conditions:
    \begin{enumerate}
        \item For any $\alpha,\,\beta\in \cA^{\bullet}(E)$, 
        \begin{equation}
        \label{Chern_eq_1}
        d\langle\alpha,\beta\rangle = \langle\nabla\alpha,\beta\rangle + (-1)^{\deg \alpha}\langle\alpha,\nabla\beta\rangle;
        \end{equation}
        \item For any holomorphic section $s$ of $E$,
        \begin{equation}
        \label{Chern_eq_2}
        \nabla(s) \in \cA^{1,0}(E).
        \end{equation}
    \end{enumerate}
\end{definition}
The Chern connection can be given explicitly if we fix a local holomorphic basis of $E$. Let $h$ be the matrix of the Hermitian metric in this basis, and let $A$ be the matrix of $(1,0)$-forms. Note that $A=h^{-1}\partial h$. The Chern connection is then (locally) given by the form 
$$
\nabla=d+A.
$$
Let $\Omega$ denote the curvature of the Chern connection $\nabla$.In the same local trivialization, 
$$
\Omega=dA+A\wedge A.
$$

\begin{remark}
In some references, the curvature appears as $dA - A\wedge A$ rather than $dA + A\wedge A$. This depends on how one treats the sign in $\nabla^2 f = (dA)f \mp A(Af)$.  Once $A$ is viewed as a matrix of one-forms, it naturally appears with a plus sign in $dA + A\wedge A$.
\end{remark}

We begin with a second variation formula for the logarithm of the norm of a holomorphic section. We then obtain curvature identities for quotient and subbundles, later specializing to Hodge bundles. The overall treatment follows the reference \cite{Filip2013} with very minor differences and is included for the sake of convenience only.

\subsection{A Second Variation Formula}

\begin{lemma}
\label{lemma:delbar_del_log_rewrite}
Assume that the Hermitian metric on $E$ is positive definite, and let $\phi$ be a holomorphic section of $E$. On the complement of its zero locus,
\begin{equation}\label{deldelbarlog_eq}
    \delbar\del \log \|\phi\|^2 = 
\frac{\langle \Omega\phi,\,\phi\rangle}{\|\phi\|^2}
+
\frac{\langle \nabla\phi,\,\phi\rangle\,\langle \phi,\,\nabla\phi\rangle
-\|\phi\|^2\,\langle \nabla\phi,\,\nabla\phi\rangle}{\|\phi\|^4}.
\end{equation}
\end{lemma}

\begin{proof}
Since $\phi$ is holomorphic, $\del\|\phi\|^2$ is simply the $(1,0)$-part of $d\|\phi\|^2$. By definition, for any $\alpha,\,\beta\in \cA^{\bullet}(E)$, 
$$
d\langle\alpha,\beta\rangle =\del\langle\alpha,\beta\rangle +\delbar \langle\alpha,\beta\rangle.
$$
From this and the properties of the Chern connection, we obtain
$$
\del\|\phi\|^2=\langle\nabla \phi, \phi\rangle.
$$
Therefore
$$
\del \log \|\phi\|^2=\frac{1}{\|\phi\|^2} \cdot \del\|\phi\|^2=\frac{\langle\nabla \phi, \phi\rangle}{\|\phi\|^2}.
$$
By the product rule,
$$
\delbar(\del \log \|\phi\|^2)=
(\delbar \frac{1}{\|\phi\|^2})\langle\nabla \phi, \phi\rangle+
\frac{1}{\|\phi\|^2} \delbar\langle\nabla \phi, \phi\rangle.
$$
Calculating the first term is straightforward:
$$
\delbar \frac{1}{\|\phi\|^2}=\frac{-1}{\|\phi\|^4} \cdot \delbar\|\phi\|^2=-\frac{\langle\phi, \nabla \phi\rangle}{\|\phi\|^4}.
$$
To calculate the second term, note that 
$$
\delbar\langle\nabla \phi, \phi\rangle=d\langle\nabla \phi, \phi\rangle.
$$
Indeed, its $(2,0)$-part is
$\del\langle\nabla\phi,\phi\rangle=\del^2\|\phi\|^2=0$. Therefore, $d\langle\nabla \phi, \phi\rangle$ only has a $(1,1)$-part.
By the properties of the Chern connection,
$$
d\langle\nabla \phi, \phi\rangle=
\langle\Omega \phi, \phi\rangle-
\langle\nabla \phi, \nabla \phi\rangle.
$$
Combining, obtain
\begin{align*}
\delbar\del \log \|\phi\|^2 
&=
-\frac{\langle\phi, \nabla \phi\rangle}{\|\phi\|^4}\langle\nabla \phi, \phi\rangle
+\frac{1}{\|\phi\|^2} \langle\Omega \phi, \phi\rangle
-\frac{1}{\|\phi\|^2} \langle\nabla \phi, \nabla \phi\rangle=
\\
&=
\frac{\langle\Omega \phi, \phi\rangle}{\|\phi\|^2}+\frac{\langle\nabla \phi, \phi\rangle \cdot\langle\phi, \nabla \phi\rangle-\|\phi\|^2 \cdot\langle\nabla \phi, \nabla \phi\rangle}{\|\phi\|^4}.
\end{align*}
\end{proof}

\begin{remark}
    For a positive-definite metric, the second term is a seminegative imaginary $(1,1)$-form by the Cauchy--Schwarz inequality, with the conventions below. The first term requires a separate curvature calculation; the Hodge-bundle application is made in the proof of Theorem~\ref{fixed_part_thm}. 
\end{remark}

\begin{lemma}\label{delbardelnorm_lemma}
For a holomorphic section $\phi$, the following identity holds everywhere, including at its zeros:
$$
\delbar\del\|\phi\|^2=d\langle\nabla\phi,\phi\rangle=\langle\Omega \phi, \phi\rangle-
\langle\nabla \phi, \nabla \phi\rangle.
$$
\end{lemma}

\subsection{Curvature of Quotients and Subbundles}

Given a short exact sequence of holomorphic vector bundles
$$
0 \;\to\; S \;\hookrightarrow\; E \;\overset{\pi}{\twoheadrightarrow}\; Q \;\to\; 0,
$$
suppose $E$ carries a non-degenerate hermitian form $\langle\cdot,\cdot\rangle$ (possibly indefinite).  Assume $\langle\cdot,\cdot\rangle|_{S}$ is also non-degenerate, so that $S\cap S^\perp=\{0\}$ and $S^\perp\simeq Q$.

\begin{definition}
    The \emph{second fundamental form} is a tensor 
$B \in \operatorname{Hom}(S, Q) \otimes \cA^{1,0}$ such that 

$$
B: S \xrightarrow{\nabla} E \otimes \cA^{1,0} \xrightarrow{\pi \otimes \mathbf{id}} Q \otimes \cA^{1,0}.
$$
More specifically, $B: S \rightarrow Q \otimes \mathcal{A}^{1,0}$ is a linear tensor, i.e. $B(f \cdot s)=f \cdot B(s)$ for any function $f$ and any section $s$ of $S$. 

We can define the \emph{adjoint} $B^{\dagger} \in \operatorname{Hom}(Q, S) \otimes \cA^{0,1}$ by taking the adjoint in the Hom-space using the hermitian form and complex conjugation on the second factor.
\end{definition}

\begin{lemma}
\label{curvlemma}
The subbundle $S$ and the quotient $Q$ have curvatures
$$
\Omega_S \;=\; \operatorname{pr}_S\circ\Omega_E\big|_S + B^\dagger \wedge B,
\quad
\Omega_Q \;=\; \operatorname{pr}_{S^\perp}\circ\Omega_E\big|_{S^\perp} + B\wedge B^\dagger.
$$
Here $Q$ is identified smoothly with $S^\perp$, and the projections are for the orthogonal splitting $E=S\oplus S^\perp$. Thus the curvature terms on the right are the diagonal blocks of $\Omega_E$, not unprojected restrictions.
\end{lemma}

\begin{proof}
In a local frame, write the connection matrix of $E$ as
$$
A_E
\;=\;
\begin{pmatrix}
A_S & -B^\dagger \\[3pt]
B & A_Q
\end{pmatrix}.
$$
Since $\Omega_E = dA_E + A_E\wedge A_E$, we can compute
$$
\Omega_E
\;=\;
\begin{pmatrix}
\Omega_S - B^\dagger \wedge B & * \\[3pt]
* & \Omega_Q - B\wedge B^\dagger
\end{pmatrix}.
$$
Comparing the block components gives the desired identities.
\end{proof}

\subsection{Conventions on Positivity}

We would like to adopt the following conventions on which forms are considered positive and negative.

\begin{definition}
A purely imaginary $(1,1)$-form $\omega$ is \emph{positive} if it can be written as
$$
\omega
\;=\;
\sum h_{ij}\,dz^i\wedge \overline{dz^j},
$$
with the coefficients $h_{ij}$ forming a positive-definite Hermitian matrix. Replacing positive-definite by positive-semidefinite defines semipositivity.  For example, $dz\wedge \overline{dz}$ is positive.

A purely real $(1,1)$-form $\omega$ is \emph{positive} if $-\sqrt{-1}\,\omega$ is a positive imaginary form. For example, $dx\wedge dy$ is positive because $dx\wedge dy=\frac{\sqrt{-1}}{2}\,dz\wedge \overline{dz}$.

For $\Omega\in\cA^{1,1}\otimes\operatorname{End}(E)$, positivity means that $\langle\Omega e,e\rangle$ is positive at each point for every nonzero fiber vector $e$. Semipositivity is defined by replacing positive by semipositive. These conventions are applied when the scalar forms have the indicated reality type.
\end{definition}

\begin{remark}
Equivalently, an imaginary $(1,1)$-form $\omega$ is semipositive if $\omega(\xi,\overline{\xi})\geq0$ for every $\xi\in T^{1,0}$. It is positive if this inequality is strict for every nonzero $\xi\in T^{1,0}$.
\end{remark}

\begin{lemma}
In the setup of the previous section, assume that the hermitian metric on $E$ is positive definite.  Then $B\wedge B^\dagger$ is semipositive and $B^\dagger \wedge B$ is seminegative.
\end{lemma}

\begin{proof}
Recall that $B\in \operatorname{Hom}(S,Q)\otimes\cA^{1,0}$ and $B^\dagger\in \operatorname{Hom}(Q,S)\otimes\cA^{0,1}$.  For a local section $q\in Q$,
$$
\langle B\wedge B^\dagger q,q\rangle
\;=\;
-\,\langle B^\dagger q, B^\dagger q\rangle,
$$
which is a positive form of the indicated type.  Similarly,
$$
\langle B^\dagger\wedge B s, s\rangle
\;=\;
-\,\langle B s, B s\rangle,
$$
showing that $B^\dagger\wedge B$ is seminegative.
\end{proof}

\subsection{Curvature of Hodge Bundles}

Let us return to a polarized variation of Hodge structures of weight $w$. Write $I$ for its real bilinear polarization and use Hermitian metrics linear in their first argument. The flat indefinite Hermitian metric and the positive Hodge metric are
\[
 h_i(u,v)=i^{-w}I(u,\overline v),
 \qquad
 h(u,v)=I(Cu,\overline v)=i^w h_i(Cu,v).
\]
The first metric is flat for $\nabla^{GM}$. On $\rmH^{p,w-p}$ one has $h_i=(-1)^p h$, so the Chern connections and curvatures for these two restricted metrics coincide. The smooth Hodge summands are identified with the holomorphic quotient bundles $\rmF^p/\rmF^{p+1}$; they are not generally holomorphic subbundles of the flat holomorphic bundle.

\begin{remark}
The bundle $\rmH_\bbC$ carries two different holomorphic structures and metrics. On the one hand, we have the flat structure (inducing a holomorphic one) and indefinite metric. On the other, we have a direct sum of holomorphic bundles, the $\rmH^{p, w-p}$, each equipped with a definite metric.
\end{remark}

The Gauss--Manin connection is the Chern connection of $(\rmH_\bbC,\,\langle\cdot,\cdot\rangle_i)$ when viewed with its flat complex structure.  Alternatively, one can regard $\rmH_\bbC$ as $\bigoplus \rmH^{p,w-p}$ with the definite metric and define the Hodge connection $\nabla^{Hg}$ as a direct sum of Chern connections.

Let $\rmF^p\subset \rmH_\bbC$ be the Hodge filtration. Since $\rmF^p$ is a holomorphic subbundle and
$\nabla^{GM}$ is the flat connection on $\rmH_\bbC$, the \emph{second fundamental form} of $\rmF^p$
is the bundle map
$$
\sigma_p:\rmF^p\longrightarrow (\rmH_\bbC/\rmF^p)\otimes \cA^{1,0}
$$
defined as follows: for a local smooth section $s$ of $\rmF^p$,
$$
\sigma_p(s):=\bigl[\nabla^{GM}s\bigr]\mod \rmF^p.
$$
Equivalently,
$$
\sigma_p=\operatorname{pr}_{\rmH_\bbC/\rmF^p}\circ \nabla^{GM}\big|_{\rmF^p}.
$$
Because $\rmF^p$ is holomorphic, the $(0,1)$-part of $\nabla^{GM}s$ remains in $\rmF^p$, so the quotient above is of type $(1,0)$.
By Griffiths transversality, the image actually lies in
$$
\rmF^{p-1}/\rmF^p\otimes \cA^{1,0}\cong \rmH^{p-1,w-p+1}\otimes \cA^{1,0}.
$$

Griffiths transversality also gives
$\nabla^{GM}(\rmF^{p+1})\subset\rmF^p\otimes\cA^{1,0}+\rmF^{p+1}\otimes\cA^{0,1}$.
Consequently $\sigma_p$ vanishes on $\rmF^{p+1}$ and induces, with the same notation, the tensor
\[
 \sigma_p:\rmH^{p,w-p}=\rmF^p/\rmF^{p+1}
 \longrightarrow\rmH^{p-1,w-p+1}\otimes\cA^{1,0}.
\]
It is these induced tensors on the Hodge summands that are summed below.
Denote the adjoint of $\sigma_p$ for $h_i$ by $\sigma_p^{\dagger}$, and its adjoint for the positive metric $h$ by $\sigma_p^{*h}$. Adjacent Hodge summands have opposite signs for $h_i$, so $\sigma_p^{\dagger}=-\sigma_p^{*h}$. We have 
\begin{theorem}[See \cite{PerBook2017}, Theorem 13.1.7]\label{nabla_relation_thm}
$$
\nabla^{GM}=\nabla^{Hg}+\sigma_{\bullet}-\sigma_{\bullet}^{\dagger}
=\nabla^{Hg}+\sigma_{\bullet}+\sigma_{\bullet}^{*h}.
$$
Here, $\sigma_{\bullet}$ denotes $\oplus_p \sigma_p$. 
\end{theorem}
Note that the curvature of $\rmH^{p, w-p}$ for either metric is the same, since the metrics agree up to a sign.

\begin{lemma}[See \cite{PerBook2017},(13.1)]
\label{curvature_lemma_1}
The curvature of each Hodge bundle $\rmH^{p,w-p}$ can be written as
$$
\Omega_{\rmH^{p,w-p}}
\;=\;
\sigma_p^\dagger\wedge\sigma_p
\;+\;
\sigma_{p+1}\wedge\sigma_{p+1}^\dagger.
$$
\end{lemma}
A convenient way to view the above is via sections:

\begin{lemma}\label{curvature_lemma_2}
For smooth sections $e,e'$ of $\rmH^{p,w-p}$, the positive Hodge metric gives
\[
 \langle\Omega_{\rmH^{p,w-p}}e,e'\rangle
 =\langle\sigma_p e,\sigma_p e'\rangle
 +\langle\sigma_{p+1}^\dagger e,\sigma_{p+1}^\dagger e'\rangle.
\]
In particular, if $\sigma_p e=0$, then $\langle\Omega_{\rmH^{p,w-p}}e,e\rangle$ is seminegative.
\end{lemma}

\begin{proof}
Pair the curvature identity with $e'$ and use $\sigma_p^\dagger=-\sigma_p^{*h}$, accounting for the wedge signs. The first term has type $(1,0)$ paired with itself and is semipositive, while the second has type $(0,1)$ paired with itself and is seminegative.
\end{proof}

\begin{corollary}
If $p_{\min}$ is the lowest nonzero Hodge index, then $\rmH^{p_{\min},w-p_{\min}}$ is seminegatively curved. 
\end{corollary}

\begin{proof}
    The map $\sigma_{p_{\min}}$ vanishes because its target Hodge summand is zero. The remaining term in the curvature formula is seminegative. In particular this applies to $\rmH^{0,w}$ when $0$ is the lowest nonzero Hodge index.
\end{proof}

\section{Random walks on groups}\label{sect4}
Throughout this section, we are working with a probability space $(X,\,\mu)$ and a locally compact group $G$ with a compactly supported probability measure $\nu$ such that $\textrm{supp}(\nu)$ generates $G$. We assume that the left action of $G$ on $X$ is measurable, measure-preserving and ergodic.

All results presented in this section are well-known and are included for the sake of convenience.

\subsection{Random Ergodic Theorem}
\label{RETheorem}
\begin{theorem}[{\cite{Furman2002}, Theorem 3.1}]
Let $f\in L^1(X,\mu)$. Then for $\mu\times \nu^\bbN$-a.e.\ $(x,\omega)\in X\times G^\bbN$,
$$
\lim_{N\to\infty}\frac1N\sum_{i=0}^{N-1} f\big(g_i(\omega)\cdots g_1(\omega)\,x\big)=\int_X f\,d\mu.
$$
\end{theorem}

\subsection{Harmonic and subharmonic functions}

\begin{definition}\label{harmonic_definitions}
Consider a $\mu$-measurable function $h: X \rightarrow \bbR$. The averaging operator $Av_G(h)$ is 
$$
Av_G(h)(x)=(\nu * h)(x)=\int_G h(g x) d \nu(g).
$$
The function $h$ is called \emph{$\nu$-harmonic} if, for $\mu$-a.e. $x \in X$,
$$
h(x)=Av_G(h)(x).
$$
The function $h$ is called $\nu$-subharmonic if, for $\mu$-a.e. $x \in X$,
$$
h(x)\leq Av_G(h)(x).
$$
The \emph{Laplacian} is
$$
L h = Av_G(h)-h.
$$
\end{definition}
Note that $Av_G(h)(x)$ is not automatically well-defined.

\begin{lemma}\label{harm_const_lemma_1}
\begin{enumerate}
\item If $f\in L^1(X,\mu)$ is $\nu$-subharmonic, then $f$ is constant $\mu$-almost everywhere.
\item If $h:X\to \bbR_{\ge 0}$ is measurable, finite almost everywhere and $\nu$-harmonic, then $h$ is constant $\mu$-almost everywhere. No integrability of $h$ is required.
\end{enumerate}
\end{lemma}

\begin{proof}
Write $P=Av_G$. Stationarity gives $\int Pv\,d\mu=\int v\,d\mu$ for every $v\in L^1(X,\mu)$.
\begin{enumerate}
\item Integrating $Pf-f\geq0$ gives $Pf=f$ almost everywhere. The strictly convex function $\Phi(t)=\sqrt{1+t^2}$ satisfies $\Phi(f)\in L^1$, and Jensen's inequality gives
\[
 \Phi(f)=\Phi(Pf)\leq P\Phi(f).
\]
Stationarity forces equality. The equality case in Jensen's inequality gives $f(gx)=f(x)$ for $\mu\times\nu$-almost every $(x,g)$. Iteration along the random walk and Theorem~\ref{RETheorem} give $f(x)=\int_Xf\,d\mu$ almost everywhere.

\item The bounded strictly concave function $\varphi(t)=t/(1+t)$ on $[0,\infty)$ gives
\[
 P\varphi(h)\leq\varphi(Ph)=\varphi(h).
\]
Again stationarity forces equality, so $h(gx)=h(x)$ for $\mu\times\nu$-almost every $(x,g)$. Apply Theorem~\ref{RETheorem} along the random walk to the bounded function $\varphi(h)$. It is constant almost everywhere, and injectivity of $\varphi$ proves the assertion for $h$.
\end{enumerate}
\end{proof}

\subsection{Geodesic approximation}
This section briefly explains an important variation of the multiplicative ergodic theorem proved by Kaimanovich.

\begin{definition}
Let $G$ be a connected real semisimple Lie group with finite center, and let $K$ be a maximal compact subgroup. Set $Y:=G/K$.
A sequence $\{y_n\}\subset Y$ is called \emph{regular} if there exist a geodesic ray
$\gamma:[0,\infty)\to Y$ and $C\ge 0$ such that
$$
d\big(y_n,\gamma(Cn)\big)=o(n).
$$
\end{definition}

\begin{theorem}[\cite{Kaimanovich1989}; see \cite{Filip2019} for more details]
\label{Kaimanovich_th}
Let $G$, $K$ and $Y$ be as above. Fix a Cartan decomposition $\mathfrak g=\mathfrak k\oplus\mathfrak p$, a maximal abelian subspace $\mathfrak a\subset\mathfrak p$, and a closed positive Weyl chamber $\mathfrak a^+$. For $y\in Y$, its Cartan projection $r(y)\in\mathfrak a^+$ is the unique element such that $y=k\exp(r(y))K$ for some $k\in K$. A sequence $\{y_n\}$ in $Y$ is regular if and only if the following two conditions are satisfied:
\begin{enumerate}
\item Small steps: $\quad \mathrm{d}\left(y_n,\,y_{n+1}\right)=o(n)$
\item Distances converge: $\alpha:=\lim \frac{r\left(y_n\right)}{n}$ exists in $\mathfrak{a}^{+}$.
\end{enumerate}
\end{theorem}

\section{Semisimplicity}\label{sect5}

\subsection{Standing assumptions and cocycle bounds}
\label{subsec:semisimplicity-setup}

Throughout this section, we work with a projection-surjective group $G\subset(\SLgr)^n$ acting on the product of strata
$$
X=\cH^{(1)}(\alpha_1,\ldots,\alpha_n) =\cH^{(1)}(\alpha_1)\times\cdots\times\cH^{(1)}(\alpha_n).
$$
We make the following additional assumption on $G$. Fix a partition $\{1,\ldots,n\}=I_1\sqcup\cdots\sqcup I_r$, and let $D_j\cong\SLgr$ act diagonally on the factors indexed by $I_j$ and trivially on the remaining factors. We assume that, up to conjugation by coordinatewise rotations,
$$
G=(D_1\cdots D_r)F,
\qquad F\leq\{\pm\Id\}^n.
$$
We work in the coordinates furnished by this rotational conjugation. We do not allow arbitrary coordinatewise conjugations: although they give equally valid group-theoretic embeddings of the diagonal factors, they need not preserve the canonical Hodge structures on the cohomology bundles. This distinction is essential for the Hodge-theoretic arguments below; see Subsection~\ref{subsec:arbitrary-conjugations}.

Let $\mu$ be an ergodic $G$-invariant probability measure on $X$. Let $\Gamma$ denote the Kontsevich--Zorich cocycle on
$$
\rmH_\bbF
=\bigoplus_{i=1}^n\pi_i^*\rmH^1_\bbF
   \bigl(\cH^{(1)}(\alpha_i)\bigr),
\qquad \bbF=\bbR\text{ or }\bbC,
$$
or a cocycle obtained from it by a tensor construction, equipped with the induced Hodge norm. We first treat a tensor construction $\rm E$ which is pure of a fixed Hodge weight $w$, so that
$$
\rm E_\mathbb C=\bigoplus_{p+q=w}\rm E^{p,\,q}.
$$
Here $p,q$ are arbitrary integers, not necessarily nonnegative; negative Hodge indices arise naturally when taking duals. For any fixed tensor construction, only finitely many Hodge types occur. Finite direct sums of different weights will be treated in Corollary~\ref{cor:graded-semisimplicity}.

Set $K_j=D_j\cap(SO_2\bbR)^n$ and
$$
K_G=(K_1\cdots K_r)F.
$$
Since $K_G\subset G$, the measure $\mu$ is $K_G$-invariant. The cocycle over $K_G$ preserves the Hodge structures and the Hodge norm. The canonical variations of Hodge structure restrict to the orbit polydiscs, whose universal parameter spaces are
$$
K_G\backslash G\cong\mathbb H^r.
$$
In particular, each individual block has parameter space $K_j\backslash D_j\cong\mathbb H$.

We will need the following estimate. For $g=(g_1,\ldots,g_n)\in(\SLgr)^n$, set
$$
\|g\|:=\max_i\|g_i\|_{\mathrm{op}},
$$
where $\|\cdot\|_{\mathrm{op}}$ denotes the Euclidean operator norm on $2\times2$ matrices.

\begin{lemma}
There exist constants $C_1,C_2\geq0$ such that
\begin{equation}\label{boundedness_eq}
\log\|\Gamma(g,x)\|\leq C_1\log\|g\|+C_2
\end{equation}
for all $g\in G$ and $x\in X$.
\end{lemma}

\begin{proof}
We first treat the Kontsevich--Zorich cocycle $\mathrm{KZ}$ on one stratum, equipped with the Hodge norm. Let $a_t=\operatorname{diag}(e^t,e^{-t})$. By \cite[Lem.~2.3]{ForniMatheusZorich2014LyapunovSpectrum}, there exist constants $c_1,c_2$ such that
\begin{equation}\label{eq:KZ-at}
\log\|\mathrm{KZ}(a_t,x)\|\leq c_1|t|+c_2.
\end{equation}
Write an arbitrary $g\in\SLgr$ in Cartan form $g=k_1a_tk_2$, with $k_1,k_2\in SO_2\bbR$ and $t\geq0$. Since rotations act isometrically for the Hodge norm, \eqref{eq:KZ-at} implies
$$
\log\|\mathrm{KZ}(g,x)\|\leq c_1t+c_2.
$$
For the Euclidean operator norm, $\|g\|=\|a_t\|=e^t$, which proves the estimate on one stratum. Other fixed matrix norms give the same estimate after changing the constants.

For a product of strata, apply the estimate in each factor. Tensor products only change the constants; duals are handled by applying the same estimate to the inverse cocycle, using $\|g^{-1}\|=\|g\|$ in each $\SLgr$ factor. Restriction to $G$ preserves the estimate.
\end{proof}

\subsection{A one-disc subharmonicity lemma}
\label{subsec:one-disc-subharmonicity}

Only one simultaneous copy of $\SLgr$ is needed. Put
\[
 H=\{(g,\ldots,g):g\in\SLgr\}\subset D_1\cdots D_r,
 \qquad K_H=H\cap K_G,
\]
and write $A_t=(a_t,\ldots,a_t)\in H$.

\begin{lemma}[Simultaneous-flow invariants]\label{lemma:mautner-blocks}
Put $D=D_1\cdots D_r$. In any strongly continuous unitary representation of $D$, the subspaces fixed by $A_1$, by $H$, and by $D$ coincide. Consequently, for the Koopman representation on $L^2(X,\mu)$, the $H$-ergodic and $D$-ergodic decompositions coincide. If $\mu$ is $DF$-ergodic, they consist of finitely many components, permuted transitively by $F$; on each component both $A_t$ and $A_1$ are ergodic.
\end{lemma}

\begin{proof}
If a vector is fixed by $A_1$, it is fixed by every $A_n$, $n\in\bbZ$. Conjugation by these elements contracts the upper and lower unipotent subgroups in each $D_\beta$ in the appropriate time direction. Mautner's lemma \cite[Lemma~1.4.8]{BekkaDeLaHarpeValette2008} therefore makes the vector fixed by these subgroups, which generate $D$. The reverse inclusions follow from $A_1\in H\subset D$.

Applied to indicator functions, this identifies the invariant sigma-algebras and hence the ergodic decompositions. The finite group $G/D$ acts ergodically on the space of $D$-ergodic components, so that space is one finite orbit. Applying the same fixed-vector argument on each component proves the last assertion.
\end{proof}

\begin{lemma}\label{lemma:simultaneous-generation}
The subgroups $H$ and $K_G$ generate $G$.
\end{lemma}

\begin{proof}
Fix a block $I_\beta$, and let $k_\beta\in K_\beta$ act by rotation through $\pi/2$ on this block and trivially on all other blocks. Since this rotation conjugates $a_{-t}$ to $a_t$, we have
\[
A_tk_\beta A_{-t}k_\beta^{-1}=a_{2t}^{(\beta)},
\]
where $a_s^{(\beta)}$ denotes the geodesic flow acting only on the block $I_\beta$. Indeed, the factors $A_t$ and $A_{-t}$ cancel outside $I_\beta$, while on $I_\beta$ their product with the rotation is $a_ta_t=a_{2t}$. Thus $\langle H,K_G\rangle$ contains the geodesic flow in every block. Since it also contains $K_\beta$, the Cartan decomposition $D_\beta=K_\beta A_\beta K_\beta$ shows that it contains all of $D_\beta$. Finally, $F\subseteq K_G$, so $\langle H,K_G\rangle=G$.
\end{proof}

The variation restricts to the simultaneous orbit discs $K_H\backslash H\cong\mathbb H$. 
By Lemma~\ref{lemma:mautner-blocks}, it suffices to work on the finitely many $D$-ergodic components of $\mu$, each of which is also $H$-ergodic. We apply the following lemma on each component.

\begin{lemma}\label{lemma:subharmonic_randomwalk}
Let $\eta$ be an $H$-invariant ergodic probability measure on $X$. Let $f$ be an $\eta$-measurable, almost everywhere finite, $K_H$-invariant function whose restriction to almost every orbit disc is subharmonic:
$
i\,\partial\bar\partial f\geq0
$ 
in the sense of currents. We allow subharmonic functions with logarithmic poles, so in particular $f$ may take the value $-\infty$ at zeros of the underlying holomorphic section.

Suppose that there is a nonnegative, finite almost everywhere, $K_H$-invariant measurable function $u$ such that $f^+\leq u$ and:
\begin{enumerate}
\item For $\eta$-almost every $x\in X$ and Lebesgue-almost every $\theta\in[0,2\pi)$,
\[
\lim_{t\to\infty}\frac{u(a_t r_\theta x)}{t}=0,
\]
where $r_\theta\in K_H$ denotes rotation through $\theta$ in the simultaneous subgroup;

\item there are constants $B_1,B_2\geq0$ such that
$$
|u(gx)-u(x)|\leq B_1\log\|g\|+B_2
$$
for $\eta$-almost every $x$ and all $g\in H$.
\end{enumerate}
Then $f$ is $\eta$-almost everywhere constant.
\end{lemma}

\begin{proof}
Choose a compactly supported $K_H$-biinvariant probability measure $\nu$ on $H$ with a smooth density and generating support, and put
$$
Pv(x)=\int_Hv(gx)\,d\nu(g).
$$
Since $\nu$ is $K_H$-biinvariant, its pushforward to $K_H\backslash H\cong\mathbb H$ is a radial probability measure centered at $K_H$. The submean inequality for subharmonic functions, first on each hyperbolic circle and then averaged over the radial variable, therefore gives
$$
Pf(x)=\int_H f(gx)\,d\nu(g)\ge f(x).
$$
Thus a function subharmonic on the orbit discs is $\nu$-subharmonic.

Write $w_N=g_N\cdots g_1$ for the random walk. 
The inversion map
$$
K_H\backslash H\longrightarrow H/K_H,
\qquad
K_Hg\longmapsto g^{-1}K_H,
$$
is an isometry.
Thus $K_Hw_N$ corresponds to
$$
w_N^{-1}K_H
=
g_1^{-1}\cdots g_N^{-1}K_H,
$$
which is the usual forward random walk on $H/K_H$ with increment distribution $\check\nu$, the image of $\nu$ under inversion.
Since $\nu$ is compactly supported, this walk has bounded steps. The rank-one random-walk statement in Theorem~\ref{Kaimanovich_th} gives a deterministic $\delta>0$ such that
$$
\frac{d(K_H,w_N^{-1}K_H)}{N}\longrightarrow\delta
$$
almost surely. Therefore Theorem~\ref{Kaimanovich_th}, applied on $H/K_H$ and transported back by inversion, gives an almost surely defined geodesic ray $\gamma$ in $K_H\backslash H$ such that
$$
d(K_Hw_N,\gamma(\delta N))=o(N).
$$
This is the rank-one random-walk argument used in \cite[Lem.~6.2]{Filip2013}. Since $\nu$ is $K_H$-biinvariant, so is $\check\nu$, and the distribution of the limiting geodesic direction is the normalized angular measure. 
Moreover, the bound in \textup{(2)} controls the change of $u$ between two points by a constant times their hyperbolic distance, plus another constant. Combining this with \textup{(1)} and the tracking estimate gives
$$
u(w_Nx)=o(N)
$$
for $\eta$-almost every $x$ and almost every random walk.

Fix an integer $M\geq1$, and set
$$
h_M=\max(f,-M)+M.
$$
Then $0\leq h_M\leq u+M$, and $Ph_M\geq h_M$. If $R=\sup_{g\in\operatorname{supp}\nu}\log\|g\|$, then
$$
0\leq\frac{h_M(w_Nx)}{N}
\leq\frac{u(x)+M+B_2}{N}+B_1R.
$$
Consequently, dominated convergence, for each fixed typical $x$, gives
$$
\frac{P^Nh_M(x)}{N}\longrightarrow0.
$$
Let $b_M=Ph_M-h_M\geq0$. Telescoping gives
$$
\frac1N\sum_{k=0}^{N-1}P^kb_M(x)
=\frac{P^Nh_M(x)-h_M(x)}{N}\longrightarrow0.
$$
For each $L>0$, apply the random ergodic theorem to the bounded function $\min(b_M,L)$, and then apply Fatou's lemma in the random-walk variable. This yields
$$
\int_X\min(b_M,L)\,d\eta
\leq\liminf_{N\to\infty}
\frac1N\sum_{k=0}^{N-1}P^k\min(b_M,L)(x)=0.
$$
It follows that $b_M=0$ almost everywhere. Thus $h_M$ is finite, nonnegative and $\nu$-harmonic. Lemma~\ref{harm_const_lemma_1}(2), applied with $H$ and $\eta$, makes it constant almost everywhere, without an integrability assumption. Letting $M\to\infty$ proves the assertion for $f$.
\end{proof}

\subsection{The fixed-part theorem}
\label{subsec:fixed-part}

\begin{theorem}\label{fixed_part_thm}
Keep the standing assumptions on $G$ and $\mu$. Let $\rmE_\bbC$ be a finite-rank measurable $G$-cocycle which induces a pure polarized variation of real or complex Hodge structures of weight $w$ on almost every orbit polydisc. Assume that its orbitwise flat connection induces the cocycle, that its polarization is flat, that $K_G$ preserves its Hodge decomposition and Hodge metric, and that the cocycle satisfies \eqref{boundedness_eq} in that metric. The variation and its Hodge data are assumed measurable over $X$, which is automatic for the Kontsevich--Zorich bundle and its finite tensor constructions.

If a measurable section $\phi$ of $\rmE_\bbC$ is $G$-equivariant, meaning
\[
 \phi(gx)=\Gamma(g,x)\phi(x)
 \qquad\text{for every }g\in G\text{ and }\mu\text{-almost every }x,
\]
then every $(p,w-p)$-component of $\phi$ is $G$-equivariant as well. In particular, these sections are flat along the connected orbit directions; equivariance also includes compatibility with the finite component group.

In particular, this applies to the Hodge bundle and its pure tensor constructions. It also applies to pure polarized subvariations with the induced metric, their duals and Hom-variations, and tensor products with constant polarized Hodge structures, whenever they are formed from these bundles.
\end{theorem}
\begin{proof}
Let $H$ be the simultaneous subgroup above, and fix one of the finitely many $D$-ergodic components $\eta$ of $\mu$. By Lemma~\ref{lemma:mautner-blocks}, both $H$ and $A_1$ act ergodically on this component. We use the proof of the fixed-part theorem in \cite{Filip2013}: its inputs here are the polarized variation on $K_H\backslash H$, the $K_H$-compatible Hodge metric and the cocycle bound \eqref{boundedness_eq}. No one-stratum description of the base is needed. We include the argument to make this applicability explicit. If $\phi=0$ almost everywhere on this component, there is nothing to prove; otherwise $\phi$ is nonzero almost everywhere.

We will prove the assertion by inductive exhaustion. Let $p$ be the largest integer such that $\phi$ is a section of $\rmF^p\rmE$ along almost every disc. Equivalently, $p$ is the smallest generically nonzero Hodge index of $\phi$. This index is almost everywhere constant on an ergodic component: the condition that a flat section lies identically in a given filtration step along its disc is $H$-invariant, and a holomorphic quotient section which is not identically zero has only discrete zeros. We can write
$$
\phi=\sum_{q\geq p}\phi^{q,\,w-q}.
$$
If we prove the flatness of the lowest nonzero component, we can subtract it from $\phi$ and repeat until no components remain.

Denote the lowest nonzero Hodge component of $\phi$ by $\psi=\phi^{p,\,w-p}$. Recall that the associated graded bundle of the Hodge filtration is
\[
\operatorname{Gr}_F^p\rmE:=F^p\rmE/F^{p+1}\rmE\cong\rmE^{p,\,w-p},
\]
which is naturally identified, fiberwise, with $\rmE^{p,\,w-p}$. Since $\phi$ is Gauss--Manin flat, it is holomorphic for the flat holomorphic structure. Moreover, $\phi$ takes values in the holomorphic subbundle $F^p\rmE$, and the quotient map $F^p\rmE\to\operatorname{Gr}_F^p\rmE$ is holomorphic. Thus $\psi$, viewed as a section of $\operatorname{Gr}_F^p\rmE$, is holomorphic along the orbit discs.

Applying Formula~\ref{deldelbarlog_eq}, with the order of $\delbar\del$ switched, gives, away from the zeros of $\psi$,
\begin{align*}
\del\delbar\log\|\psi\|^2 ={}&-\frac{\langle\Omega_{\rmE^{p,w-p}}\psi,\,\psi\rangle}
             {\|\psi\|^2}\\
&+\frac{
\|\psi\|^2\langle\nabla^{Hg}\psi,\,\nabla^{Hg}\psi\rangle -\langle\nabla^{Hg}\psi,\,\psi\rangle
 \langle\psi,\,\nabla^{Hg}\psi\rangle}
{\|\psi\|^4}.
\end{align*}
Here $\nabla^{Hg}$ is the Chern connection of $\rmE^{p,\,w-p}$. After multiplication by $i$, the second term is nonnegative by Cauchy--Schwarz. Since $\sigma_p\psi=0$, Lemma~\ref{curvature_lemma_2} identifies the first term as
$$
-\frac{\langle\Omega_{\rmE^{p,w-p}}\psi,\,\psi\rangle} {\|\psi\|^2} =-\frac{\langle\sigma_{p+1}^{\dagger}\psi,\,
                   \sigma_{p+1}^{\dagger}\psi\rangle}
{\|\psi\|^2}.
$$
After multiplication by $i$, this term is also nonnegative. Thus $f=\log\|\psi\|^2$ is subharmonic on the orbit discs, including its logarithmic singularities at zeros of $\psi$. The zeros have $\eta$-measure zero, by invariance of $\eta$ and discreteness of the zeros on almost every disc.

We check the hypotheses of Lemma~\ref{lemma:subharmonic_randomwalk} using
$$
u=(\log\|\phi\|^2)^+.
$$
Hodge orthogonality gives $f^+\leq u$. For almost every initial point, Oseledets' theorem gives the limit
$$
\lambda_\phi
=
\lim_{t\to\infty}
\frac1t\log\|\phi(A_t x)\|.
$$
Choose a positive-measure set on which $\|\phi\|$ is bounded both above and away from zero. By Poincaré recurrence, almost every point of this set returns to it along a sequence $t_k\to\infty$, and hence
$$
\frac1{t_k}\log\|\phi(A_{t_k}x)\|\to0.
$$
Since the Oseledets limit exists, it follows that $\lambda_\phi=0$. Ergodicity of $A_1$ makes this conclusion valid almost everywhere on the component. Then $K_H$-invariance of $\eta$ and Fubini in the angular variable give the same conclusion for almost every initial direction. Hence $u$ grows sublinearly along almost every geodesic. This is the argument of \cite[Lem.~6.1]{Filip2013}.

Finally, invariance of $\phi$, the estimate \eqref{boundedness_eq} for $g$ and $g^{-1}$, and $\|g^{-1}\|=\|g\|$ give
$$
|u(gx)-u(x)|\leq2C_1\log\|g\|+2C_2
\qquad(g\in H).
$$
Both $f$ and $u$ are $K_H$-invariant, because rotations preserve the Hodge decomposition and norm. The lemma therefore shows that $\log\|\psi\|^2$ is $\eta$-almost everywhere constant. Leafwise continuity and invariance of $\eta$ imply that $\|\psi\|$ is constant on almost every orbit disc. This constant is positive, since $\psi$ is not identically zero; continuity then rules out isolated zeros. Hence $\psi$ is nowhere vanishing on those discs.

It follows that $\delbar\del\|\psi\|^2=0$ there. By Lemma~\ref{delbardelnorm_lemma},
$$
0=\delbar\del\|\psi\|^2 =\langle\Omega_{\rmE^{p,\,w-p}}\psi,\,\psi\rangle -\langle\nabla^{Hg}\psi,\,\nabla^{Hg}\psi\rangle.
$$
After multiplication by $i$, both terms on the right are nonpositive, so both vanish. Thus $\sigma_{p+1}^{\dagger}\psi=0$ and $\nabla^{Hg}\psi=0$ along the discs. Together with $\sigma_p\psi=0$, Theorem~\ref{nabla_relation_thm} shows that $\psi$ is Gauss--Manin flat there. It is also equivariant under $K_H$, so it is flat along the $H$-orbits. Subtracting $\psi$ and repeating proves the assertion for all Hodge components.

Applying the argument on each $D$-ergodic component makes every Hodge component of $\phi$ $H$-equivariant over $(X,\mu)$. Each component is already $K_G$-equivariant, since $\phi$ is $G$-equivariant and $K_G$ preserves the Hodge decomposition. Lemma~\ref{lemma:simultaneous-generation} therefore gives $G$-equivariance, as required.
\end{proof}

\subsection{Invariant subbundles and the Deligne torus}
\label{subsec:deligne-torus}

Denote by $\rmE_\bbF$ the Hodge bundle or a tensor construction from it. Let $a$ be an element of the Deligne torus acting on $\rmE_\bbC$.

\begin{theorem}\label{deligne-torus-thm}
Suppose that $V\subset\rmE_\bbC$ is a $G$-invariant subbundle. Then $a\cdot V$ is also $G$-invariant.
\end{theorem}

\begin{proof}
By Lemma~\ref{main-reductive-lemma}, the cocycle on $\rmE_\bbC$ has reductive algebraic hull. Thus $V$ has a $G$-invariant complement $W$. The projection $\pi_V$ onto $V$ along $W$ is a $G$-invariant section of $\operatorname{End}(\rmE_\bbC)$. By Theorem~\ref{fixed_part_thm}, each of its Hodge components is $G$-invariant. Since $\operatorname{End}(E)=E\otimes E^\vee$, functoriality of the Deligne-torus action gives
$$
a\cdot T=aTa^{-1}
$$
on $\operatorname{End}(E)$. Hence $a\pi_Va^{-1}$ is $G$-invariant and is the projection onto $aV$ along $aW$. Therefore, $aV$ is $G$-invariant.
\end{proof}

\subsection{Deligne semisimplicity}
\label{subsec:deligne-semisimplicity}
We now state the core result of this section. For $\bbF=\bbC$, Hodge structures and variations are understood in the complex sense.

\begin{theorem}[Deligne semisimplicity]\label{thm__Deligne-semisimplicity}
Let $\rm E_\bbF$ be the pure tensor construction of Hodge weight $w$ fixed above, where $\bbF=\bbR$ or $\bbC$.
\begin{enumerate}
\item There exist pairwise nonisomorphic irreducible $G$-invariant subbundles $V_i\subset\rm E_\bbF$ and finite-dimensional vector spaces $W_i$ such that
\begin{equation}\label{E-decomp-eq2}
\rm E_\bbF\cong\bigoplus_i E_i,
\qquad E_i\cong V_i\otimes_{A_i}W_i,
\end{equation}
where $E_i$ is the full isotypic summand corresponding to $V_i$ and
\[
A_i=\operatorname{End}_G(V_i)^{\mathrm{op}}
\]
is a division algebra. Here $V_i$ is a right $A_i$-module and $W_i$ is a left $A_i$-module.

Each $V_i$ can be equipped with a pure variation of Hodge structure of some weight $\nu_i$ along the orbit polydiscs, and each $W_i$ with a constant pure Hodge structure of weight $w-\nu_i$, such that the evaluation isomorphism $V_i\otimes_{A_i}W_i\cong E_i$ is Hodge-compatible. The $A_i$-actions preserve these Hodge structures.

The variation of Hodge structure on $V_i$ is constructed on its underlying invariant bundle; the chosen inclusion $V_i\subset\rm E_\bbF$ need not be a morphism of variations of Hodge structure. In particular, $V_i$ need not itself be a sub-variation of the ambient Hodge structure.

\item Under the identifications in part~(1), every $G$-invariant subbundle $U\subset\rm E_\bbF$ has the form
\[
U=\bigoplus_i V_i\otimes_{A_i}W_i',
\]
where $W_i'\subset W_i$ are $A_i$-submodules. These submodules need not be Hodge substructures of $W_i$, so an invariant subbundle $U$ need not inherit a variation of Hodge structure from $\rm E_\bbF$.

Nevertheless, by Theorem~\ref{deligne-torus-thm}, the image of $U$ under every element of the Deligne torus is again a $G$-invariant subbundle.
\end{enumerate}
\end{theorem}

\begin{proof}
1. By Lemma~\ref{main-reductive-lemma}, the $G$-cocycle is semisimple. Choose a nonzero $G$-invariant subbundle $V$ of smallest possible fiber dimension. Every nonzero $G$-invariant endomorphism of $V$ is an isomorphism, since otherwise its kernel or image would give a proper nonzero invariant subbundle. Thus
$$
A=\operatorname{End}_G(V)^{\mathrm{op}}
$$
is a division algebra. The invariant morphism spaces used below are finite-dimensional by ergodicity.

We view $V$ as a right $A$-module by $v\cdot a=a(v)$, where $a$ is viewed as the underlying endomorphism of $V$, and put
$$
W=\operatorname{Hom}_G(V,\rmE_\bbF).
$$
These morphisms need not respect Hodge structures. We make $W$ a left $A$-module by precomposition:
$$
a\cdot f=f\circ a.
$$
The evaluation map
$$
\operatorname{ev}:V\otimes_AW\longrightarrow\rmE_\bbF,
\qquad \operatorname{ev}(v\otimes f)=f(v),
$$
is well-defined because
$$
\operatorname{ev}((v\cdot a)\otimes f)
=f(a(v))=(f\circ a)(v) =\operatorname{ev}(v\otimes(a\cdot f)).
$$
Denote its image by $\rmE'$. This is the full $V$-isotypic summand: it is the sum of all invariant subbundles isomorphic to $V$. Indeed, decompose $\rmE_\bbF$ into irreducible invariant summands. Schur's lemma shows that $\operatorname{Hom}_G(V,\,U)$ vanishes for every summand $U$ not isomorphic to $V$ and is a free left $A$-module of rank one for each copy of $V$. This also proves that
$$
\operatorname{ev}:V\otimes_AW\longrightarrow\rmE'
$$
is an isomorphism and that
$$
B:=\operatorname{End}_G(\rmE')
\cong\operatorname{End}_A(W).
$$

We first check that $\rmE'$ is a sub-variation of Hodge structures. Theorem~\ref{fixed_part_thm}, applied to $\operatorname{End}(\rmE_\bbF)$, gives the finite-dimensional algebra $\operatorname{End}_G(\rmE_\bbF)$ a Hodge structure compatible with multiplication. The Deligne torus therefore acts on this algebra by algebra automorphisms. It fixes each primitive central idempotent: these form a finite set, whereas the torus is connected. The canonical projection onto the isotypic summand $\rmE'$ is such an idempotent. Hence it is of Hodge type $(0,0)$, and its image $\rmE'$ is a sub-variation along the orbit polydiscs.

Applying Theorem~\ref{fixed_part_thm} to its invariant endomorphisms gives $B$ its induced Hodge structure. Equivalently, $B$ is the corresponding Hodge subalgebra of $\operatorname{End}_G(\rmE_\bbF)$. Transfer this structure through $B\cong\operatorname{End}_A(W)$. Lemma~\ref{end-hodge-lemma} then gives an $A$-compatible Hodge structure on $W$.

It remains to endow $V$ with a variation of Hodge structure making evaluation Hodge-compatible. The algebra $B=\operatorname{End}_A(W)$ acts on $W$ and on $\rmE'\cong V\otimes_AW$, in the latter case through $\Id_V\otimes b$. Identify $V$ with
$$
\operatorname{Hom}_B(W,\rmE') =\{\phi\in\operatorname{Hom}_\bbF(W,\rmE'):
  \phi(bf)=b\phi(f)\text{ for all }b\in B,\ f\in W\}
$$
by sending $v$ to the evaluation map $f\mapsto f(v)$. Both $B$-actions are morphisms of Hodge structures, so this intertwiner subbundle is a sub-variation of $\operatorname{Hom}_\bbF(W,\rmE')$, with $W$ viewed as a constant variation. The evaluation isomorphism is therefore Hodge-compatible, as is the right $A$-action on $V$. 
This construction equips $V$ with a variation of Hodge structure on its underlying flat bundle; it does not require the chosen inclusion $V\subset\rmE_\bbF$ to preserve the ambient Hodge decomposition.
 The Hodge structures on $V$ and $W$ may be shifted simultaneously in opposite directions.

The polarization-orthogonal complement of $\rmE'$ is $G$-invariant and is also its Hodge-orthogonal complement, since $\rmE'$ is a sub-variation. Repeating the construction on this complement proves part~(1). Each isotypic summand is removed in its entirety, so the resulting $V_i$ are pairwise nonisomorphic.

2. Consider a $G$-invariant subbundle $U\subset\rmE_\bbF$. By semisimplicity, $U$ is a direct sum of irreducible invariant subbundles. Each such subbundle is isomorphic to exactly one $V_i$ and hence lies in its full isotypic summand $E_i$. Consequently,
$$
U=\bigoplus_i(U\cap E_i).
$$
For each $i$, let
$$
W_i'=\operatorname{Hom}_G(V_i,U)
\subset\operatorname{Hom}_G(V_i,\rmE_\bbF)=W_i.
$$
This is an $A_i$-submodule, and the same evaluation argument gives
$$
U\cap E_i\cong V_i\otimes_{A_i}W_i'.
$$
This proves part~(2).
\end{proof}

\begin{corollary}[Finite direct sums of different weights]
\label{cor:graded-semisimplicity}
Keep the standing assumptions of this section, and let
\[
 \rmE_\bbF=\bigoplus_{w\in\mathcal W}\rmE_\bbF^{[w]},
 \qquad \bbF=\bbR\text{ or }\bbC,
\]
where $\mathcal W\subset\bbZ$ is finite and every $\rmE^{[w]}$ is a pure tensor construction of Hodge weight $w$. Then there are pairwise nonisomorphic irreducible $G$-bundles $V_i$, each carrying a pure variation of Hodge structure of some weight $\nu_i$ along the orbit polydiscs, and finite-dimensional spaces
\[
 W_i=\bigoplus_m W_i^{[m]},
\]
where $W_i^{[m]}$ is a pure Hodge structure of weight $m$, such that evaluation gives a $G$-equivariant, Hodge-compatible isomorphism
\[
 \rmE_\bbF\cong\bigoplus_i V_i\otimes_{A_i}W_i,
 \qquad A_i=\operatorname{End}_G(V_i)^{\mathrm{op}}.
\]
The $A_i$-actions have Hodge type $(0,0)$, and
\[
 \rmE_\bbF^{[w]}
 \cong\bigoplus_i V_i\otimes_{A_i}W_i^{[w-\nu_i]}.
\]
Every measurable $G$-invariant subbundle $U\subset\rmE_\bbF$ has the form
\[
 U=\bigoplus_i V_i\otimes_{A_i}W_i',
\]
for $A_i$-submodules $W_i'\subset W_i$. These submodules need not respect either the weight or the Hodge decomposition. The bundle $U$ respects the ambient Hodge decomposition if and only if every $W_i'$ is a Hodge substructure of $W_i$.
\end{corollary}

\begin{proof}
Apply Theorem~\ref{thm__Deligne-semisimplicity} to every $\rmE^{[w]}$ and choose one representative $V_i$ of each isomorphism class of underlying irreducible $G$-cocycles occurring in any of these decompositions. Retain the pure Hodge structure of weight $\nu_i$ supplied by that theorem, with the $A_i$-action of type $(0,0)$ supplied by Lemma~\ref{end-hodge-lemma} and the intertwiner construction in its proof.

Put $W_i=\operatorname{Hom}_G(V_i,\rmE_\bbF)$, with the left $A_i$-action by precomposition. These spaces are finite-dimensional by ergodicity. The invariant projections onto the weight summands give
\[
 W_i=\bigoplus_w\operatorname{Hom}_G(V_i,\rmE_\bbF^{[w]}).
\]
The bundle $\operatorname{Hom}_\bbF(V_i,\rmE_\bbF^{[w]})$ is a pure variation of weight $w-\nu_i$. The fixed-part theorem applies to it. To check the cocycle bound, recall that the proof of the pure theorem realizes $V_i$ as a subvariation of $\operatorname{Hom}_\bbF(M_i,\rmE^{[w_i]})$ for a constant pure Hodge structure $M_i$. The induced Hodge norm on this subvariation satisfies the same growth bound as the ambient bundle; taking duals and tensor products preserves that bound. The flat polarization and $K_G$-compatibility are inherited by these operations as well. Thus all the hypotheses of Theorem~\ref{fixed_part_thm}, in its stated generality, hold for the indicated Hom-variation.

Every Hodge component of an invariant map $V_i\to\rmE^{[w]}$ is consequently invariant. Hence
\[
 W_i^{[m]}:=\operatorname{Hom}_G(V_i,\rmE_\bbF^{[\nu_i+m]})
\]
carries a constant pure Hodge structure of weight $m$, with absent weight summands interpreted as zero. Precomposition by $A_i$ preserves these Hodge structures.

The evaluation classification in the proof of Theorem~\ref{thm__Deligne-semisimplicity} applies to the underlying semisimple cocycle, independently of its weights. It identifies $V_i\otimes_{A_i}W_i$ with the full $V_i$-isotypic summand and classifies invariant subbundles by $W_i'=\operatorname{Hom}_G(V_i,U)$. Evaluation is Hodge-compatible because a map of type $(p,q)$ takes $V_i^{a,b}$ to $\rmE^{a+p,b+q}$, giving the formula for each weight summand.

If the $W_i'$ are Hodge substructures, their evaluated images respect the Hodge decomposition. Conversely, if $U$ respects that decomposition, every Hodge component of an invariant map into $U$ still lands in $U$ and is invariant. Thus each $W_i'$ is a Hodge substructure.
\end{proof}

\begin{remark}
The ambient weight splitting is part of the hypothesis. To see why invariant subbundles need not respect it, consider two trivial real line bundles $L_0$ and $L_2$, equipped with constant pure Hodge structures of types $(0,0)$ and $(1,1)$, respectively. Their direct sum has weights $0$ and $2$, but the two summands are isomorphic as flat bundles. Thus, if $e_0$ and $e_2$ are nonzero flat sections of the respective summands, the diagonal line $\bbR(e_0+e_2)$ is invariant but does not respect the weight decomposition. It corresponds to a subspace of the multiplicity space of the single trivial irreducible cocycle that is not a Hodge substructure. The corollary concerns finite direct sums of pure variations; it makes no assertion about nonsplit extensions of different weights.
\end{remark}

\begin{corollary}[Fixed part and Deligne-torus invariance for finite direct sums]
\label{cor:graded-fixed-part}
For a finite direct sum $\rmE=\bigoplus_w\rmE^{[w]}$ of pure tensor constructions, every Hodge component of a measurable invariant section is invariant. The Deligne torus takes invariant subbundles of $\rmE$ to invariant subbundles.
\end{corollary}

\begin{proof}
The weight projections are invariant, so apply Theorem~\ref{fixed_part_thm} to each projected section. For the second assertion apply this observation to
\[
 \operatorname{End}(\rmE)
 =\bigoplus_{a,b}\operatorname{Hom}(\rmE^{[a]},\rmE^{[b]}),
\]
whose displayed summands are pure of weights $b-a$, and repeat the invariant-projection proof of Theorem~\ref{deligne-torus-thm}. Reductivity holds for the direct sum because it is an algebraic representation of the original reductive cocycle.
\end{proof}

\subsection{Why arbitrary conjugations are excluded}
\label{subsec:arbitrary-conjugations}

The restriction on $G$ at the beginning of this section is essential. In particular, Theorem~\ref{deligne-torus-thm} fails if one allows an arbitrary conjugate of a diagonal subgroup. We give a genus-two example in which the obstruction lies entirely in the tautological planes.

Let
\[
X=\cH^{(1)}(2)\times\cH^{(1)}(2).
\]
Each factor parametrizes unit-area genus-two translation surfaces with one zero of order two, corresponding to a cone point of angle $6\pi$. Equip each factor with its normalized Masur--Veech measure $\nu$, and put $\mu=\nu\otimes\nu$. Set
\[
Q=\begin{pmatrix}2&0\\0&1/2\end{pmatrix}\in\SLgr
\]
and consider the subgroup
\[
G_Q=\{(g,Q^{-1}gQ):g\in\SLgr\}
\subset(\SLgr)^2.
\]
Both coordinate projections of $G_Q$ are surjective. This group is conjugate to the diagonal subgroup, but it is not covered by our standing hypothesis.

The measure $\mu$ is invariant under the full product group, whose action is ergodic. Since $Q$ commutes with $a_t=\operatorname{diag}(e^t,e^{-t})$, the group $G_Q$ contains the simultaneous geodesic flow $(a_t,a_t)$. Lemma~\ref{lemma:mautner-blocks}, applied to the full product action, makes this flow ergodic. Hence $G_Q$ is ergodic as well.

We now consider the real cohomology bundle
\[
\rmH_\bbR
=\pi_1^*\rmH^1_\bbR\oplus\pi_2^*\rmH^1_\bbR.
\]
For $z_i=(S_i,\omega_i)$, let
\[
T_i(z_i)
=\operatorname{span}_{\bbR}
\{[\operatorname{Re}\omega_i],[\operatorname{Im}\omega_i]\}
\subset H^1(S_i,\bbR)
\]
be the tautological plane, and define its frame by
\[
F_i(z_i):\bbR^2\longrightarrow T_i(z_i),
\qquad
F_i(z_i)(a,b)
=a[\operatorname{Re}\omega_i]+b[\operatorname{Im}\omega_i].
\]
These frames trivialize the tautological bundles, not the full cohomology bundles. The tautological planes are preserved by both the Kontsevich--Zorich cocycle and the Weil operator.

In these moving frames, the Kontsevich--Zorich action of $g\in\SLgr$ is represented by $g^{-T}$. Indeed, the action on real and imaginary periods gives
\[
\alpha_i^{KZ}(g,z_i)F_i(z_i)
=F_i(gz_i)g^{-T}.
\]
Consequently, the two cocycle matrices for the action of $(g,Q^{-1}gQ)\in G_Q$ on $T_1\oplus T_2$ are
\[
\rho_1(g)=g^{-T},
\qquad
\rho_2(g)=(Q^{-1}gQ)^{-T}
=Qg^{-T}Q^{-1},
\]
where we used $Q^T=Q$. In particular,
\[
\rho_2(g)Q=Q\rho_1(g).
\]

Define a rank-two subbundle $V\subset\rmH_\bbR$ by
\[
V_{(z_1,z_2)}
=
\left\{
\bigl(F_1(z_1)u,F_2(z_2)Qu\bigr):
u\in\bbR^2
\right\}.
\]
Thus, in the tautological frames, $V$ is the graph of $Q$. The preceding intertwining identity shows that $V$ is $G_Q$-invariant.

We claim that its image under the Weil operator is not $G_Q$-invariant. With our convention $C|_{\rmH^{p,q}}=i^{p-q}$, we have
\[
C[\operatorname{Re}\omega_i]
=-[\operatorname{Im}\omega_i],
\qquad
C[\operatorname{Im}\omega_i]
=[\operatorname{Re}\omega_i].
\]
Thus, in each tautological frame, the Weil operator has matrix
\[
J=\begin{pmatrix}0&1\\-1&0\end{pmatrix},
\]
and on $T_1\oplus T_2$ it is represented by $J\oplus J$. If $u=Jv$, then
\[
(J\oplus J)(v,Qv)
=(u,-JQJu).
\]
For our choice of $Q$,
\[
-JQJ=Q^{-1}.
\]
Therefore $CV$, in the same tautological frames, is the graph of $Q^{-1}$.

For this graph to be $G_Q$-invariant, we would need
\[
\rho_2(g)Q^{-1}=Q^{-1}\rho_1(g)
\qquad\text{for every }g\in\SLgr.
\]
Using the formulas above, this is equivalent to
\[
Q^2g^{-T}=g^{-T}Q^2
\qquad\text{for every }g\in\SLgr.
\]
But take
\[
g=\begin{pmatrix}1&4\\0&1\end{pmatrix}.
\]
Then
\[
g^{-T}=\begin{pmatrix}1&0\\-4&1\end{pmatrix},
\qquad
Q^2=\begin{pmatrix}4&0\\0&1/4\end{pmatrix},
\]
and these matrices do not commute. Thus $CV$ is not $G_Q$-invariant. Since the cocycle matrices in the tautological frames depend only on $g$, this failure occurs at every point of $X$ for the same fixed group element. In particular, it cannot be removed by modifying the bundle on a $\mu$-null set.

Complexifying gives a $G_Q$-invariant subbundle
\[
V_\bbC=V\otimes_\bbR\bbC
\subset\rmH_\bbR\otimes_\bbR\bbC
\]
whose Weil transform is not $G_Q$-invariant either, since complexification reflects equality of real subspaces. Thus Theorem~\ref{deligne-torus-thm} fails for this ergodic $G_Q$-invariant probability measure.

The example also rules out the Hodge-compatible semisimplicity classification of Theorem~\ref{thm__Deligne-semisimplicity} for arbitrary conjugates of diagonal subgroups. Indeed, under such a decomposition, every invariant subbundle would have the form
\[
U=\bigoplus_i V_i\otimes_{A_i}W_i',
\]
and Hodge compatibility would give
\[
CU=\bigoplus_i V_i\otimes_{A_i}C_{W_i}(W_i').
\]
Here $C_{W_i}$ denotes the Weil operator on the constant multiplicity space $W_i$. Since the $A_i$-action has Hodge type $(0,0)$, each $C_{W_i}(W_i')$ is again an $A_i$-submodule. Consequently $CU$ would be invariant, contradicting the example above.

\begin{remark}[Other genera and general conjugations]
The construction works on the product of any two connected components of unit-area strata, equipped with their normalized Masur--Veech measures. Only the tautological planes enter the calculation. In particular, the same obstruction occurs for every pair of genera $g_1,g_2\geq2$, including unequal genera.

One can also identify the obstruction for an arbitrary $Q\in\SLgr$. In the tautological frames, the second cocycle matrix is now
\[
\rho_2(g)=Q^Tg^{-T}Q^{-T},
\]
so the graph of $Q^T$ defines a $G_Q$-invariant subbundle $V_Q$. Since
\[
JQ^TJ^{-1}=Q^{-1},
\]
its Weil transform is the graph of $Q^{-1}$. This graph is invariant if and only if
\[
g^{-T}(Q^{-T}Q^{-1})
=(Q^{-T}Q^{-1})g^{-T}
\qquad\text{for every }g\in\SLgr.
\]
The commutant of the standard contragredient representation consists of scalar matrices. Since $Q^{-T}Q^{-1}$ is positive definite with determinant one, the displayed condition is equivalent to $Q^{-T}Q^{-1}=\Id$, or $Q\in SO_2(\bbR)$. The product measure remains $G_Q$-ergodic: the measure-preserving map $(z_1,z_2)\mapsto(z_1,Qz_2)$ conjugates this action to the diagonal action, which is ergodic by Lemma~\ref{lemma:mautner-blocks}.

Thus every nonrotational twist by an element of $\SLgr$ produces this obstruction. For a rotational twist, the intertwining graph is preserved by the Weil operator, in agreement with our standing hypothesis.
\end{remark}

\section{Rigidity}\label{sect6}

Throughout this section we keep the standing hypothesis on $G$ from Section~\ref{sect5}. After suppressing coordinatewise rotations, write
\[
 G=DF,\qquad D=D_1\cdots D_r,\qquad F\leq\{\pm\Id\}^n,
\]
where $D_\beta$ acts diagonally on the original block $I_\beta$. Let $\mu$ be an ergodic $G$-invariant probability measure on the product of area-one strata, and let $\rmE_\bbF$ be a fixed pure tensor construction of the Hodge bundle, or a finite direct sum of such constructions, with its Hodge metric. We consider an arbitrary measurable $G$-invariant subbundle $\rmV\subset\rmE_\bbF$, where $\bbF=\bbR$ or $\bbC$.

The proof follows the leafwise argument of \cite{Filip2013}. Section~\ref{sect5} supplies invariant Hodge-orthogonal complements. We then use a single simultaneous \Teichmuller~flow to prove polynomiality along stable and unstable plaques. A separate-polynomiality theorem on products of quadrics gives joint polynomiality. The supporting local-geometry and separate-polynomiality arguments are recorded in Appendices~\ref{app:local-geometry} and~\ref{app:separate-polynomiality}. Appendix~\ref{app:tempered-regularity} verifies the noncompact regularity needed for the invariance principle.

\subsection{The measure and its local period model}
\label{subsec:rigidity-local-model}

\begin{proposition}[Reduction to block marginals]
\label{prop:affine-measure-reduction}
The measure $\mu$ has finitely many $D$-ergodic components, permuted transitively by $F$. Each component $\eta$ is a product
\[
 \eta=\eta_1\otimes\cdots\otimes\eta_r,
\]
where $\eta_\beta$ is an ergodic $D_\beta$-invariant probability measure on $\prod_{i\in I_\beta}\cH^{(1)}(\alpha_i)$.
\end{proposition}

\begin{proof}
The finite-component assertion is Lemma~\ref{lemma:mautner-blocks}.

The marginal $\eta_\beta$ is $D_\beta$-ergodic: the inverse image of a $D_\beta$-invariant set is $D$-invariant. For two blocks, disintegrate $\eta=\int\eta_x\,d\eta_1(x)$. Since $D_2$ acts trivially on the first coordinate, $\eta_x$ is $D_2$-invariant for almost every $x$.
Its barycentre is the ergodic measure $\eta_2$; extremality gives $\eta_x=\eta_2$ almost everywhere. Iteration proves the assertion.
\end{proof}

For each block marginal we impose the following hypothesis, working on smooth branches off the null self-intersection locus:
\begin{quote}
\textup{(S)} The support is locally a smooth immersed submanifold $M$, and the marginal measure is locally equivalent to smooth volume on $M$.
\end{quote}
The general rigidity statements in this section assume \textup{(S)}. For the full product action this hypothesis is automatic, as explained in Corollary~\ref{thm__joint-polynomiality}. The product measure-rigidity theorem of \cite[Theorem~1.3.3]{BrownEskinFilipRodriguezHertz2025} motivates this input; we do not use an independently scaled affine-cone description or assert here that \textup{(S)} follows from that formulation.

In period coordinates, write a point as $q=x+iy$ in $\bigoplus_i(E_i\oplus iE_i)$, where $E_i=H^1(X_i,\Sigma_i;\bbR)$ is the real relative cohomology period space of component $i$. The area of that component is $\omega_i(x_i,y_i)$, where $\omega_i$ is the pullback of the absolute intersection pairing.

\begin{lemma}[Local quadric model from smooth support]
\label{lemma:prime-local-model}
For an ergodic diagonal $\SLgr$-invariant measure on $m$ area-one strata satisfying \textup{(S)}, its smooth support $M$ has the following local description.
There is a partition
\[
 \{1,\ldots,m\}=J_1\sqcup\cdots\sqcup J_k
\]
such that a sufficiently small neighborhood in each smooth branch of $M$ is, in linear period coordinates, a product of open subsets of
\[
 Q_{\Omega_a}
 =\{(x_a,y_a)\in L_a\times L_a:\Omega_a(x_a,y_a)=1\}.
\]
Here $L_a\subset\bigoplus_{i\in J_a}E_i$ is real linear and $\Omega_a$ is a nonzero alternating form. All the component area pairings for $i\in J_a$ pull back to $\Omega_a$ on $L_a$.
The measure is locally equivalent to the product smooth measure class. Every vector in $L_a$ with zero absolute part lies in $\ker\Omega_a$.

The partition is determined by the support; its parts can be smaller than an original diagonal block. No invariance of the bundle $V$ under separate actions on the $J_a$ is asserted or required.
\end{lemma}

The proof is given in Appendix~\ref{app:local-geometry}. Apply the lemma to each block marginal and take the product. Thus a sufficiently small neighborhood in every smooth branch of a $D$-ergodic component has the form
\begin{equation}\label{eq:rigidity-product-chart}
 U=U_1\times\cdots\times U_s,
 \qquad U_a\subset Q_{\Omega_a},
\end{equation}
with smooth product measure class. We call these the \emph{local quadric factors}. Their partition refines the original partition into the $I_\beta$. For example, a diagonal joining of two identical genus-one factors has one local quadric factor, while their product Haar measure has two even when the acting group is diagonal.

\begin{theorem}[Joint polynomiality on the normalized support]
\label{thm__joint-polynomiality-fixed-area}
Assume \textup{(S)} for the block marginals of the $D$-ergodic components of $\mu$. On each smooth local branch of their supports, every matrix coefficient of the Hodge-orthogonal projection $\Pi_{\rmV}$, in a flat trivialization, agrees almost everywhere with the restriction of an ambient polynomial in real period coordinates, with coefficients in $\bbF$.
\end{theorem}

We prove this theorem in Subsection~\ref{subsec:rigidity-joint} after establishing the leafwise statement.

\subsection{Polynomiality along stable and unstable leaves}
\label{subsec:rigidity-leafwise}

\begin{lemma}[Invariant Hodge complements]
\label{lemma:invariant-hodge-complement}
The Hodge-orthogonal complement $\rmV^{\perp_H}$ is $G$-invariant. Consequently $\Pi_{\rmV}$ is a $G$-equivariant measurable section of $\operatorname{End}(\rmE)$.
\end{lemma}

\begin{proof}
For a real pure tensor construction, write $I$ for the polarization and $C$ for the Weil operator. Then
\[
 h(v,u)=I(Cv,u)=I(v,C^{-1}u),
 \qquad
 \rmV^{\perp_H}=(C^{-1}\rmV)^{\perp_I}.
\]
Theorem~\ref{deligne-torus-thm} makes $C^{-1}\rmV$ invariant, and flatness of $I$ makes its annihilator invariant. For a complex subbundle use
$h(v,u)=I(v,C^{-1}\overline u)$ and the invariant bundle $C^{-1}\overline{\rmV}$. For a finite direct sum of weights take the direct sum of the bilinear polarizations and use Corollary~\ref{cor:graded-fixed-part}. The argument uses only their flatness and nondegeneracy; no common symmetry sign is required. Positivity of $h$ gives the direct sum $\rmE=\rmV\oplus\rmV^{\perp_H}$ and hence equivariance of its projection.
\end{proof}

Work from now on with one $D$-ergodic component $\eta$. Use the simultaneous flow
\[
 g_t=(a_t,\ldots,a_t)\in D,
 \qquad a_t=\operatorname{diag}(e^t,e^{-t}).
\]
\begin{lemma}\label{lemma:simultaneous-flow}
The flow $g_t$ is ergodic for $\eta$. On the local model \eqref{eq:rigidity-product-chart}, its strong unstable and stable plaques are respectively the connected local fibers with all $y_a$ fixed and all $x_a$ fixed. Their tangent spaces are
\[
 T^+=\bigoplus_a\{(v_a,0):\Omega_a(v_a,y_a)=0\},
 \qquad
 T^-=\bigoplus_a\{(0,w_a):\Omega_a(x_a,w_a)=0\}.
\]
All tangent Lyapunov exponents on $T^+$ are positive and those on $T^-$ are negative.
\end{lemma}

\begin{proof}
Ergodicity of the flow and its time-one map follows from Lemma~\ref{lemma:mautner-blocks}.

The factorwise tangent-spectrum calculation in Lemma~\ref{lemma:tangent-sign-decomposition} identifies the positive directions with the area-preserving real period variations and the negative directions with the imaginary ones. Intersecting with the tangent of the local model gives the displayed formulas. Their integral plaques are affine fibers, since $\Omega_a(x_a,y_a)=1$ is affine in either variable when the other is fixed. Appendix~\ref{app:tempered-regularity} constructs tempered affine plaque charts and proves exponential contraction in the appropriate time direction, using the AGY norm and the spectral gap just identified. The plaques have the full positive or negative tangent dimension, so they are the strong unstable or stable plaques. Relative directions have tangent exponents $1$ and $-1$ and cause no exception.
\end{proof}

Write the forward Lyapunov filtration of $\rmE$ for $g_t$ as
\[
 \rmE=\rmE^{\leq\lambda_1}\supsetneq\cdots
 \supsetneq\rmE^{\leq\lambda_m}\supsetneq\{0\},
 \qquad \lambda_1>\cdots>\lambda_m,
\]
and put $\rmE^{\leq\lambda_{m+1}}=0$. For an invariant subbundle, its filtration terms are its intersections with these terms.

\begin{lemma}[Stable invariance principle]
\label{lemma:invariance-principle-block}
Gauss--Manin transport along almost every stable plaque preserves the forward Lyapunov filtration. If a measurable $g_t$-invariant bundle satisfies
\[
 \rmE^{\leq\lambda_{i+1}}\subseteq\rmB\subseteq\rmE^{\leq\lambda_i},
\]
then $\rmB$ has a representative flat on almost every stable plaque. The corresponding statements for the backward filtration and unstable plaques hold after reversing time.
\end{lemma}

\begin{proof}
Lemma~\ref{lemma:tempered-flat-charts} supplies contracting flat plaque charts with tempered radii and bounds. All almost-everywhere assertions below use \textup{(S)} and Fubini, so regular points have full smooth conditional measure on almost every plaque.

First, Gauss--Manin transport preserves the forward filtration. Indeed, almost every pair of regular points on a stable plaque has simultaneous forward returns to compact sets of measure greater than $1/2$, by Birkhoff's theorem. Transport and its inverse are uniformly bounded at sufficiently close pairs in these sets. Equivariance identifies the growth rates along a subsequence of these returns, and the Oseledets limits at both endpoints give the assertion.

Consequently each quotient
\[
 Q_i=\rmE^{\leq\lambda_i}/\rmE^{\leq\lambda_{i+1}},
\]
has flat stable trivializations and an induced Gauss--Manin holonomy $P_i$. With the quotient Hodge norm, its cocycle and transport bounds follow from those on $\rmE$. In the flat quotient trivializations the fiber transition and cocycle matrices remain spatially constant. Thus the smoothness and temperedness requirements of \cite[\S D.3.2]{BrownEskinFilipRodriguezHertz2025} hold directly for $Q_i$, and the Hodge growth bound gives logarithmic integrability. Since $Q_i$ has the single exponent $\lambda_i$, \cite[\S B.2.3 and Theorem~B.2.4]{BrownEskinFilipRodriguezHertz2025} provide its standard stable holonomy $H_i$, preserving the invariant subbundle $\rmB/\rmE^{\leq\lambda_{i+1}}$.

We check that $H_i=P_i$. Both are equivariant and equal the identity on the diagonal. On compact sets of arbitrarily large measure, their local plaque norms and Lipschitz constants are bounded: for $P_i$ this follows from the flat charts, and for $H_i$ from the smooth single-exponent holonomy construction in the cited references. Shrink these sets also to bound the plaque radii from below.

Fix a typical pair $q,q'$ on the same stable plaque and take simultaneous return times $n$ to such a set. For some $\chi>0$ their distance at time $n$ is $O(e^{-\chi n})$. Since $Q_i$ has only the exponent $\lambda_i$, the norms of its forward cocycle and inverse give, for any sufficiently small $\epsilon>0$,
\[
 \|H_i(q,q')-P_i(q,q')\|
 \leq C(q,q',\epsilon)e^{2\epsilon n}e^{-\chi n}.
\]
Here equivariance conjugates the difference at time $n$ back to time zero; the factors with exponent $\lambda_i$ cancel. Taking $2\epsilon<\chi$ proves $H_i=P_i$. This is the single-exponent uniqueness argument recalled in \cite[\S B.2.9]{BrownEskinFilipRodriguezHertz2025}.

Thus Gauss--Manin holonomy preserves $\rmB/\rmE^{\leq\lambda_{i+1}}$; taking its inverse image shows that it preserves $\rmB$. A flat chart gives a constant representative on almost every plaque. Time reversal gives the unstable assertion.
\end{proof}

\begin{lemma}[Leafwise analyticity]
\label{thm__lyapunov-filtrations-leafwise-analytic}
\label{cor:projection-leafwise-analytic}
The bundle $\rmV$, its Hodge-orthogonal complement, and the projection $\Pi_{\rmV}$ vary real-analytically on almost every stable and almost every unstable plaque.
\end{lemma}

\begin{proof}
Put $\rmW=\rmV^{\perp_H}$. The forward filtration splits as
\[
 \rmE^{\leq\lambda_i}
 =\rmV^{\leq\lambda_i}\oplus\rmW^{\leq\lambda_i}.
\]
Induct from the smallest exponent upward, proving analyticity of both filtrations along stable plaques. Their smallest terms are flat by Lemma~\ref{lemma:invariance-principle-block}. At the next step the same lemma makes
\[
 \rmB_i=\rmV^{\leq\lambda_i}+\rmE^{\leq\lambda_{i+1}}
        =\rmV^{\leq\lambda_i}\oplus\rmW^{\leq\lambda_{i+1}}
\]
flat. The last decomposition is Hodge-orthogonal. Since its second summand is analytic by induction and the Hodge metric is real-analytic in period coordinates, its orthogonal complement inside $\rmB_i$ is analytic. Interchange $\rmV$ and $\rmW$ and continue. The top terms are the full bundles. Their associated projection is analytic as well. Reverse time for unstable plaques.
\end{proof}

We use the following elementary consequence of recurrence. If all Lyapunov exponents of a cocycle are strictly negative, every measurable invariant section is zero. Indeed, a nonzero invariant section has a positive-measure set on which its norm is bounded above and away from zero. Recurrence to that set contradicts its exponential decay. This argument does not require the section itself to be integrable.

\begin{lemma}[Leafwise polynomiality]\label{lemma__polynomiality}
There is an integer $d$ such that $\Pi_{\rmV}$ agrees with a polynomial of degree at most $d$ on almost every connected local stable and unstable plaque, in flat plaque coordinates.
\end{lemma}

\begin{proof}
Let $\Lambda$ bound the Lyapunov exponents of the induced cocycle on $\operatorname{End}(\rmE)$ from above, in both time directions. By Lemma~\ref{lemma:simultaneous-flow}, choose $\chi>0$ below the smallest positive exponent on $T^+$ and below the absolute value of every negative exponent on $T^-$. These constants exist on the ergodic component under consideration.

Along unstable plaques, take the full homogeneous derivative tensor
\[
 T_k=\frac1{k!}D^k_{\mathrm{unstable}}\Pi_{\rmV}
 \in\operatorname{End}(\rmE)\otimes\operatorname{Sym}^k((T^+)^*).
\]
The analytic plaque representatives define measurable tensors. Affine changes of flat plaque coordinates make them intrinsic, and differentiating equivariance of $\Pi_{\rmV}$ makes $T_k$ an invariant section of the indicated tensor cocycle. Every Lyapunov exponent of this cocycle is at most $\Lambda-k\chi$. Choose $k$ so that $k\chi>\Lambda$. The recurrence observation gives $T_k=0$ almost everywhere.

By the smooth measure class and Fubini, on almost every plaque this analytic tensor vanishes at almost every point, hence identically. The projection is therefore a polynomial of degree at most $k-1$ on that connected plaque. Reverse time for stable plaques. Finally, there are only finitely many $D$-ergodic components; take the maximum of the resulting degree bounds. It is essential here to use the full homogeneous tensor, since the tangent cocycle can mix monomials of the same degree.
\end{proof}

\subsection{Joint polynomiality}
\label{subsec:rigidity-joint}

The algebraic input is the following consequence of symplectic separate polynomiality. Its proof, including the almost-everywhere assembly, is given in Appendix~\ref{app:separate-polynomiality}.

\begin{proposition}[Separate polynomiality on products of quadrics]
\label{prop:product-quadric-polynomiality}
Let $U=\prod_{a=1}^sU_a$, where each $U_a$ is a connected, relatively compact open subset of
\[
 Q_{\Omega_a}=\{(x_a,y_a)\in L_a\times L_a:
                         \Omega_a(x_a,y_a)=1\},
\]
and each $\Omega_a$ is a nonzero alternating form, possibly degenerate. Let $f\in L^1(U)$. Suppose that, for every $a$ and almost every choice of the other quadric coordinates, $f$ is locally polynomial of degree at most $d$ on almost every affine $x_a$-fiber with $y_a$ fixed, and on almost every affine $y_a$-fiber with $x_a$ fixed. Then $f$ agrees almost everywhere on $U$ with an ambient polynomial. Its degree has a bound depending only on $d$ and the dimensions of the $L_a$.
\end{proposition}

\begin{proof}[Proof of Theorem~\ref{thm__joint-polynomiality-fixed-area}]
Work in a relatively compact product chart \eqref{eq:rigidity-product-chart}, and let $\pi$ be a flat matrix coefficient of $\Pi_{\rmV}$. It is locally bounded: the Hodge operator norm of an orthogonal projection is at most $1$, and the Hodge and coordinate norms are uniformly equivalent on compact subsets.

By Lemmas~\ref{lemma:simultaneous-flow} and~\ref{lemma__polynomiality}, $\pi$ is polynomial of uniformly bounded degree on almost every plaque with all $y_a$ fixed, and likewise with all $x_a$ fixed. In foliation coordinates, the measure is equivalent to smooth product measure. Fubini, first over plaques and then over the other affine fiber variables within a plaque, therefore supplies exactly the factorwise hypotheses of Proposition~\ref{prop:product-quadric-polynomiality}. This step uses almost-everywhere restrictions, never a choice of a prescribed slice through an exceptional set.

The proposition gives a polynomial in the $L_a$-coordinates. For $\bbF=\bbC$, apply it to real and imaginary parts. Since the $L_a$ embed linearly in the ambient relative cohomology spaces, a linear retraction extends this polynomial to ambient period coordinates. Repeating on the finitely many $D$-ergodic components proves the theorem.
\end{proof}

\subsection{Extension to the cone of original blocks}

Let $C_G(\mathcal M)$ be the cone obtained from the fixed-area support by allowing one positive scalar $t_\beta$ for all components in $I_\beta$, independently for different $\beta$. On this cone all component areas within $I_\beta$ are equal; denote their common value by $A_\beta$. Thus $A_\beta=t_\beta^2$.
Extend the projection by pullback under area normalization, using the natural flat identification along scalar rays.
We use period coordinates after the rotational coordinate change in the standing hypothesis, so the original block geodesic flows act by simultaneous real expansion and imaginary contraction.

\begin{theorem}[Polynomiality on the cone of original blocks]
\label{thm:original-block-cone}
Under the hypotheses of Theorem~\ref{thm__joint-polynomiality-fixed-area}, each matrix coefficient of the extended projection, locally on $C_G(\mathcal M)$, has the form
\[
 \widetilde\pi(x,y)
 =\frac{P(x,y)}{\prod_{\beta=1}^r A_\beta^{D_\beta}},
 \qquad D_\beta\in\mathbb Z_{\geq0},
\]
where $P$ has coefficients in $\bbF$ and is bihomogeneous of total degree $(D_\beta,D_\beta)$ in all the $x$-variables and all the $y$-variables belonging to the original block $I_\beta$.
\end{theorem}

\begin{proof}
On a small connected period neighborhood of the normalized support, write $\pi=Q|_Z$ almost everywhere, where $Q$ is polynomial and
\[
 Z=\prod_aQ_{\Omega_a}
\]
is the local algebraic model from Lemma~\ref{lemma:prime-local-model}.
The flow in the original block $\beta$ acts by
\[
 (x_i,y_i)\longmapsto(e^t x_i,e^{-t}y_i),
 \qquad i\in I_\beta.
\]
Equivariance of the projection and its flat trivialization imply invariance of $Q|_Z$ under this local flow. Initially this equality holds almost everywhere; it holds everywhere on the overlapping open sets because both sides are polynomial restrictions and the measure is of full smooth class there.

Decompose $Q$ into its finitely many weight components for the product of these $r$ split tori. Linear independence of their characters shows that the restriction of every nonzero-weight component vanishes. This can be checked on a smaller open set where all sufficiently small flow parameters are allowed. We may therefore replace $Q$ by its weight-zero component. Decomposing further by total degrees gives
\[
 Q=\sum_{\boldsymbol d}Q_{\boldsymbol d},
\]
where $Q_{\boldsymbol d}$ has total bidegree $(d_\beta,d_\beta)$ in each original block.

Area normalization multiplies both $x_i$ and $y_i$, for $i\in I_\beta$, by $A_\beta^{-1/2}$. Hence
\[
 \widetilde\pi(x,y)
 =\sum_{\boldsymbol d}
   \frac{Q_{\boldsymbol d}(x,y)}{\prod_\beta A_\beta^{d_\beta}}.
\]
Choose integers $D_\beta\geq\max_{\boldsymbol d}d_\beta$ and combine denominators. The resulting numerator is
\[
 P=\sum_{\boldsymbol d}Q_{\boldsymbol d}
             \prod_\beta A_\beta^{D_\beta-d_\beta},
\]
which has the required bidegrees, since area has bidegree $(1,1)$.
In this expression one may represent $A_\beta$ by the component area polynomial for any fixed $i_\beta\in I_\beta$; these polynomials agree on the cone in question. Linear retractions onto the $L_a$, applied separately to real and imaginary variables and respecting the original blocks, extend the numerator to ambient period space without changing its block bidegrees.
\end{proof}

\begin{remark}[Independent scaling of every component]
Even for a nontrivial diagonal block, the fixed-area theorem gives the following weaker formula on the independently scaled cone:
\[
 \widetilde\pi(x,y)
 =Q\bigl((A_i^{-1/2}x_i,A_i^{-1/2}y_i)_i\bigr)
 \in\bbF[x,y,A_1^{-1/2},\ldots,A_n^{-1/2}].
\]
Positive square roots are understood. This formula is homogeneous of degree zero under each positive componentwise scalar dilation.
When all scalings in an original block agree, the preceding theorem eliminates its square roots. Without the additional independent flow invariance, it does not eliminate the square roots separately in each component.
\end{remark}

\begin{corollary}[Regular representatives]\label{cor:regular-representatives}
Under the hypotheses of Theorem~\ref{thm__joint-polynomiality-fixed-area}, $\rmV$ has a unique continuous representative on each smooth branch of the support of every $D$-ergodic component. Its Hodge-orthogonal projection is locally the restriction of an ambient polynomial, and the resulting subbundle is real-analytic. These representatives are $G$-equivariant wherever the corresponding smooth branches are mapped to one another.
\end{corollary}

\begin{proof}
Let $P$ be a polynomial representative of $\Pi_{\rmV}$ in a flat chart. The identities $P^2=P$, $P^{*h}=P$, and $\operatorname{tr}P=\operatorname{rank}\rmV$ hold almost everywhere and therefore everywhere on the branch, by continuity and full support of the smooth measure class. Thus $P$ has constant rank and its image is a real-analytic subbundle. On overlaps within a branch, polynomial representatives agree almost everywhere and hence everywhere. Any continuous representative of $\rmV$ has a continuous Hodge-orthogonal projection, so the same argument proves uniqueness. For each fixed $g\in G$, the equivariance identity extends from almost every point to every point where both branches are smooth. No identification of different branches at a self-intersection is asserted.
\end{proof}

\begin{corollary}[Polynomial projectors after arbitrary $\SLgr$-conjugations]\label{cor:conjugated-polynomial-projectors}
Let $G_0$ satisfy the standing group hypothesis, let $c=(Q_1,\ldots,Q_n)\in(\SLgr)^n$, and put $G=cG_0c^{-1}$. Let $\mu$ be an ergodic $G$-invariant probability measure such that $\mu_0:=(c^{-1})_*\mu$ satisfies \textup{(S)} for the block marginals of its $G_0^\circ$-ergodic components. For any measurable $G$-invariant subbundle $V$ of a fixed tensor construction of the cohomology bundle, or a finite direct sum of such constructions, there is a $G$-equivariant projector onto $V$ whose coefficients are locally polynomial on every smooth branch in real period coordinates. In particular, $V$ has a unique continuous, real-analytic representative on these branches.

The polynomial projector need not be orthogonal for the canonical Hodge metric. The canonical Hodge-orthogonal projector is nevertheless real-analytic.
\end{corollary}

\begin{proof}
Use the ambient product cocycle to define
\[
 (V_0)_y=\Gamma(c,y)^{-1}V_{cy}.
\]
The cocycle identity makes $V_0$ $G_0$-invariant over $(X,\mu_0)$. Let $P_0$ be its polynomial Hodge-orthogonal projector from Theorem~\ref{thm__joint-polynomiality-fixed-area} and Corollary~\ref{cor:regular-representatives}. Set
\[
 P(x)=\Gamma(c,c^{-1}x)P_0(c^{-1}x)\Gamma(c,c^{-1}x)^{-1}.
\]
Then $P^2=P$, its image represents $V$, and it is $G$-equivariant. The map $x\mapsto c^{-1}x$ is linear in period coordinates, while the bundle map $\Gamma(c,\cdot)$ has locally constant matrices in compatible flat trivializations. Hence $P$ has polynomial coefficients. Uniqueness and analyticity of the represented bundle follow by transport from Corollary~\ref{cor:regular-representatives}. Its canonical Hodge-orthogonal projection is real-analytic because both the subbundle and the Hodge metric are real-analytic. This does not assert Deligne-torus invariance for $G$.
\end{proof}

\begin{corollary}[Full product action]\label{thm__joint-polynomiality}
Let $G=(\SLgr)^n$ and let $\mu$ be any ergodic $G$-invariant probability measure on $X$. No additional hypothesis \textup{(S)} is needed. For every measurable invariant subbundle of the bundles considered above, the normalized extension of each Hodge-orthogonal projection coefficient to the independently scaled factorwise cone lies locally in
\[
 \bbF[x_1,\ldots,x_n,y_1,\ldots,y_n,A_1^{-1},\ldots,A_n^{-1}],
\]
and has a representative $P/\prod_iA_i^{D_i}$ with $P$ of bidegree $(D_i,D_i)$ in each pair $(x_i,y_i)$.
\end{corollary}
\begin{proof}
Proposition~\ref{prop:affine-measure-reduction} gives $\mu=\mu_1\otimes\cdots\otimes\mu_n$, with each $\mu_i$ an ergodic $\SLgr$-invariant measure on one stratum. By the single-stratum measure classification \cite{EM18}, every $\mu_i$ is an affine measure, so its support and measure class satisfy \textup{(S)} on smooth branches. Every original block is a singleton; Theorem~\ref{thm:original-block-cone} therefore gives the conclusion.
\end{proof}

\subsection{Sharpness of independent factorwise localization}
\label{subsec:sharpness-full-product}

Independent area normalization can introduce square roots even when the normalized projection is polynomial. We use the genus-two setting and tautological frames of Subsection~\ref{subsec:arbitrary-conjugations}:
\[
 \mathcal M=\cH^{(1)}(2)\times\cH^{(1)}(2),
 \qquad G=\{(g,g):g\in\SLgr\}.
\]
Equip $\mathcal M$ with the product of the normalized Masur--Veech measures. The full product action is ergodic, so Lemma~\ref{lemma:mautner-blocks} makes its simultaneous geodesic flow and hence its diagonal action ergodic. The measure has full smooth support and satisfies \textup{(S)}.

Write $x_i=[\operatorname{Re}\omega_i]$, $y_i=[\operatorname{Im}\omega_i]$, and $I_i$ for the absolute intersection pairing. Since $I_i(x_i,y_i)=1$, the tautological frame and its symplectic dual are
\[
 F_i:\bbR^2\longrightarrow H^1(S_i;\bbR),\qquad F_i(a,b)=ax_i+by_i,
 \qquad
 G_i(v)=\bigl(I_i(v,y_i),I_i(x_i,v)\bigr).
\]
Then $G_iF_i=\Id$ and $\langle F_i u,F_i v\rangle_H=u\cdot v$. The tautological plane is Weil-invariant, so its symplectic complement is also its Hodge-orthogonal complement. Consequently $G_i$ is the Hodge adjoint of $F_i$, and $F_iG_i$ is the Hodge-orthogonal projection onto that plane. Both maps have entries linear in the period coordinates in flat frames.

Fix $\lambda\in\bbR\setminus\{0\}$ and set
\[
 \rmV_\lambda=\{(F_1u,\lambda F_2u):u\in\bbR^2\}
 \subset H^1(S_1;\bbR)\oplus H^1(S_2;\bbR).
\]
The identical contragredient cocycle matrices in the two tautological frames make this bundle $G$-invariant. Orthogonal projection sends $(v_1,v_2)$ to $(F_1u_*,\lambda F_2u_*)$, where
\[
 u_*=(G_1v_1+\lambda G_2v_2)/(1+\lambda^2).
\]
Thus, on the full cohomology bundle,
\[
 \Pi_\lambda=\frac1{1+\lambda^2}
 \begin{pmatrix}
 F_1G_1&\lambda F_1G_2\\
 \lambda F_2G_1&\lambda^2F_2G_2
 \end{pmatrix}.
\]
Every entry is polynomial on the area-one locus.

On the independently scaled cone, write $\widehat x_i,\widehat y_i$ for the unnormalized periods, put $A_i=I_i(\widehat x_i,\widehat y_i)>0$, and define $\widehat F_i,\widehat G_i$ by the same linear formulas. Area normalization gives
\[
 F_i=A_i^{-1/2}\widehat F_i,
 \qquad G_i=A_i^{-1/2}\widehat G_i,
 \qquad
 \widetilde\Pi_{\lambda,12}
 =\frac{\lambda}{1+\lambda^2}
 \frac{\widehat F_1\widehat G_2}{\sqrt{A_1A_2}}.
\]
Fix unit-area base surfaces with periods $x_i^0,y_i^0$ and vary the first surface by the matrix $\operatorname{diag}(t,1)\in\mathrm{GL}_2^+(\bbR)$, keeping the second fixed. For $t>0$ near $1$, this is a period-coordinate slice of the cone with
\[
 \widehat x_1=t x_1^0,\quad \widehat y_1=y_1^0,
 \qquad \widehat x_2=x_2^0,\quad\widehat y_2=y_2^0,
 \qquad A_1=t,\quad A_2=1.
\]
The tautological planes are constant along this slice as subspaces of the flat cohomology fibers. In the fixed bases $(x_2^0,y_2^0)$ and $(x_1^0,y_1^0)$, the restriction of the off-diagonal projection block is
\[
 \widetilde\Pi_{\lambda,12}\big|_{T_2^0}
 =\frac{\lambda}{1+\lambda^2}
 \begin{pmatrix}\sqrt t&0\\0&t^{-1/2}\end{pmatrix}.
\]
Extend these fixed bases to flat bases of the full cohomology spaces. Every element of
\[
 \bbR[\text{period coordinates},A_1^{-1},A_2^{-1}]
\]
restricts to a Laurent polynomial in $t$ on this slice, whereas $\sqrt t$ and $t^{-1/2}$ are not Laurent polynomials on any open interval. Hence independent factorwise localization fails. When the two areas agree, the same projection formula has only an integer power of their common area in its denominator, as required by Theorem~\ref{thm:original-block-cone}. The argument uses only tautological planes and applies in any pair of genera at least two.

\appendix

\section{The local quadric model}
\label{app:local-geometry}

Fix one ergodic diagonal block marginal $\nu$ satisfying \textup{(S)}, with smooth support $M$. Use period coordinates $q=x+iy\in\bigoplus_{i=1}^m(E_i\oplus iE_i)$, where $E_i$ is real relative cohomology, and let $\omega_i$ be the pullback of the absolute intersection pairing. Thus $\omega_i(x_i,y_i)=1$. Put
\[
 K_0(q)=\bigoplus_{i=1}^m
 \{v_i\in E_i:\omega_i(v_i,x_i)=\omega_i(v_i,y_i)=0\},
\]
and, for $c=(c_i)\in\bbR^m$, put $X_c(q)=(c_ix_i)_i$ and $Y_c(q)=(c_iy_i)_i$.
The ambient unstable and stable tangent spaces are
\[
 E^u_q=Y_{\bbR^m}(q)\oplus K_0(q),
 \qquad
 E^s_q=iX_{\bbR^m}(q)\oplus iK_0(q).
\]
The common geodesic flow $a_t$ and its time-one map are ergodic for $\nu$ by the single-block case of Lemma~\ref{lemma:mautner-blocks}. Since the upper triangular subgroup contains this flow, it is ergodic as well, as required by the cited measure-rigidity results.

\begin{lemma}[Sign decomposition of the area-one tangent cocycle]
\label{lemma:tangent-sign-decomposition}
For the common \Teichmuller~flow on the product of $m$ area-one strata, the sum of the positive Oseledets spaces of the tangent cocycle is
\[
 E^u_q=Y_{\bbR^m}(q)\oplus K_0(q),
\]
the sum of the negative Oseledets spaces is
\[
 E^s_q=iX_{\bbR^m}(q)\oplus iK_0(q),
\]
and the zero-exponent space is
\[
 H_{\bbR^m}(q),\qquad H_c=X_c-iY_c.
\]
\end{lemma}

\begin{proof}
In one factor, let $\lambda$ be a Lyapunov exponent of the relative Kontsevich--Zorich cocycle. Under the tangent cocycle, a real period variation acquires the additional factor $e^t$, whereas an imaginary period variation acquires the factor $e^{-t}$. Their tangent exponents are therefore respectively $1+\lambda$ and $-1+\lambda$.

On the symplectic orthogonal of the tautological plane, the Kontsevich--Zorich exponents lie strictly between $-1$ and $1$; on the relative kernel they are zero. Hence every real direction in $K_0(q)$ has positive tangent exponent and every corresponding imaginary direction in $iK_0(q)$ has negative tangent exponent. See \cite[Appendix~A]{EM18} and the tangent-spectrum estimate in \cite[\S7.2, especially Remark~7.3]{Filip2013}.

The tautological directions can be read off directly from the action on period coordinates. The field $Y_c$ has tangent exponent $2$, the field $iX_c$ has tangent exponent $-2$, and
\[
 H_c=X_c-iY_c
\]
has exponent zero. The remaining radial combination is transverse to the area-one locus. Applying this factorwise proves the assertion.
\end{proof}

\begin{lemma}[Real-defined complex tangents force local linearity]
\label{lemma:real-complex-tangents-linear}
Let $C$ be a smooth submanifold of a complex vector space. Suppose that every tangent space $T_zC$ is complex linear and invariant under complex conjugation. Then, near each point, $C$ is an open subset of an affine complex-linear subspace whose direction is defined over $\bbR$. If, in addition, the position vector field $z\mapsto z$ is tangent to $C$, then this affine subspace is linear.
\end{lemma}

\begin{proof}
Choose complex-linear graph coordinates at a base point, with both the tangent space and a complementary space defined over $\bbR$. Complex-linearity of the tangent spaces implies that the graphing map $f$ is holomorphic. Since every tangent graph is invariant under conjugation, the matrix $Df(z)$ is real in these real-defined bases. Its entries are therefore real-valued holomorphic functions and hence constant. Thus the graph is affine complex-linear with real-defined direction.

If the position vector field is tangent, write the affine space as $z_0+S$. For every $z\in C$ one has $z\in S$. Since also $z-z_0\in S$, it follows that $z_0\in S$, and therefore $z_0+S=S$.
\end{proof}

\begin{proof}[Proof of Lemma~\ref{lemma:prime-local-model}]
All vector fields below are the usual fundamental vector fields on the labeled product of strata, so their definitions do not depend on the choice of flat period chart. Define
\[
 U_c=Y_c,\qquad V_c=iX_c,\qquad
 H_c=X_c-iY_c,\qquad R_c=X_c+iY_c.
\]
The first three fields preserve each component area; $R_c$ generates componentwise scalar dilation.

\emph{The neutral directions determine a partition.}
Consider the three subspaces of $\bbR^m$
\[
 \begin{aligned}
 W_+(q)&=\{c:U_c(q)\in T_qM\},\\
 W_-(q)&=\{c:V_c(q)\in T_qM\},\\
 W_0(q)&=\{c:H_c(q)\in T_qM\}.
 \end{aligned}
\]
Since
\[
 (a_t)_*U_c=e^{2t}U_c,\qquad
 (a_t)_*V_c=e^{-2t}V_c,\qquad
 (a_t)_*H_c=H_c,
\]
these are measurable $a_t$-invariant Grassmann-valued maps. Ergodicity makes them constant almost everywhere; denote their values by $W_+,W_-,W_0$. For each vector in a basis of one of these spaces, the corresponding smooth vector field is tangent almost everywhere.
By \textup{(S)}, that full-measure set is dense on every smooth branch. Continuity extends tangency everywhere on the branch.
We may therefore take Lie brackets of these tangent fields.

Let $\boldsymbol1=(1,\ldots,1)$. The fields $U_{\boldsymbol1},V_{\boldsymbol1},H_{\boldsymbol1}$ are tangent by diagonal invariance. Direct calculation gives, with $cd$ denoting coordinatewise multiplication,
\[
 [U_c,V_d]=-H_{cd},\qquad
 [H_c,U_{\boldsymbol1}]=-2U_c,\qquad
 [H_c,V_{\boldsymbol1}]=2V_c.
\]
It follows that $W_+=W_-=W_0=:W$, that $\boldsymbol1\in W$, and that $c,d\in W$ implies $cd\in W$.
Every unital subalgebra of $\bbR^m$ for coordinatewise multiplication is the space of functions constant on the parts of a partition.
Indeed, put $i\sim j$ when $c_i=c_j$ for every $c\in W$; choose an element of $W$ separating the finitely many equivalence classes.
Lagrange interpolation in that element gives every class indicator.
Thus, for a partition $J_1,\ldots,J_k$,
\begin{equation}\label{eq:neutral-partition}
 W=\operatorname{span}_{\bbR}
 \{\boldsymbol1_{J_1},\ldots,\boldsymbol1_{J_k}\}.
\end{equation}

\emph{Recovering the full support tangent.}
By Lemma~\ref{lemma:tangent-sign-decomposition}, the ambient positive, zero and negative Oseledets spaces are respectively
\[
 E^u_q,\qquad H_{\bbR^m}(q),\qquad E^s_q.
\]
The ambient tangent cocycle and its restriction to $TM$ are log-integrable in the standard relative-cohomology norms. For an invariant measurable subbundle of an invertible cocycle, the Oseledets space of a given exponent is the intersection of the ambient Oseledets space with that subbundle. Applying this to $TM$, and using
\[
 T_qM\cap H_{\bbR^m}(q)=H_W(q),
\]
we obtain, almost everywhere,
\[
 T_qM=(T_qM\cap E^u_q)\oplus H_W(q)\oplus(T_qM\cap E^s_q).
\]

Set $\mathcal U_q=T_qM\cap E^u_q$ and $\mathcal S_q=T_qM\cap E^s_q$. Then, almost everywhere,
\begin{equation}\label{eq:full-support-tangent}
 T_qM=\mathcal U_q\oplus H_W(q)\oplus\mathcal S_q.
\end{equation}

\emph{Tangent balance from horocycle invariance.}
Let $n_\tau=\left(\begin{smallmatrix}1&\tau\\0&1\end{smallmatrix}\right)$ act diagonally. Shrink the Oseledets regular set to the finite intersection of its images and inverse images under $n_1$, $n_{-1}$ and rotation by $\pi/2$. This remains conull. We may therefore use the Oseledets decomposition simultaneously at every point needed below. Fix $\tau\in\{1,-1\}$. If $s=i(X_c+v)\in\mathcal S_q$, with $v\in K_0(q)$, then
\[
 (n_\tau)_*s=\tau(X_c+v)+i(X_c+v).
\]
At $q'=n_\tau q=(x+\tau y)+iy$, the decomposition into ambient positive, zero and negative tangent spaces is
\[
 (n_\tau)_*s
 =\underbrace{(-\tau^2Y_c+\tau v)}_{\in E^u_{q'}}
  +\underbrace{H_{\tau c}(q')}_{\text{zero exponent}}
  +\underbrace{i\bigl(X_c(q')+v\bigr)}_{\in E^s_{q'}}.
\]
Each summand belongs to $T_{q'}M$ by its Oseledets splitting.
The inverse horocycle acts as the identity on purely real tangent vectors; moreover $E^u_{q'}$ and $E^u_q$ have the same real period description because their imaginary periods are equal. Pulling the first summand back to $q$ gives
\begin{equation}\label{eq:horocycle-balance-inclusion}
 -\tau^2Y_c+\tau v\in\mathcal U_q.
\end{equation}
These computations can be made using flat transport along the horocycle path, so they do not require one period chart to contain the entire path.

Ergodicity makes both tangent dimensions constant almost everywhere.
Rotation by $\pi/2$ takes one ambient space to the other and preserves $M$ and $\nu$, so these dimensions are equal. For $\tau=\pm1$, the linear maps $i(X_c+v)\mapsto-\tau^2Y_c+\tau v$ are ambient isomorphisms. Thus the inclusions in \eqref{eq:horocycle-balance-inclusion} are equalities on the two support tangent spaces. The two choices $\tau=1,-1$ give
\begin{equation}\label{eq:local-balance}
 B_+(\mathcal S_q)=\mathcal U_q=B_-(\mathcal S_q),
 \qquad B_\pm\bigl(i(X_c+v)\bigr)=Y_c\pm v.
\end{equation}
For $i(X_c+v)\in\mathcal S_q$, adding and subtracting its two images gives $Y_c,v\in\mathcal U_q$. Applying $B_+^{-1}$ shows that $iX_c,iv\in\mathcal S_q$ as well. We obtain a real subspace $K(q)\subset K_0(q)$ with
\begin{equation}\label{eq:balanced-tangents}
 \mathcal U_q=Y_W(q)\oplus K(q),\qquad
 \mathcal S_q=iX_W(q)\oplus iK(q).
\end{equation}

\emph{Adding exactly the required radial directions.}
Let $C$ be the local cone obtained from $M$ by independently scaling the parts $J_a$, using one positive scalar $t_a$ for all $i\in J_a$.
It is a smooth submanifold of dimension $\dim M+k$: its component areas are $t_a^2$, so the area functions recover the parameters and area normalization recovers the point of $M$.
At a point of $M$, its tangent is
\[
 T_qC=T_qM\oplus R_W(q).
\]
Combining \eqref{eq:full-support-tangent} and \eqref{eq:balanced-tangents}, and adding and subtracting $H_c$ and $R_c$, gives
\[
 T_qC=L(q)\oplus iL(q),\qquad
 L(q)=K(q)\oplus X_W(q)\oplus Y_W(q)
\]
almost everywhere. The same assertion holds at scaled points, since these scalings are real linear and commute with multiplication by $i$ and with complex conjugation.
Being complex linear and invariant under conjugation are closed conditions on tangent spaces. Smoothness and density extend both properties to every point of $C$.

\emph{The cone is locally linear.}
By Lemma~\ref{lemma:real-complex-tangents-linear}, $C$ is locally open in an affine complex-linear space with real-defined direction $L\oplus iL$. The common radial field is tangent to $C$, so the position vector belongs to $L\oplus iL$ at every point. The final assertion of that lemma therefore shows that the affine offset vanishes. Thus $C$ is locally open in the linear space $L\oplus iL$ itself.

\emph{Splitting the linear space and the area equations.}
Let $p_a$ be the coordinate projection onto the components $J_a$.
Since dilation in $J_a$ preserves $C$, its infinitesimal field $z\mapsto p_az$ is tangent to $L\oplus iL$ for every $z$ in an open subset of that linear space. Linearity implies $p_a(L\oplus iL)\subset L\oplus iL$. Therefore
\[
 L=\bigoplus_{a=1}^kL_a,\qquad
 L_a=p_aL\subset\bigoplus_{i\in J_a}E_i.
\]
On $C$, all component areas within $J_a$ are equal. Pulling their pairings back to $L_a$, their differences are bilinear polynomials vanishing on an open set of $L_a\times L_a$, hence vanish identically. Denote the common alternating form by $\Omega_a$; it is nonzero because these areas are positive.

The original support is locally precisely the area-one slice of $C$. The equations reduce to $\Omega_a(x_a,y_a)=1$, one for each $a$. This proves the claimed local product description. Each quadric is smooth on that level, since its area differential satisfies
\[
 d\bigl((x,y)\mapsto\Omega_a(x,y)\bigr)_{(x_a,y_a)}(x_a,0)=1.
\]
The measure-class conclusion follows from \textup{(S)}. The final claim follows because all area pairings factor through absolute cohomology.
\end{proof}

\section{Separate polynomiality on products of quadrics}
\label{app:separate-polynomiality}

We prove Proposition~\ref{prop:product-quadric-polynomiality}. The finite-dimensional assembly below is the bounded-degree form of ordinary separate polynomiality; compare \cite[Lemma~7.6]{Filip2013}. The remaining argument treats the constrained affine fibers of a quadric.

\begin{lemma}[Finite-dimensional assembly]
\label{lemma:finite-dimensional-assembly}
Let $U=U_1\times\cdots\times U_s$ be a product of connected, relatively compact smooth coordinate neighbourhoods, and for each $a$ let $\mathcal P_a\subset L^2(U_a)$ be a finite-dimensional space of real-analytic functions. Suppose that $f\in L^1(U)$ has the property that, for every $a$, for almost every choice of the variables outside $U_a$, the $U_a$-slice of $f$ belongs almost everywhere to $\mathcal P_a$. Then
$$
f\in
\mathcal P_1\otimes\cdots\otimes\mathcal P_s
$$
almost everywhere on $U$.
\end{lemma}

\begin{proof}
It is enough to treat two factors and iterate. Choose a basis $p_1,\ldots,p_N$ of $\mathcal P_1$. Choose a nonnegative smooth compactly supported density $\rho_1$ on $U_1$ which is positive on a nonempty open set. The Gram matrix
$$
G_{ij}=\int_{U_1}p_i(u)p_j(u)\rho_1(u)\,du
$$
is positive definite, because a nonzero real-analytic function cannot vanish on an open set. Hence the compactly supported kernels
$$
\varphi_i(u)=\rho_1(u)\sum_j(G^{-1})_{ij}p_j(u)
$$
satisfy
$$
\int_{U_1}p_j(u)\varphi_i(u)\,du=\delta_{ij}.
$$
For almost every $v\in U_2$, the $U_1$-slice of $f$ lies in $\mathcal P_1$, so
$$
f(u,v)=\sum_i c_i(v)p_i(u)
$$
almost everywhere in $u$, with
$$
c_i(v)=\int_{U_1}f(u,v)\varphi_i(u)\,du.
$$
By Fubini, $u\mapsto f(u,\cdot)$ is a Bochner-integrable $L^1(U_2)$-valued map. For almost every $u$, the $U_2$-slice $f(u,\cdot)$ belongs to the finite-dimensional space $\mathcal P_2$. Since a finite-dimensional subspace of $L^1(U_2)$ is closed, the Bochner integral
$$
c_i(\cdot)=\int_{U_1}\varphi_i(u)f(u,\cdot)\,du
$$
also belongs to $\mathcal P_2$. Therefore
$$
f(u,v)=\sum_i p_i(u)c_i(v)
$$
belongs almost everywhere to $\mathcal P_1\otimes\mathcal P_2$. Iterating this two-factor argument proves the general case.
\end{proof}

\begin{lemma}[Local finite-jet consequence of horizontal polynomiality]
\label{lemma:local-finite-jet}
Let $\mathbb G$ be a real Lie group of dimension $N$, and let $S$ be a fixed finite subset of its Lie algebra that Lie-generates the algebra. Fix $k\geq1$. A distribution $T$, defined near the identity and satisfying
\[
 Y^kT=0\qquad(Y\in S),
\]
is represented near the identity by a real-analytic function. The space of solution germs is finite-dimensional, with dimension bounded in terms of $k$ and $N$.
\end{lemma}

\begin{proof}
We localize the argument of \cite[Lemmas~3.9--3.10 and Proposition~3.11]{AntonelliLeDonne2022}. Choose a neighborhood $V$ on which the equations hold. The horizontal word chart in Lemma~3.10 of that paper may be chosen with all intermediate flow segments in $V$: first multiply the elements of $S$ by a sufficiently small common positive scalar. This preserves both Lie generation and the equations. The chart uses $2N$ segments and maps a neighborhood of its parameter base point diffeomorphically onto a neighborhood $U$ of the identity.

For a smooth solution $u$ on $V$, the Taylor propagation formula of their Lemma~3.9 is local along these segments and gives
\[
 u(q)=\sum_{|\alpha|\leq M}a_\alpha(q)\,\partial^\alpha u(e),
 \qquad q\in U,
\]
where the $a_\alpha$ are real-analytic and depend only on the chosen chart. One may take $M=2N(k-1)$: each of the $2N$ segments contributes at most $k-1$ differentiations. Thus every smooth solution on $V$ restricts to the same finite-dimensional analytic space on $U$.

For a distributional solution, choose the chart and its intermediate segments inside a relatively compact subset of $V$. Convolve locally with a smooth approximate identity on the left. Left-invariant differential operators commute with this convolution, so the regularizations satisfy the same equations on a fixed smaller neighborhood containing all these segments. They converge to $T$ in distributions. Their restrictions to $U$ belong to the finite-dimensional span of the $a_\alpha$, which is closed in the space of distributions. The limit therefore has an analytic representative in that span.

Finally, the same bound $M$ works no matter how small the initial neighborhood of a solution germ is. The displayed representation shows that a germ whose $M$-jet at the identity vanishes is zero. The jet map on the space of solution germs is therefore injective, giving a dimension bound depending only on $k$ and $N$.
\end{proof}

\begin{theorem}[Symplectic separate polynomiality]
\label{thm:symplectic-separate-polynomiality}
Let $L$ be a finite-dimensional real vector space, let $\Omega$ be a nonzero alternating bilinear form on $L$, and set
$$
Q_\Omega:=\{(x,y)\in L\times L:\Omega(x,y)=1\}.
$$
Let $U\subset Q_\Omega$ be connected and open, and let $f\in L^1_{\mathrm{loc}}(U)$. Assume that there is an integer $d\geq0$ such that:
\begin{enumerate}
\item for almost every $y$ for which $U_y:=\{x:(x,y)\in U\}$ is nonempty, the function $x\mapsto f(x,y)$ agrees, on each sufficiently small relatively open subset of $U_y$, almost everywhere with the restriction of a polynomial of degree at most $d$ on the affine hyperplane $\{x:\Omega(x,y)=1\}$;
\item for almost every $x$ for which $U^x:=\{y:(x,y)\in U\}$ is nonempty, the function $y\mapsto f(x,y)$ agrees, on each sufficiently small relatively open subset of $U^x$, almost everywhere with the restriction of a polynomial of degree at most $d$ on the affine hyperplane $\{y:\Omega(x,y)=1\}$.
\end{enumerate}
Then there exists a polynomial $P\in\bbR[L\oplus L]$ such that $f=P|_{Q_\Omega}$ almost everywhere on $U$. Moreover, $\deg P$ is bounded in terms of $d$ and $\dim L$ only.
\end{theorem}

\begin{proof}
We first assume that $\Omega$ is nondegenerate.

\smallskip
\noindent\emph{Step 1: homogeneous-space model.}
Write $\dim L=2m$, choose $e,f\in L$ with $\Omega(e,f)=1$, and put $W:=\langle e,f\rangle^\perp$. Let
$$
G:=\operatorname{Sp}(L,\Omega),\qquad H:=\{g\in G:ge=e,\ gf=f\}\cong\operatorname{Sp}(W).
$$
The action of $G$ on $Q_\Omega$ is transitive, since every pair $(x,y)$ with $\Omega(x,y)=1$ extends to a symplectic basis. Thus
$$
G/H\cong Q_\Omega
$$
as real algebraic homogeneous spaces.

Fix a point of $U$ and, after applying an element of $G$, assume it is $(e,f)$. Shrink $U$ and choose an open neighbourhood $\mathcal O\subset G$ of the identity mapping into $U$. Define
$$
\widetilde f(g):=f(ge,gf).
$$
Using a local product decomposition for the submersion $G\to G/H$, this is an $L^1_{\mathrm{loc}}$ function, hence a distribution, on $\mathcal O$. It is locally right $H$-invariant.

\smallskip
\noindent\emph{Step 2: horizontal differential equations.}
For $w\in W$ and $a\in\bbR$, define $X_{w,a}\in\mathfrak{sp}(L,\Omega)$ by
$$
X_{w,a}f=0,\qquad X_{w,a}e=w+af,\qquad X_{w,a}u=\Omega(u,w)f\quad(u\in W).
$$
A direct check gives $X_{w,a}^2=0$. Let $\mathfrak n_f$ be the space of all $X_{w,a}$. Then $\exp(\mathfrak n_f)$ fixes $f$ and acts simply transitively on
$$
\{(x,f):\Omega(x,f)=1\}.
$$
There is an opposite algebra $\mathfrak n_e$ fixing $e$ and acting simply transitively on
$$
\{(e,y):\Omega(e,y)=1\},
$$
and every $X\in\mathfrak n_e$ also satisfies $X^2=0$.

Regard elements of $\mathfrak g:=\mathfrak{sp}(L,\Omega)$ as left-invariant vector fields on $G$, whose flows are $g\mapsto g\exp(tX)$. The standard contact grading is
$$
\mathfrak g=
\mathfrak g_{-2}\oplus\mathfrak g_{-1}\oplus
(\bbR h_0\oplus\mathfrak h)\oplus
\mathfrak g_1\oplus\mathfrak g_2,
$$
with
$$
\mathfrak n_f=\mathfrak g_{-2}\oplus\mathfrak g_{-1},\qquad
\mathfrak n_e=\mathfrak g_1\oplus\mathfrak g_2,\qquad
[\mathfrak g_2,\mathfrak g_{-2}]=\bbR h_0.
$$
Choose once and for all a finite set
$$
S=\{Y_1,\ldots,Y_s\}\subset
\mathfrak n_f\cup\mathfrak h\cup\mathfrak n_e
$$
which Lie-generates $\mathfrak g$. Since the curves generated by the elements of $S\cap(\mathfrak n_f\cup\mathfrak n_e)$ are affine lines in the corresponding leaves, the two hypotheses on $f$, together with Fubini in local flow boxes, imply
$$
Y_j^{d+1}\widetilde f=0
$$
in the sense of distributions whenever $Y_j\in\mathfrak n_f\cup\mathfrak n_e$. Local right $H$-invariance gives
$$
Y_j\widetilde f=0
$$
whenever $Y_j\in\mathfrak h$. In particular, after increasing the exponent if necessary, there is one integer
$$
k=d+1
$$
such that
$$
Y_j^k\widetilde f=0
\qquad(j=1,\ldots,s)
$$
in the sense of distributions.

\smallskip
\noindent\emph{Step 3: local finite-jet lemma.}
Apply Lemma~\ref{lemma:local-finite-jet} to the fixed group $G=\operatorname{Sp}(L,\Omega)$ and the generating set $S$. Let $\mathcal A$ be the full space of analytic solution germs at the identity of
\[
 Y_j^{d+1}u=0\qquad(j=1,\ldots,s).
\]
It is finite-dimensional, contains the germ of $\widetilde f$, and its dimension is bounded in terms of $d$ and $\dim L$.

\smallskip
\noindent\emph{Step 4: finite-dimensional representation.}
For $Y\in\mathfrak g$, let $R_Y$ be the right-invariant vector field generated by $Y$. Left- and right-invariant vector fields commute, so every $R_Y$ preserves the full solution-germ space $\mathcal A$. Since
$$
[R_Y,R_Z]=-R_{[Y,Z]},
$$
the operators
$$
\rho(Y):=-R_Y|_{\mathcal A}
$$
define a finite-dimensional representation of $\mathfrak g$.

Let $v\in\mathcal A$ be the germ of $\widetilde f$ and let $\ell(u)=u(e)$. Complexify $\rho$. By Weyl's theorem and highest-weight theory, every irreducible finite-dimensional $\mathfrak{sp}(L_\bbC,\Omega_\bbC)$-module is the differential of an algebraic representation of the simply connected complex symplectic group
$$
G_\bbC=\operatorname{Sp}(L_\bbC,\Omega_\bbC).
$$
Hence $\rho_\bbC$ integrates to an algebraic representation
$$
\Pi:G_\bbC\to\operatorname{GL}(\mathcal A_\bbC).
$$
Set
$$
F(g):=\ell(\Pi(g^{-1})v).
$$
By construction, every iterated right-invariant derivative of $F$ at the identity agrees with the corresponding derivative of the analytic germ $\widetilde f$. Hence the two analytic germs agree near the identity.

\smallskip
\noindent\emph{Step 5: descent to $Q_\Omega$.}
For $Z\in\mathfrak h$, let $L_Z$ denote the left-invariant vector field whose flow is $g\mapsto g\exp(tZ)$. Since $\widetilde f$ is locally right $H$-invariant and $F=\widetilde f$ near the identity,
$$
L_ZF=0
$$
on a nonempty real-open subset of $G$. The function $L_ZF$ is regular on $G_\bbC$, and the real form $G$ is Zariski dense in $G_\bbC$; hence $L_ZF$ vanishes identically. Since $H_\bbC\cong\operatorname{Sp}(W_\bbC)$ is connected, $F$ is right $H_\bbC$-invariant. Therefore it descends to a regular function on
$$
G_\bbC/H_\bbC\cong(Q_\Omega)_\bbC.
$$
Since
$$
\bbC[(Q_\Omega)_\bbC]
=
\bbC[L_\bbC\oplus L_\bbC]/(\Omega(x,y)-1),
$$
the descended function is the restriction of an ambient polynomial. Taking real parts if necessary yields
$$
P\in\bbR[L\oplus L]
$$
agreeing with $f$ almost everywhere near the chosen point.

For the degree bound, fix a standard symplectic model of dimension $2m$ and the generating set $S$ above. Let $\mathcal A_H\subset\mathcal A$ be the subspace of locally right $H$-invariant germs. The preceding argument applies to every element of this finite-dimensional space. Choose a basis and ambient polynomial representatives $P_1,\ldots,P_N$ of its elements. Then $D=\max_j\deg P_j$ bounds the degree needed for any germ in $\mathcal A_H$. The space and the chosen representatives depend only on $d$ and the fixed symplectic model. Transitivity transports this bound to every point of $Q_\Omega$, and the linear symplectic action does not increase polynomial degree. A linear identification with the standard model proves the required bound for every nondegenerate $\Omega$.

\smallskip
\noindent\emph{Step 6: degenerate $\Omega$.}
Let $K=\ker\Omega$ and choose $L=K\oplus L_0$ with $\Omega_0:=\Omega|_{L_0}$ nondegenerate. Then
$$
Q_\Omega\cong K\times K\times Q_{\Omega_0}.
$$
Work on a sufficiently small product neighbourhood
$$
B_x\times B_y\times U_0
\subset
K\times K\times Q_{\Omega_0}.
$$
For almost every $(x_0,y_0)\in U_0$, the hypotheses on the $x$- and $y$-fibers imply, by Fubini, that the $(k_x,k_y)$-slice is separately polynomial of degree at most $d$ in $k_x$ and $k_y$. By Lemma~\ref{lemma:finite-dimensional-assembly}, that slice belongs almost everywhere to
$$
\mathcal P_d(K)\otimes\mathcal P_d(K),
$$
where $\mathcal P_d(K)$ denotes the polynomials on $K$ of degree at most $d$.

Choose a monomial basis $\{m_\alpha\}$ of $\mathcal P_d(K)$ and test functions $\{\varphi_\alpha\}\subset C_c^\infty(B_x)$, $\{\psi_\beta\}\subset C_c^\infty(B_y)$ dual to this basis:
$$
\int m_\gamma\varphi_\alpha=\delta_{\alpha\gamma},
\qquad
\int m_\gamma\psi_\beta=\delta_{\beta\gamma}.
$$
Using these bounded coefficient functionals gives
$$
f(k_x,k_y,x_0,y_0)
=
\sum_{\alpha,\beta}
m_\alpha(k_x)m_\beta(k_y)f_{\alpha\beta}(x_0,y_0)
$$
almost everywhere on the smaller neighbourhood, where
$$
f_{\alpha\beta}(x_0,y_0)
=
\int f(k_x,k_y,x_0,y_0)\varphi_\alpha(k_x)\psi_\beta(k_y)\,dk_x\,dk_y.
$$
No point evaluation on an almost-everywhere class is used.

The original leafwise polynomiality in the $x_0$- and $y_0$-variables is preserved by these coefficient functionals. Hence each $f_{\alpha\beta}$ satisfies the nondegenerate theorem for $Q_{\Omega_0}$ and is therefore the restriction of an ambient polynomial, with degree bounded in terms of $d$ and $\dim L$. Substituting these finitely many polynomial representatives gives an ambient polynomial representative for $f$ and preserves a uniform degree bound.

\smallskip
\noindent\emph{Step 7: connectedness.}
We have proved that every point of $U$ has a neighbourhood on which $f$ agrees almost everywhere with an ambient polynomial of uniformly bounded degree. The affine variety $Q_\Omega$ is irreducible: if $\Omega$ is nondegenerate this follows from the homogeneous-space description above, while in the degenerate case
$$
Q_\Omega\cong K\times K\times Q_{\Omega_0}
$$
reduces the claim to the nondegenerate factor. On an overlap, two local polynomial representatives agree almost everywhere on a nonempty relatively open subset of $Q_\Omega$, hence define the same regular function there and therefore the same element of
$$
\bbR[Q_\Omega]
=
\bbR[L\oplus L]/(\Omega(x,y)-1).
$$
Since $U$ is connected, the local regular functions patch to one class in $\bbR[Q_\Omega]$. Choosing an ambient polynomial representative of that class gives the desired $P$ on all of $U$.
\end{proof}

\begin{remark}[Dimension two]
If $\dim L=2$ and $\Omega$ is nondegenerate, then $Q_\Omega\cong\operatorname{SL}_2(\bbR)$. The two affine leaf families are the right cosets of the two opposite unipotent one-parameter subgroups. The proof reduces to the fact that if both root operators act locally nilpotently with uniformly bounded order, then the generated $\mathfrak{sl}_2$-module is finite-dimensional; its matrix coefficients are polynomial functions of the matrix entries.
\end{remark}

\begin{proof}[Proof of Proposition~\ref{prop:product-quadric-polynomiality}]
Fix $a$ and apply Theorem~\ref{thm:symplectic-separate-polynomiality} to almost every $U_a$-slice. Its uniform degree bound places all these slices in a single finite-dimensional space $\mathcal P_a$ of restrictions of ambient polynomials. These functions are analytic and belong to $L^2(U_a)$ after shrinking to the relatively compact chart under consideration. Lemma~\ref{lemma:finite-dimensional-assembly} gives
\[
 f\in\mathcal P_1\otimes\cdots\otimes\mathcal P_s
 \quad\text{almost everywhere on }U.
\]
A tensor product of polynomial restriction spaces consists of restrictions of polynomials in all the variables. Their total degree is bounded by the sum of the individual degree bounds, proving the assertion.
\end{proof}

\section{Tempered flat charts and stable holonomies}
\label{app:tempered-regularity}

We verify the regularity used in Subsection~\ref{subsec:rigidity-leafwise}. In particular, local smoothness by itself is not the justification for applying the noncompact invariance principle. Throughout this appendix $\eta$ is one $D$-ergodic component satisfying \textup{(S)}, $f=g_1$ is the simultaneous time-one map, and we work on smooth branches. A positive measurable function $b$ is called \emph{tempered} if
\[
 \frac{\log b(f^nq)}{|n|}\longrightarrow0
 \quad\text{as }n\longrightarrow\pm\infty
\]
for almost every $q$.

\begin{lemma}[Tempering by compact returns]
\label{lemma:compact-return-tempering}
Let $f$ be an invertible ergodic transformation of a probability space $(Y,\eta)$, preserving $\eta$, and let $K\subset Y$ have positive measure. For almost every $q$, its backward return time
\[
 N(q)=\min\{j\geq0:f^{-j}q\in K\}
\]
is finite and satisfies
\begin{equation}\label{eq:return-gap-tempered}
 N(f^nq)=o(|n|)\qquad(n\longrightarrow\pm\infty).
\end{equation}
In particular, $c\exp(bN(q))$ is tempered for every $c>0$ and $b\in\bbR$.
\end{lemma}

\begin{proof}
Birkhoff's theorem in both time directions gives positive visit frequency $\eta(K)$. A gap of length at least $\epsilon|n|$ ending at arbitrarily large positive or negative times contradicts this frequency: subtract the visit counts at the endpoints of an interval of length $\epsilon|n|$ inside the gap. The asymptotic difference is $\eta(K)\epsilon|n|$, whereas the interval contains no visits. Thus \eqref{eq:return-gap-tempered} holds, and the assertion about exponentials follows. No integrability of $N$ is needed.
\end{proof}

\begin{lemma}[Tempered regularity in flat plaque charts]
\label{lemma:tempered-flat-charts}

The local stable plaques, given in period coordinates by affine slices with all real periods fixed, admit measurable parametrizations with tempered positive radii on which $f$ contracts exponentially. The corresponding statement holds for local unstable plaques, obtained by fixing all imaginary periods, with $f^{-1}$ in place of $f$.

In Gauss--Manin trivializations centered at the base points of these plaques, all positive-order spatial derivatives of the cocycle matrices and of the fiber transition matrices vanish. Their zeroth-order norms, inverse norms, and the comparison constants with the Hodge metric have tempered bounds. Thus the Hodge cocycle, its fixed tensor constructions, duals, and finite direct sums are smooth along stable and unstable plaques in the sense of \cite[\S D.3.2]{BrownEskinFilipRodriguezHertz2025}. Their Gauss--Manin holonomies are locally Lipschitz in smooth flat charts on compact sets of arbitrarily large measure.
\end{lemma}

\begin{proof}
\emph{Norms and affine dynamics.}
Equip the product of relative cohomology spaces with the maximum of the Avila--Gou\"ezel--Yoccoz norms, denoted $\|\cdot\|_Y$. Use its restriction for tangent vectors to our support and its path distance along plaques. The estimates we need are \cite[Proposition~3.4, Lemma~3.6 and Theorem~A.2]{EM18}: small affine stable and unstable period charts exist, their norms and distances are locally comparable, and
\begin{equation}\label{eq:agy-tangent-bound}
 e^{-2|t|}\|v\|_{Y,q}
 \leq\|Dg_t(q)v\|_{Y,g_tq}
 \leq e^{2|t|}\|v\|_{Y,q}.
\end{equation}
The extra factor $e^{|t|}$ relative to Gauss--Manin transport comes from the action on the real and imaginary period coordinates. These inequalities hold factorwise and hence for the maximum product norm, independently of whether the measure is a product. For each fixed tensor construction $\rmE$, the Hodge growth bound gives constants $c,c_0$ with
\begin{equation}\label{eq:hodge-time-bound}
 \|\Gamma(g_t,q)\|_H,\ \|\Gamma(g_t,q)^{-1}\|_H
 \leq e^{c|t|+c_0}.
\end{equation}
The norm on an inverse map is taken between its corresponding endpoint fibers.

The local model of Lemma~\ref{lemma:prime-local-model} makes the candidate stable plaques affine: $q+v$ remains on the support when $v\in T^-_q$ is small. Changes of period coordinates are affine, and
\[
 f(q+v)=f(q)+Df(q)v.
\]
Thus the plaque dynamics is exactly its linear tangent cocycle, not merely its first-order approximation. The sign calculation in Lemma~\ref{lemma:tangent-sign-decomposition} applies with the AGY norm as well; see the norm comparison and spectral estimates in \cite[\S4.1 and Appendix~A.1]{EM18}.

\emph{Compact returns give tempered radii and metric comparisons.}
Choose a compact set $K$ of positive $\eta$-measure contained in a relatively compact period chart of a smooth branch. There are $\delta>0$ and $C<\infty$ such that, for every $z\in K$, its candidate stable plaque contains the affine $\delta$-ball about $z$ and, on this ball, Gauss--Manin transport and its inverse have norm at most $C$ in the Hodge metric. The analogous tangent-norm comparisons are uniformly bounded there as well. Shrinking $\delta$ ensures that these balls stay in the chosen chart.

The time-one map is ergodic by Lemma~\ref{lemma:mautner-blocks}. Let $N$ be the backward return time from Lemma~\ref{lemma:compact-return-tempering}, and write $z=f^{-N(q)}q$. Transport the $\delta$-plaque at $z$ by $f^{N(q)}$. Inequality~\eqref{eq:agy-tangent-bound} shows that its image contains the affine ball at $q$ of radius
\[
 \rho(q)=\delta e^{-2N(q)}
\]
in the norm $\|\cdot\|_{Y,q}$. For a point $q'$ in this ball, write $z'=f^{-N(q)}q'$. Flatness and equivariance give
\[
 P_{q,q'}
 =\Gamma(g_{N(q)},z')\,P_{z,z'}\,
   \Gamma(g_{N(q)},z)^{-1}.
\]
Consequently
\begin{equation}\label{eq:tempered-holonomy-bound}
 \|P_{q,q'}\|_H,\ \|P_{q,q'}^{-1}\|_H
 \leq C e^{2cN(q)+2c_0}.
\end{equation}
Applying the same argument to the tangent cocycle gives tangent-norm comparison bounds $C'e^{4N(q)}$ on these plaques; these also bound path distance by coordinate displacement. Lemma~\ref{lemma:compact-return-tempering} makes the radii and all comparison bounds tempered. Thus the required noncompact control follows from uniform flow bounds and compact returns.

\emph{Contracting plaque charts.}
All exponents of $Df|_{T^-}$ are negative. Its logarithmic integrability follows from \eqref{eq:agy-tangent-bound}. The adapted-norm construction for an integrable linear cocycle therefore provides $\chi>0$, a measurable norm $\|\cdot\|_{s,q}$, and a tempered $L(q)\geq1$ such that
\[
 \|v\|_{Y,q}\leq\|v\|_{s,q}\leq L(q)\|v\|_{Y,q},
 \qquad
 \|Df(q)v\|_{s,fq}\leq e^{-\chi}\|v\|_{s,q}.
\]
For this standard construction and temperedness see \cite[\S A.1]{BrownEskinFilipRodriguezHertz2025}. Choose $0<\epsilon<\chi$ and take the positive slowly varying minorant
\[
 r(q)=\inf_{j\in\bbZ} e^{\epsilon|j|}\min\{1,\rho(f^jq)\}.
\]
Temperedness makes this infimum positive; it satisfies
$r(q)\leq\rho(q)$ and $e^{-\epsilon}\leq r(fq)/r(q)\leq e^{\epsilon}$.
Hence the affine chart
\[
 \Phi_q:\{v\in T^-_q:\|v\|_{s,q}<r(q)\}\longrightarrow M,
 \qquad v\longmapsto q+v,
\]
satisfies $f\Phi_q(v)=\Phi_{fq}(Df(q)v)$ on its whole domain. Its iterates remain in the corresponding charts and contract at rate $e^{-\chi n}$ in the adapted norms. The tempered comparison bounds above imply exponential contraction also in plaque path distance, after any fixed small loss in the exponent. These affine discs have the full negative Oseledets tangent dimension and give the local strong stable manifolds. The construction is unchanged after reversing time. It also supplies the contraction assertion used in Lemma~\ref{lemma:simultaneous-flow}.

\emph{Cocycle and transition derivatives.}
Trivialize $\rmE$ on this plaque by $v\mapsto P_{q,\Phi_q(v)}$. In the centered trivializations at $q$ and $fq$, equivariance makes the cocycle matrix identically $\Gamma(f,q)$, independent of $v$. On a connected overlap of two plaque charts the fiber change of coordinates is also constant, namely the appropriate Gauss--Manin transport between their centers. Its norm and inverse are bounded by \eqref{eq:tempered-holonomy-bound} when the second center lies in the first local plaque. Every positive-order spatial derivative of these matrices is zero, in any smooth parametrization of the plaque. The affine base transitions, measured in the AGY norms at their centers, have only their first derivative, controlled by the tangent comparison bounds. The adapted norms are used to select the contracting plaque domains; the fiber estimates use the fixed Hodge metrics at the centers.

One may take countably many local flat foliation charts, shrinking their plaque domains as above and using their integer translates under $f$. The centered trivializations just constructed give the tempered overlap and cocycle estimates in \cite[\S D.3.2]{BrownEskinFilipRodriguezHertz2025}; in particular their $C^k$ bounds are tempered for every fixed $k$. Tensor products, duals and finite direct sums preserve spatial constancy and replace the zeroth-order bounds only by fixed products and maxima of such bounds.

Finally, a finite collection of relatively compact smooth flat charts covers a compact set of measure arbitrarily close to one. In these charts, local Gauss--Manin holonomy has constant matrices; smooth changes of frame have bounded first derivatives on smaller compact charts. Consequently, the Gauss--Manin holonomy matrices depend locally Lipschitz-continuously on their endpoints (equivalently, they are locally H\"older continuous with exponent $1$).
 All almost-everywhere restrictions to support plaques are legitimate by \textup{(S)} and Fubini. 
\end{proof}

\printbibliography

\end{document}